\documentclass{amsart}
\usepackage{amssymb}
\usepackage{epsfig}
\usepackage{epic,eepic}
\usepackage{float}
\usepackage{color}
\usepackage{stackengine}
\usepackage{bbm}
\usepackage{hyperref}
\hypersetup{
    colorlinks=true,
    linkcolor=blue,
    filecolor=magenta,      
    urlcolor=cyan,
}
\allowdisplaybreaks

\DeclareMathAlphabet\mathbfcal{OMS}{cmsy}{b}{n}

\hsize \textwidth      
\advance \hsize by -\marginparwidth
\evensidemargin \oddsidemargin 

\newtheorem{theorem}{Theorem}[section]
\newtheorem{proposition}[theorem]{Proposition}
\newtheorem{lemma}[theorem]{Lemma}

\newtheorem{hypothesis}[theorem]{Assumption}

\theoremstyle{definition}
\newtheorem{definition}[theorem]{Definition}

\theoremstyle{remark}
\newtheorem{remark}[theorem]{Remark}

\newcommand\ubar[1]{\stackunder[1.2pt]{$#1$}{\rule{1.0ex}{.2ex}}}

\newcommand{\balpha}{{\boldsymbol{\alpha}}}
\newcommand{\bbeta}{{\boldsymbol{\beta}}}
\newcommand{\bgamma}{{\boldsymbol{\gamma}}}

\newcommand{\bnu}{{\boldsymbol{\nu}}}
\newcommand{\bphi}{{\boldsymbol{\phi}}}
\newcommand{\bpsi}{{\boldsymbol{\psi}}}

\newcommand\FF{\ensuremath{\mathbb{F}}}

\newcommand{\PP}{\ensuremath{\mathbb{P}}}

\newcommand\bL{\ensuremath{\mathbf{L}}}
\newcommand\bW{\ensuremath{\mathbf{W}}}
\newcommand\bX{\ensuremath{\mathbf{X}}}
\newcommand\bY{\ensuremath{\mathbf{Y}}}
\newcommand\bZ{\ensuremath{\mathbf{Z}}}
\newcommand\bp{\ensuremath{\mathbf{p}}}

\newcommand\bcM{\mathbfcal{M}}

\begin{document}
\title{A Probabilistic Approach to Extended Finite State Mean Field Games}
\author{Ren\'e Carmona \and Peiqi Wang}
\address{Department of Operations Research and Financial Engineering, Princeton University}

\begin{abstract}
We develop a probabilistic approach to continuous-time finite state mean field games. Based on an alternative description of continuous-time Markov chain by means of semimartingale and the weak formulation of stochastic optimal control, our approach not only allows us to tackle the mean field of states and the mean field of control in the same time, but also extend the strategy set of players from Markov strategies to closed-loop strategies. We show the existence and uniqueness of Nash equilibrium for the mean field game, as well as how the equilibrium of mean field game consists of an approximative Nash equilibrium for the game with finite number of players under different assumptions of structure and regularity on the cost functions and transition rate between states.
\end{abstract}

\maketitle

\section{Introduction}

Mean field game in which players' states belong to a finite space is first studied in \cite{gomes2013}. The dynamics of each player's states is depicted by a continuous-time Markov chain, whose transition rate matrix is a function of the player's control and probability distribution of all players' states. By assuming that each player adopts a Markovian strategy, the Nash equilibrium can be characterized by a HJB equation corresponding to the optimal control of continuous-time Markov chain on the one hand, and a Kolmogorov equation on how probability distribution of player's states evolves on the other hand. Due to the finite nature of the state space, both equations turn out to be ordinary differential equations and existence of the solution to this forward-backward system can be obtained by a fixed point argument. Continuous-time finite state mean field games were applied to model socio-economic phenomena such as paradigm shift in a scientific community and consumer choice in \cite{gomes2014}. In \cite{kolokoltsov2016}, the strategic aspect of cyber attack and defense is analyzed through a finite state mean field game model, in which the author introduces a major player - the hacker - whose action influences each minor player - the computer user - in terms of their payoff and dynamics. Theoretical aspects of finite state mean field games with major and minor players are investigated in \cite{carmona2016discrete} where existence of Nash equilibria and results on approximate Nash equilibrium for finite player game in small duration are obtained, along with the master equation characterizing the Nash equilibrium.

\vspace{3mm}
In this paper, we develop a probabilistic framework for continuous-time finite state mean field game. Our starting point is a semimartingale representation of continuous-time Markov chain introduced in \cite{elliott1995}: Let $(X_t)_{0\le t\le T}$ be a continuous-time Markov chain with $m$ states which are identified with the $m$ standard basis vectors in $\mathbb{R}^m$, then we can write:
\[
X_{t} = X_0 + \int_{(0,t]} Q^*(t)\cdot X_{t-} dt + \mathcal{M}_t.
\]
where $Q(t)$ is the transition rate matrix (also known as the Q-matrix) with $Q^*(t)$ being its transpose and $\mathcal{M}$ is a martingale. We immediately notice the analogy with diffusion processes and apply Girsanov Theorem to construct equivalent probability measures under which the process $X$ admits a different transition rate process. This opens a pathway to formulating the optimal control problem of continous-time Markov chain in a so-called \emph{weak} fashion. Indeed, in the context of optimal control of diffusion processes, the weak formulation links the control of the drift to the control of the probability measure (as opposed to the control of the path) and identifies the value function of the control problem as the solution to a backward stochastic differential equation (BSDE). By the comparison principle of the BSDE, the optimality of the control problem can be obtained by optimizing the driver of the BSDE which coincides with the Hamiltonian function. It turns out that such procedure can be transplanted to the case of optimal control of continuous-time Markov chain, thanks to the theory of BSDE driven by Markov chain developed in \cite{cohen2008} and \cite{cohen2010}.

\vspace{3mm}
Once the optimal control problem can be characterized by a BSDE, our next step is to develop a probabilistic approach to the mean field game. Probabilistic approach to mean field game is first proposed in \cite{carmona2013probabilistic}, where player's optimization problem is treated in the \emph{strong} formulation. By applying Pontryagin's Maximum Principle, the optimality of player's control problem is characterized by a forward-backward stochastic differential equation (FBSDE). Later in \cite{carmona2015probabilistic}, the authors consider the weak formulation of control problem and use the argument of change of measure which we briefly described above to obtain the BSDE characterizing the optimality. In both cases, the existence of Nash equilibria of the mean field game boils down to the well-posedness of a BSDE (or FBSDE) in which the probability distribution of the solution enters into the driver and the terminal condition of the equation. These are the so-called McKean-Vlasov type of BSDE (or FBSDE) for which the existence of the solution can be obtained by a fixed-point argument \`a la Schauder.

\vspace{3mm}
By developing the weak formulation, our contributions to finite state mean field game are three-fold. On the one hand, the flexibility of the probabilistic approach allows us to incorporate not only the mean field of state, but also the mean field of control into the dynamics and cost functionals of individual players. Mean field of control is known to be notoriously intractable via PDE method, due to the difficulties in deriving the equation obeyed by the flow of probability measure of the optimal control. Under the probabilistic framework however, the mean field of state and the mean field of control can be dealt with in similar manners, although the treatment of mean field of control is more involved in terms of the topological argument. On the other hand, using the weak formulation we are able to show Nash equilibria exist among all closed-loop strategies, including the strategies depending on the past history of player's states, whereas the PDE approach can only accommodate Markovian strategies. 

\vspace{3mm}
Lastly, the weak formulation we develop for the finite state mean field game will serve as a launching pad to tackle the finite state mean field agent-principal problem. Such model is a form of Stackelberg game in which the principal fixes a contract first and a large population of agents reaches Nash equilibrium according to the contract proposed by the principal. By fixing a contract we actually mean that the principal chooses a control which enters into each agent's dynamics and cost functions. One meaningful direction in probing mean field agent-principal problems is to understand how the principal can choose the optimal contract so that its own cost function depending on agent's distribution is minimized. To the best of our knowledge, this type of problem is first investigated in \cite{elie2016} where the agent's dynamics is a diffusion. The main idea is to formulate the optimal contract problem as a Mckean-Vlasov optimal control problem, in which the state process to be controlled is the Mckean-Vlasov BSDE characterizing the Nash equilibrium in the weak formulation of the mean field game. With the help of the weak formulation we develop in this paper, we believe that the same technique can be applied to the case of finite state mean field agent-principal problem, which could lead to potential applications in epidemics and cyber security.

\vspace{3mm}
We would also like to mention a few literatures related to our paper. In \cite{cecchin2017} the authors proposed a probabilistic framework for finite state mean field game where the player's dynamics of states is represented by stochastic differential equations driven by Poisson random measures. By using Ky Fan's fixed point theorem, the authors obtained existence and uniqueness of the Nash equilibrium in relaxed open-loop as well as relaxed feedback controls. Then under stronger assumption that guarantees uniqueness of optimal non-relaxed feedback control, the authors deduced existence of Nash equilibria in non-relaxed feedback form. In \cite{doncel2017}, continuous-time mean field games with finite state space and finite action space were studied. The authors proved existence of Nash equilibrium among relaxed feedback controls. In \cite{benazzoli2017} the authors investigated mean field games where each player's state follows a jump-diffusion process and the player controls the sizes of the jumps. The formulation is based on weak formulation of stochastic controls and martingale problems. Existence of Nash equilibrium among relaxed controls and Markovian controls is established.

\vspace{3mm}
The rest of the paper is organized as follows. In Section 2, we introduce the weak formulation of finite state mean field game, which is based on a semimartingale representation of continuous-time Markov chain and an argument of change of measure. We state the assumptions used throughout the paper and give the precise definition of the Nash equilibrium in the weak formulation. In Section 3, we analyze player's optimal control problem when facing a fixed mean field of state and control, by characterizing the value function and the optimal control using a BSDE driven by Markov chain. Section 4 is devoted to the existence and the uniqueness of the Nash equilibrium. Finally in Section 5, we formulate the game with finite number of players and show the Nash equilibrium of the mean field game is an approximate Nash equilibrium of the game with finite number of players.

\section{The Weak Formulation for Finite State Mean Field Games}

\subsection{Notations}
If  $M$ is a square real matrix, we denote by by $M^*$  its transpose and $M^+$ its Moore-Penrose pseudo inverse. For a column vector $X$, we denote by $diag(X)$ the square diagonal matrix whose diagonal elements are given by the entries of $X$. If  $\gamma$ is a random variable on a probability space $(\Omega,\mathcal{F},\mathbb{P})$, we denote its law or its distribution, namely  the push-forward of $\mathbb{P}$ by $\gamma$by $\mathbb{P}_{\#\gamma} := \mathbb{P}\circ \gamma^{-1}$.

\vspace{3mm}
For two square integrable martingales $L$, $M$, we denote by $[L, M]$ the quadratic covariation process of $L$ and $M$. For two semimartingales $L$ and $M$, we denote by $\langle L, M\rangle$ the predictable quadratic covariation process of $L$ and $M$. For a semimartingale $L$ such that $L_0 = 0$, we denote by $\mathcal{E}(L)$ the process of Dol\'eans-Dade exponential of $L$. See Chapter II.6 in \cite{protter2005} for the definitions of these standard concepts.

\subsection{Controlled probability measure}
For the control of continuous-time finite state Markov chains we adopt the formalism first introduced in \cite{elliott1995}, and later developed in \cite{cohen2008} and \cite{cohen2010}. If $\bX = (X_t)_{0\le t\le T}$ is a continuous-time Markov chain with $m$ states, we identify these states with the basis vectors $e_i$ in $\mathbb{R}^m$ and we denote by $E$ the resulting state space $E=\{e_1,\dots,e_m\}$. We assume that the sample paths $t\rightarrow X_t$ are \emph{c\`adl\`ag}, i.e. right continuous with left limits, and continuous at $T$. In other words, we force $X_{T-}=X_T$. 

\vspace{3mm}
We first construct a canonical probability space for $\bX$. Let $\Omega$ be the space of c\`adl\`ag functions from $[0,T]$ to $E$ which are continuous at $T$, and let $\bX$ be the canonical process on $\Omega$, that is $X_t(\omega) := \omega_t$. We denote by $\mathbb{F} := (\mathcal{F}_t)_{t \in [0,T]}$ with $\mathcal{F}_t := \sigma\{X_s, s\le t\}$ the natural filtration generated by $\bX$,
and we set $\mathcal{F} := \mathcal{F}_T$. Throughout the rest of the paper, we fix a probability measure $\mathbf{p}^{\circ}$ on the set $E$. It will be used as the initial distribution of the process $\bX$. On the filtered space $(\Omega, \mathbb{F}, \mathcal{F})$, we consider the probability measure $\mathbb{P}$ under which $\bX$ is a continuous-time Markov chain with initial distribution $\mathbf{p}^{\circ}$ and transition rates between any two different states equal to $1$. This means that for $i,j \in \{1,\dots,m\}$, $i\neq j$ and $\Delta t >0$, we have $\mathbb{P}[X_{t+\Delta t} = e_j | \mathcal{F}_t ] = \mathbb{P}[X_{t+\Delta t} = e_j| X_t]$ and $\mathbb{P}[X_{t+\Delta t} = e_j | X_t = e_i] = \Delta t + o(\Delta t)$. By Appendix B in \cite{elliott1995}, the process $X$ has the  representation:
\begin{equation}
\label{eq:canonical_representation}
X_t = X_0 + \int_{(0,t]} Q^0\cdot X_{t-} dt + \mathcal{M}_t,
\end{equation}
where $Q^0$ is the square matrix with diagonal elements all equal to $-(m-1)$ and off-diagonal elements all equal to $1$, and  $\\bcM=(\mathcal{M}_t)_{t\ge 0}$ is a $\mathbb{R}^m$-valued $\mathbb{P}$-martingale. The multiplication $\cdot$ is understood as matrix multiplication. 
Indeed, $Q^0$ is the transition rate matrix of $\bX$ under the probability measure $\mathbb{P}$.
\begin{remark}
The representation originally proposed in \cite{elliott1995} is:
\[
X_t = X_0 + \int_{(0,t]} Q^0\cdot X_{t} dt + \mathcal{M}_t.
\]
However since $X_t$ is only discontinuous on a countable set, we can replace $X_t$ by $X_{t-}$ in the integral. The reason for this slight change of representation is to make the integrand a predictable process, which will be suitable for the change of measure argument in what follows.
\end{remark}
We shall refer to the probability measure $\mathbb{P}$ as the \emph{reference measure} on the sample space. The first step of the weak formulation of mean field game consists in depicting how each player's control as well as the mean field determine the probability measure of the sample path. We denote by $\mathcal{S}$ the $m$-dimensional simplex:
\[
\mathcal{S} := \{ p \in \mathbb{R}^m;\; \sum_{i=1}^m p_i = 1, p_i \ge 0\},
\]
which we identify with the space of probability distributions on $E$. Let $A$ be a compact subset of $\mathbb{R}^l$ from which the players can choose their controls. Denote by $\mathcal{P}(A)$ the space of probability measures on $A$. We introduce a function $q$:
\[
[0,T] \times \{1,\dots,m\}^2 \times A \times \mathcal{S}\times\mathcal{P}(A) \rightarrow q(t, i, j, \alpha, p,\nu),
\] 
and we denote by $Q(t,\alpha,p,\nu)$ the matrix $[q(t, i, j, \alpha, p,\nu)]_{1\le i,j\le m}$. Throughout the rest of the paper, we make the following assumption on $q$:

\begin{hypothesis}
\label{hypo:boundedness}
(i) For all $(t,\alpha, p, v) \in [0,T] \times A \times \mathcal{S}\times\mathcal{P}(A)$, the matrix $Q(t, \alpha, p,\nu)$ is a Q-matrix.

\noindent(ii) There exist constants $C_1,C_2>0$ such that for all $(t,i,j,\alpha,p,\nu) \in [0,T]\times E^2\times A\times\mathcal{S}\times \mathcal{P}(A)$ such that $i\neq j$, we have $0 < C_1< q(t,i,j,\alpha,p,\nu) < C_2$.

\noindent(iii) There exists a constant $C>0$ such that for all $(t,i,j)\in[0,T]\times E^2$, $\alpha,\alpha' \in A$, $p,p' \in \mathcal{S}$ and $\nu,\nu'\in\mathcal{P}(A)$, we have:
\[
|q(t,i,j,\alpha,p,\nu) - q(t,i,j,\alpha', p', \nu')| \le C( \|\alpha - \alpha'\| + \|p - p'\| + \mathcal{W}_1(\nu,\nu')).
\]
where $\mathcal{W}_1$ denotes the $1$-Wasserstein distance between probability measures on $A$.
\end{hypothesis}

\vskip 6pt\noindent
Recall that a matrix $Q=[Q_{ij}]$ is called a Q-matrix if $Q_{ij}\ge 0$ for $i\ne j$ and 
\[
\sum_{j \neq i} Q_{ij} = -Q_{ii},\qquad \text{for all } i.
\]
\begin{remark}
Assumption \ref{hypo:boundedness} is analog to the non-degeneracy condition in the diffusion-based mean field game models. It guarantees that the probability measure $\mathbb{Q}^{(\alpha,p,\nu)}$ defined in \eqref{eq:q_measure} below, is equivalent to the reference measure $\mathbb{P}$. In some applications of continuous-time Markov chain models, it happens that jumps from some states to others are forbidden, in which case the transition rate function $q$ would satisfy $q(t,i,j,\alpha,p, \nu)\equiv 0$ for some couples $(i,j)$. For example, this is the case in the botnet defense model proposed by \cite{kolokoltsov2016}, as well as in the extended version of the model which includes an attacker studied in \cite{carmona2016discrete}. When that happens, we need to use a different reference probability measure $\mathbb{P}$: we set the transition rate to $1$ for all the jumps, except for those that are forbidden, for which we set the transition rate to $0$. Fortunately, this is the only modification we need to make in order to accommodate this kind of special case. The arguments presented in the following can be trivially extended to be compatible with this modified reference probability.
\end{remark}

We state without proof a useful property  of the martingale $\bcM$. The proof of this result can be found in \cite{cohen2008}:
\begin{lemma}
The predictable quadratic variation of the martingale $\bcM$ under $\mathbb{P}$ is given by the formula:
\begin{equation}\label{eq:quad_var}
\langle \bcM, \bcM \rangle_t = \int_0^t \psi_t dt,
\end{equation}
where $\psi_t$ is given by:
\begin{equation}\label{eq:matrix_psi}
\psi_t := diag(Q^0 \cdot X_{t-}) - Q^0 \cdot diag( X_{t-}) - diag( X_{t-}) \cdot Q^0.
\end{equation}
\end{lemma}

\vskip 6pt
\noindent If we define for each $i$ the matrix $\psi^i$ by:
$$
\psi^i := diag(Q^0 \cdot e_i) - Q^0 \cdot diag(e_i) - diag(e_i) \cdot Q^0,
$$ 
then clearly we have $\psi_t = \sum_{i = 1}^m \mathbbm{1}(X_{t-} = e_i)\psi^i$. Since each $\psi^i$ is a semi-definite positive matrix, so is $\psi_t$. We define the corresponding (stochastic) seminorm $\|\cdot\|_{X_{t-}}$ on $\mathbb{R}^m$ by:
\begin{equation}\label{eq:seminorm}
\|Z\|^2_{X_{t-}} := Z^*\cdot \psi_t\cdot Z.
\end{equation}
The semi-norm $\|\cdot\|_{X_{t-}}$ can be rewritten in a more explicit way. For $i\in\{1,\dots,m\}$, let us define the seminorm $\|\cdot\|_{e_i}$ on $\mathbb{R}^m$ by $\|Z\|^2_{e_i} := Z^* \cdot \psi^i \cdot Z = \sum_{j\neq i} |Z_j - Z_i|^2$. Then it is easy to see that $\|Z\|_{X_{t-}} = \sum_{i=1}^m\mathbbm{1}(X_{t-} = i)\|Z\|_{e_i}$.

\vspace{3mm}
Since $\psi_t$ is symmetric, we have $(\psi_t^+)^* = \psi_t^+$. Recall that $\psi_t^+$ is the Moore-Penrose generalized inverse of the matrix $\psi_t$. On the other hand, it is straightforward to verify that for all $t\in [0,T]$ and $w \in \Omega$, the range of the matrix $\psi_t$ (i.e. the linear space spanned by the columns of $\psi_t$) is the space $\{q\in\mathbb{R}^m;\; \sum_{i=1}^m q_i = 0\}$. Therefore for all $q\in\mathbb{R}^m$ with $\sum_{i=1}^m q_i = 0$, we have $\psi_t \cdot \psi_t^+ \cdot q = q$. This holds in particular for any row vector from any $Q$-matrix, or any vector of the form $(e_j - e_i)$.

\vspace{3mm}
In order for the paper to be as self-contained as possible, we also recall the following version of Girsanov Theorem on change of probability measure. See Theorem III.41 in \cite{protter2005} or Lemma 4.3 in \cite{sokol2015}.

\begin{theorem}
\label{thm:girsanov}
Let $T>0$ and $\bL=(L_t)_{t\ge 0}$ be a martingale defined on $[0,T]$ with $\Delta L_t\ge -1$. Assume that the Dol\'eans-Dade exponential $\mathcal{E}(\bL)$ of $\bL$ is a uniformly integrable martingale and let $\mathbb{Q}$ be the probability measure having Radon-Nikodym derivative $\mathcal{E}(\bL)_T$ with respect to $\mathbb{P}$. If the quadratic covariation process $[\bcM, \bL]$ is integrable under $\mathbb{P}$, then $\bcM - \langle \bcM,\bL \rangle$ is a martingale under $\mathbb{Q}$, where the predictable quadratic covariation $\langle \bcM,\bL \rangle$ is computed under the measure $\mathbb{P}$.
\end{theorem}

We now describe how the control of a player and the mean field affect the probability law of $\bX$. Let us define the player's strategy set $\mathbb{A}$ to be the collection of $\mathbb{F}$-predictable processes $\balpha=(\alpha_t)_{t\in[0, T]}$ such that $\alpha_t \in A$ for $t \in [0, T]$. Given a flow of probability measures $\bp=(p_t)_{t\in[0, T]}$ on $E$, and a flow of probability measures $\bnu=(\nu_t)_{t\in[0, T]}$ on $A$, we define the scalar martingale $\bL^{(\balpha,\bp,\bnu)}$ under $\mathbb{P}$ by:
\begin{equation}\label{eq:martingale_L}
L^{(\balpha,\bp,\bnu)}_t := \int_0^t X_{s^-}^* \cdot(Q(s, \alpha_s, p_s, \nu_s) - Q^0)\cdot\psi_s^+ \cdot d\mathcal{M}_s.
\end{equation}
Clearly, the jumps of this are given by:
\begin{equation}\label{eq:jump_of_l}
\Delta L^{(\balpha,\bp,\bnu)}_t = X_{t^-}^* \cdot (Q(t, \alpha_t, p_t, \nu_t) - Q^0)\cdot \psi_t^+ \cdot \Delta X_t.
\end{equation}
One can easily check that $\psi_t^+\cdot(e_j - X_{t-}) =\frac{m-1}{m} e_j - \sum_{i\neq j}\frac{1}{m} e_i $ when $X_{t-}= e_i \neq e_j$. Therefore when $X_{t-} = e_i \neq e_j = X_t$, we have:
\begin{align*}
\Delta L^{(\balpha,\bp,\bnu)}_t  =&\;\; X_{t-}^*\cdot(Q(t, \alpha_t, p_t, \nu_t) - Q^0)\cdot\psi_t^+\cdot(e_j - X_{t-}) \\
=&\;\; e_i^*\cdot(Q(t, \alpha_t, p_t, \nu_t) - Q^0)\cdot\left[\frac{m-1}{m} e_j - \sum_{k\neq j}\frac{1}{m} e_k\right]\\
=&\;\; \frac{m-1}{m} (q(t,i,j,\alpha_t,p_t,\nu_t) - q^0_{i,j}) - \frac{1}{m} \sum_{k\neq j}(q(t,i,k,\alpha_t,p_t,\nu_t) - q^0_{i,k})\\
=&\;\; q(t,i,j,\alpha_t,p_t,\nu_t) - q^0_{i,j}\\ 
=&\;\; q(t,i,j,\alpha_t,p_t,\nu_t) - 1,
\end{align*}
where the last equality is due to the fact that $\sum_{k=1}^m(q(t,i,k,\alpha_t,p_t,\nu_t) - q^0_{i,k}) = 0$. Therefore we have $\Delta L^{(\balpha,\bp,\bnu)}_t \ge -1$. By Theorem III.45 in \cite{protter2005} and the remark that follows, in order to show that $\mathcal{E}(\bL^{(\balpha,\bp,\bnu)})$ is uniformly integrable, it suffices to show $\mathbb{E}[\exp(\langle \bL^{(\balpha,\bp,\bnu)}, \bL^{(\balpha,\bp,\bnu)}\rangle_T)] < \infty$. This is straightforward since we have:
\begin{align*}
&\langle \bL^{(\balpha,\bp,\bnu)}, \bL^{(\balpha,\bp,\bnu)}\rangle_T \\
&=\hskip 24pt
 \int_0^T X_{s^-}^*\cdot (Q(s, \alpha_s, p_s, \nu_s) - Q^0)\cdot\psi_s^+ \cdot \frac{d\langle \bcM, \bcM\rangle_s}{ds} \cdot (X_{s^-}^* \cdot (Q(s, \alpha_s, p_s, \nu_s) - Q^0)\cdot\psi_s^+)^*ds\\
&=\hskip 24pt
 \int_0^T X_{s^-}^* \cdot (Q(s, \alpha_s, p_s, \nu_s) - Q^0)\cdot\psi_s^+\cdot(Q^*(s, \alpha_s, p_s, \nu_s) - Q^0)\cdot X_{s^-}ds.
\end{align*}
and the integrand is bounded by some constant by Assumption \ref{hypo:boundedness}.

\vskip 6pt
We now apply Girsanov's Theorem. It is straightforward to obtain that:
\begin{align*}
\langle \bcM, \bL^{(\balpha,\bp,\bnu)} \rangle_t =& \int_{0}^t d\langle \bcM, \bcM \rangle_s \cdot (\psi_s^+)^*\cdot(Q^*(s, \alpha_s, p_s, \nu_s) - Q^0) \cdot X_{s-} \\
=& \int_{0}^t \psi_s\cdot\psi_s^+ \cdot (Q^*(s, \alpha_s, p_s, \nu_s) - Q^0)\cdot X_{s-} ds \\
=&  \int_{0}^t ( Q^*(s, \alpha_s, p_s, \nu_s) - Q^0)\cdot X_{s-} ds.
\end{align*}
In the last equality, we use the fact that $( Q^*(s, \alpha_s, p_s, \nu_s) - Q^0) X_{s}$ is the difference between two row vectors coming from $Q$-matrices, 
therefore is invariant by $\psi_s\cdot\psi_s^+$. Let us define the probability measure $\mathbb{Q}^{(\balpha,\bp,\bnu)}$ by:
\begin{equation}
\label{eq:q_measure}
\frac{d\mathbb{Q}^{(\balpha,\bp,\bnu)}}{d\mathbb{P}} := \mathcal{E}(\bL^{(\balpha,\bp,\bnu)})_T.
\end{equation}
By Theorem \ref{thm:girsanov}, we know that the process $\bcM^{(\balpha,\bp,\bnu)}$, defined as:
\begin{equation}\label{eq:martingale_M}
\mathcal{M}^{(\balpha,\bp,\bnu)}_t := \mathcal{M}_t - \int_{0}^t ( Q^*(s, \alpha_s, p_s, \nu_s) - Q^0)\cdot  X_{s-} ds,
\end{equation}
is a $\mathbb{Q}^{(\balpha,\bp,\bnu)}$-martingale. Therefore the canonical decomposition \eqref{eq:canonical_representation} of $X$ under $\PP$ can be rewritten as:
\begin{equation}\label{eq:X_decomp}
X_t = X_0 + \int_0^t Q^*(s, \alpha_s, p_s, \nu_s)\cdot X_{s-} dt + \mathcal{M}^{(\balpha,\bp,\bnu)}_t.
\end{equation}
This means that under the measure $\mathbb{Q}^{(\balpha,\bp,\bnu)}$, the stochastic intensity rate of $\bX$ is given by $Q(t, \alpha_t, p_t, \nu_t)$. In addition, since $\mathbb{Q}^{(\balpha,\bp,\bnu)}$ and $\mathbb{P}$ coincides on $\mathcal{F}_0$, the law of $X_0$ under $\mathbb{Q}^{(\balpha,\bp,\bnu)}$ is the same as under the reference measure $\mathbb{P}$, which is $\mathbf{p}^{\circ}$.  In particular, when $\balpha$ is a Markov control, i.e. of the form $\alpha_t = \phi(t, X_{t-})$ for some measurable function $\phi$, $\bX$ becomes a continuous-time Markov chain with intensity rate $q(t,i,j,\phi(t,i),p_t,\nu_t)$ under the measure $\mathbb{Q}^{(\balpha,\bp,\bnu)}$.

\begin{remark}
In the optimal control literature, admissible controls are often classified into the categories of open-loop controls and closed-loop controls. Open-loop controls are often referred to controls adapted to the underlying filtration, which is often generated by the noise process. Closed-loop controls, on the other hand, are controls that are adapted to the filtration generated by the history of the state process. In our set up, however, we see that the underlying filtration is indeed the one generated by the past path of the state process. Therefore this difference vanishes.
\end{remark}

\subsection{Weak formulation of mean field games}
Let $f:[0,T] \times E \times A \times \mathcal{S} \times \mathcal{P}(A) \rightarrow \mathbb{R}$ and $g:E \times \mathcal{S}\rightarrow \mathbb{R}$ be respectively the running and terminal cost functions. In the rest of the paper, we make the following assumptions on the regularity of the cost functions.
\begin{hypothesis}\label{hypo:lipschitz_cost}
There exists a constant $C>0$ such that for all $(t,i,j)\in[0,T]\times E^2$, $\alpha,\alpha' \in A$, $p,p' \in \mathcal{S}$ and $\nu,\nu'\in\mathcal{P}(A)$, we have:
\begin{align}
|f(t,e_i,\alpha,p,\nu) - f(t,e_i,\alpha', p', \nu')| \le& C( \|\alpha - \alpha'\| + \|p - p'\| + \mathcal{W}_1(\nu,\nu')),\\
|g(e_i,p) - g(e_i,p') | \le& C \|p - p'\|.
\end{align} 
\end{hypothesis}

\vskip 6pt\noindent
When a player chooses a strategy $\balpha\in\mathbb{A}$ and the mean field is $(\bp,\bnu)$, its cost is:
\begin{equation}\label{eq:total_cost}
J(\balpha,\bp,\bnu):=\mathbb{E}^{\mathbb{Q}^{(\balpha,\bp,\bnu)}}\Bigl[\int_0^T f(t, X_t, \alpha_t, p_t, \nu_t) dt + g(X_T, p_T)\Bigr].
\end{equation}
Each player aims at minimizing its cost, that is, it solves the optimization problem:
\begin{equation}\label{eq:optimization_pb}
V(\bp,\bnu) := \inf_{\balpha\in\mathbb{A}}\mathbb{E}^{\mathbb{Q}^{(\balpha,\bp,\bnu)}}\Bigl[\int_0^T f(t, X_t, \alpha_t, p_t, \nu_t) dt + g(X_T, p_T)\Bigr].
\end{equation}

\noindent
The key idea of the theory of mean field games lies in the limit scenario of having infinitely many players in the game, where a single player's strategy $\balpha$ does not alter the mean field $(\bp,\bnu)$. Therefore when each player solves its own optimization problem, it considers $(\bp,\bnu)$ as given. A Nash equilibrium is then achieved when the law of $X_t$ under the player controlled probability law, along with the distribution of its control under the same probability law, coincide with $(\bp,\bnu)$. This justifies the following definition of a Nash equilibrium for the weak formulation of finite state mean field games.

\begin{definition}
\label{def:equilibrium}
Let $\bp^*:[0,T]\rightarrow \mathcal{S}$, and $\bnu^*:[0,T]\rightarrow \mathcal{P}(A)$ be two measurable functions and $\balpha^*\in\mathbb{A}$. We say that the tuple $(\balpha^*,\bp^*,\bnu^*)$ is a Nash equilibrium for the weak formulation of the mean field game if:

\noindent (i) $\balpha^*$ minimizes the cost when the mean field is given by $(\bp^*,\bnu^*)$:
\begin{equation}\label{eq:def_nasheq_optimality}
\balpha^* \in \arg\inf_{\alpha \in \mathbb{A}}\mathbb{E}^{\mathbb{Q}^{(\balpha,\bp^*,\bnu^*)}}\left[\int_0^T f(t, X_t, \alpha_t, p^*_t, \nu^*_t) dt + g(X_T, p^*_T)\right].
\end{equation}
\noindent (ii) $(\balpha^*,\bp^*,\bnu^*)$ satisfies the consistency conditions whereby for each time $t\in[0,T]$ it holds:
\begin{equation}\label{eq:def_nasheq_cons1}
p^*_t = \{\mathbb{Q}^{(\balpha^*,\bp^*,\bnu^*)}[X_t =e_i]\}_{i = 1, \dots, m},
\end{equation}
\begin{equation}\label{eq:def_nasheq_cons2}
\nu^*_t = \mathbb{Q}^{(\balpha^*,\bp^*,\bnu^*)}_{\#\alpha^*_t}.
\end{equation}
\end{definition}

\section{Individual Player's Optimization Problem}

Before introducing and solving the individual player optimization problem, we provide the necessary background on stochastic equations based on continuous time Markov chains.

\subsection{BSDE driven by continuous-time Markov chain}
We first recall some of the results on BSDEs driven by continuous-time Markov chains obtained in \cite{cohen2008} and \cite{cohen2010}. Recall that $\bcM$ is the $\mathbb{P}$-martingale in the canonical decomposition of the Markov chain $\bX$ in (\ref{eq:canonical_representation}). We consider the following BSDE with unknown $(\bY,\bZ)$, where $\bY$ is an adapted and c\`adl\`ag process in $\mathbb{R}$, and $\bZ$ is an adapted and left-continuous process in $\mathbb{R}^m$:
\begin{equation}
\label{eq:bsed_markov}
Y_t = \xi + \int_t^T F(w,s,Y_s,Z_s)ds - \int_t^T Z_s^* \cdot d\mathcal{M}_s.
\end{equation}
Here $\xi$ is a $\mathcal{F}_T$-measurable $\mathbb{P}$-square integrable random variable and $F$ is the driver function, assumed to be such that the process $t \rightarrow F(w,t,y,z)$ is predictable for all $y,z$. 

\vspace{3mm}
Recalling the definition (\ref{eq:seminorm}) of the stochastic semi-norm $\|\cdot\|_{X_{t-}}$, we have the following existence and uniqueness result. See Theorem 1.1 in \cite{cohen2010}.

\begin{lemma}
\label{lem:bsde_ex_un}
Assume that there exists $C > 0$ such that $dt\otimes d\mathbb{P}$-a.s., for all $y, y' \in \mathbb{R}$ and $z, z'\in\mathbb{R}^m$ we have:
\[
|F(w,t,y,z) - F(w,t,y',z')| \le C (|y - y'| + \|z - z'\|_{X_{t-}}).
\]
Then the BSDE (\ref{eq:bsed_markov}) admits a solution $(\bY,\bZ)$ satisfying
\[
\mathbb{E}\left[\int_0^T |Y_t|^2 dt\right] < +\infty,\quad\quad\mathbb{E}\left[\int_0^T \|Z_t\|_{X_{t-}}^2 dt\right] < +\infty.
\]
In addition, the solution is unique in the sense that if $(\bY^1, \bZ^1)$ and $(\bY^2, \bZ^2)$ are two solutions, then $\bY^1$ and $\bY^2$ are indistinguishable and we have $\mathbb{E}[\int_0^T \|Z^1_t- Z^2_t\|^2_{X_{t-}}dt] = 0$.
\end{lemma}

We also have the following stability property, which can be proved by mimicking the argument used in the proof of Theorem 2.1 in \cite{hu1997}.

\begin{lemma}
\label{lem:bsde_apriori}
For $n\ge 0$, let $(\bY^n, \bZ^n)$ be the solution to the BSDE \eqref{eq:bsed_markov} with driver $F^n$ and terminal condition $\xi^n$. 
Assume that for each $n$, $F^n$ satisfies the Lipschitz continuity assumption in Lemma \ref{lem:bsde_ex_un} with the same constant. In addition, assume that the following conditions hold:

\noindent(i) $\lim_{n\to\infty}\mathbb{E}[|\xi^n - \xi^0|^2] = 0$.

\noindent(ii) For each $t\le T$, $\lim_{n\to\infty}\mathbb{E}[(\int_t^T |F^n(w,s,Y^0_s,Z^0_s) - F^0(w,s,Y^0_s,Z^0_s)| ds)^2  ] = 0$.

\noindent(iii) There exists $C>0$ such that $\mathbb{E}[(\int_t^T (F^n(w,s,Y^0_s,Z^0_s) - F^0(w,s,Y^0_s,Z^0_s)) ds)^2  ] \le C$ for all $t\le T$ and $n\ge 0$.

\noindent Then we have:
\[
\lim_{n\rightarrow+\infty}\mathbb{E}\left[\int_t^T \|Z_s^n - Z_s^0\|_{X_{s-}}^2 ds\right] + \mathbb{E}[|Y^n_t - Y^0_t|^2] = 0.
\]
\end{lemma}

Finally we state a crucial comparison result for linear BSDEs. See Theorem 3.16 in \cite{cohen2010}.

\begin{lemma}
\label{lem:linear_bsde_comp}
Let $\bgamma$ be a bounded predictable process in $\mathbb{R}^m$, $\bbeta$ a bounded predictable process in $\mathbb{R}$, $\bphi$ a non-negative predictable process in $\mathbb{R}$ such that $\mathbb{E}[\int_0^T \|\phi_t\|^2 dt] < +\infty$, and $\xi$ a non-negative square-integrable $\mathcal{F}_T$-measurable random variable in $\mathbb{R}$, and let us assume that $(\bY,\bZ)$ solves the linear BSDE:
\begin{equation}\label{eq:linear_bsde}
Y_t = \xi + \int_t^T (\phi_u +\beta_u Y_u + \gamma_u^* \cdot Z_u) du -\int_t^T Z_u^*\cdot d\mathcal{M}_u.
\end{equation}
If for all $t\in(0,T]$ and $j$ such that $e_j^*\cdot Q^0 \cdot X_{t-} > 0$, we have $1 + \gamma_t^*\cdot\psi_t^+\cdot(e_j - X_{t-}) \ge 0$ where $\psi_t^+$ is the Moore-Penrose inverse of the matrix $\psi_t$ defined in equation (\ref{eq:matrix_psi}), then $\bY$ is nonnegative.
\end{lemma}

Later in the treatment of games with finitely many players, we will need to consider BSDEs driven by multiple independent continuous-time Markov chains. It turns out that all the results above regarding BSDEs driven by one single continuous-time Markov chain can be easily extended to this more general setting. For the sake of completeness, we state and prove these results in the appendix.

\subsection{Hamiltonian}
We define the Hamiltonian for the optimization problem of the individual player as the function $H$ from $[0,T] \times E \times \mathbb{R}^m\times A \times \mathcal{S} \times \mathcal{P}(A)$ into $\mathbb{R} $ by:
\begin{equation}
\label{eq:hamiltonian_def}
H(t,x,z,\alpha,p,\nu) := f(t,x,\alpha,p,\nu) + x^*\cdot(Q(t, \alpha, p, \nu) - Q^0)\cdot z.
\end{equation} 
Since the process $X$ takes value in the set $\{e_1,\dots,e_m\}$, it is more convenient to consider $m$ Hamiltonian functions $H_i$ defined for $i=1,\cdots,m$ by $H_i(t,z,\alpha,p,\nu):=H(t,e_i,z,\alpha,p,\nu)$. Clearly we have:
\begin{equation}\label{eq:reduced_hamiltonian_def}
H_i(t,z,\alpha,p,\nu) = f(t,e_i, \alpha,p,\nu) + \sum_{j\neq i} (z_j - z_i)(q(t, i, j, \alpha, p,\nu) - 1).
\end{equation}
We denote by $\hat H_i$ the corresponding minimized Hamiltonian:
\[
\hat H_i(t,z,p,\nu) := \inf_{\alpha \in A} H_i(t,z,\alpha,p,\nu),
\]
and to show the existence of Nash equilibria, we make the following assumption on the minimizer of the Hamiltonian.

\begin{hypothesis}
\label{hypo:lipschitz_optimizer}
(i) For any $t\in[0,T]$, $i\in\{1,\dots,m\}$, $z\in\mathbb{R}^m$, $p\in\mathcal{S}$ and $\nu\in\mathcal{P}(A)$, the mapping $\alpha \rightarrow H_i(t,z,\alpha,p,\nu)$ admits a unique minimizer which does not depend on the mean field of control $\nu$. We denote the minimizer by $\hat a_i(t,z,p)$.

\noindent (ii) $\hat a_i$ is measurable on $[0,T]\times\mathbb{R}^m\times\mathcal{S}$ and there exist constants $C_1>0$ and $C_2\ge0$ such that for all $i\in\{1,\dots,m\}$, $z, z'\in\mathbb{R}^m$, $p, p'\in\mathcal{S}$:
\begin{equation}\label{eq:reduced_optimizer_lipschitz}
\|\hat a_i(t,z,p) - \hat a_i(t,z',p')\| \le C_1\|z - z'\|_{e_i} + (C_1 + C_2 \|z\|_{e_i})\|p - p'\|.
\end{equation}
\end{hypothesis}

\begin{remark}
For the sake of convenience, we choose to make the assumption directly on the uniqueness and the regularity of the minimizer of the Hamiltonian. One possible way to make sure Assumption \ref{hypo:lipschitz_optimizer} holds is to impose linearity on the transition rate function $q$, and strong convexity of the running cost function $f$. For example, the following set of conditions will guarantee that Assumption \ref{hypo:lipschitz_optimizer} holds:
\end{remark}

\begin{hypothesis}
\label{hypo:transition_rate_eq_existence}
\noindent
(i) $A$ is a convex and compact subset of $\mathbb{R}^l$.

\noindent(ii) The transition rate function $q$ takes the form $q(t,i,j,\alpha,p,\nu) = q_0(t,i,j,p,\nu) + q_1(t,i,j,p)\cdot\alpha$, where $q_0:[0,T]\times E^2\times\mathcal{S}\times\mathcal{P}(A) \rightarrow \mathbb{R}$ and $q_1:[0,T]\times E^2 \times \mathcal{S} \rightarrow \mathbb{R}^l$ are two continuous mappings.

\noindent(iii) The running cost function $f$ is of the form $f(t, x, \alpha, p, \nu) = f_0(t, x, \alpha, p) + f_1(t, x, p, \nu)$, where for each $i\in\{1,\dots,m\}$, the mapping $f_0(\cdot, e_i, \cdot, \cdot)$ (resp. $f_1(\cdot, e_i, \cdot, \cdot)$) is continuous on $[0,T]\times A \times \mathcal{P}$ (resp. $[0,T]\times\mathcal{S}\times\mathcal{P}(A)$).

\noindent(iv) For all $(t,e_i,p)\in[0,T]\times E \times \mathcal{S}$, the mapping $\alpha \rightarrow f_0(t,e_i, \alpha, p)$ is once continuously differentiable and there exists a constant $C>0$ such that:
\begin{equation}
\|\nabla_{\alpha} f_0(t, e_i, \alpha, p) - \nabla_{\alpha} f_0(t, e_i, \alpha, p')\| \le C\|p - p'\|.
\end{equation}

\noindent (v) $f_0$ is $\gamma$-uniformly convex in $\alpha$, i.e., for all $(t,e_i,p)\in[0,T]\times E \times \mathcal{S}$ and $\alpha, \alpha' \in A$, we have:
\begin{equation}\label{eq:strong_convex_f0}
f_0(t,e_i,\alpha,p) - f_0(t,e_i,\alpha',p) - (\alpha - \alpha') \cdot \nabla_{\alpha} f_0(t, e_i, \alpha, p) \ge \gamma \|\alpha' - \alpha\|^2
\end{equation}
\end{hypothesis}

\vskip 6pt
We define the functions  $\hat H$ and $\hat a$ by:
\begin{align}
\hat H(t,x,z,p,\nu) :=& \sum_{i=1}^m \mathbbm{1}(x = e_i)\hat H_i(t,z,p,\nu),\label{eq:def_h_hat}\\
\hat a(t,x,z,p) :=& \sum_{i=1}^m  \mathbbm{1}(x = e_i)\hat a_i(t,z,p).\label{eq:def_a_hat}
\end{align}
From item (i) of Assumption \ref{hypo:lipschitz_optimizer} and the definition of the reduced Hamiltonian $H_i$, it is clear that $\hat a(t,x,z,p)$ is the unique minimizer of the mapping $\alpha \rightarrow H(t,x,z,\alpha,p,\nu)$, and the minimum equals $\hat H(t,x,z,p,\nu)$. In addition, from Assumptions \ref{hypo:boundedness}, \ref{hypo:lipschitz_cost}, \ref{hypo:lipschitz_optimizer}, and the definition of the stochastic semi-norm $\|\cdot\|_{X_{t-}}$, it is easy to deduce the regularity of the mappings $\hat H$ and $\hat a$.

\begin{lemma}
\label{lem:lip_h_hat}
There exists a constant $C>0$ such that for all $(\omega,t)\in\Omega\times(0,T]$, $p,p'\in\mathcal{S}$, $\nu,\nu' \in \mathcal{P}(A)$ and $z,z'\in\mathbb{R}^m$, we have:
\begin{equation}\label{eq:lipschitz_h_hat}
|\hat H(t,X_{t-},z,p,\nu) - \hat H(t,X_{t-},z',p',\nu')| 
\le C\|z - z'\|_{X_{t-}} + C(1 +  \|z\|_{X_{t-}})(\|p - p'\| + \mathcal{W}_1(\nu, \nu')),
\end{equation}
\begin{equation}
\label{eq:lipschitz_full_optimizer}
|\hat a(t,X_{t-},z,p) - \hat a(t,X_{t-},z',p')| \le C\|z - z'\|_{X_{t-}} + C(1 +  \|z\|_{X_{t-}})\|p - p'\|.
\end{equation}
\end{lemma}

\begin{proof}
Inequality (\ref{eq:lipschitz_full_optimizer}) is an easy consequence of Assumption \ref{hypo:lipschitz_optimizer}  and the definition of the stochastic seminorm $\|\cdot\|_{X_{t-}}$. We now deal with the regularity of $\hat H$. By Berge's maximum theorem, the continuity of $H_i$ and the compactness of $A$ imply the continuity of $\hat H_i$. Let $z,z' \in \mathbb{R}^m$, $p,p'\in\mathcal{S}$ and $\nu,\nu'\in\mathcal{P}(A)$. For any $\alpha \in A$, we have:
\begin{align*}
\hat H_i(t,z,p,\nu) - H_i(t,z',\alpha,p',\nu')& \le  H_i(t,z,\alpha,p,\nu) - H_i(t,z',\alpha,p',\nu') \\
=&\;\; f(t, e_i, \alpha, p, \nu) - f(t, e_i, \alpha, p', \nu') + \sum_{j\neq i} [(z_j - z_i) - (z'_j - z'_i)] q(t,i,j,\alpha, p', \nu')\\
&\hskip -25pt
+\sum_{j\neq i} (z_j - z_i) [q_0(t,i,j, p, \nu) - q_0(t,i,j, p', \nu')] + (z_j - z_i) [q_1(t,i,j, p) - q_1(t,i,j, p')] \cdot \alpha \\
\le &\;\; C\|z - z'\|_{e_i} + C(1 +  \|z\|_{e_i})(\|p - p'\| + \mathcal{W}_1(\nu, \nu')),
\end{align*}
where we used the Lipschitz property of $f$ and $q$, and the boundedness of $A$ and $q$. Since the above is true for all $\alpha \in A$, taking supremum of the left-hand side, we obtain:
\[
\hat H_i(t,z,p,\nu) - \hat H_i(t,z',p',\nu') \le C\|z - z'\|_{e_i} + C(1 +  \|z\|_{e_i})(\|p - p'\| + \mathcal{W}_1(\nu, \nu')).
\]
Exchanging the roles of $z$ and $z'$, we obtain:
\[
|\hat H_i(t,z,p,\nu) - \hat H_i(t,z',p',\nu')| \le C\|z - z'\|_{e_i} + C(1 +  \|z\|_{e_i})(\|p - p'\| + \mathcal{W}_1(\nu, \nu')),
\]
and (\ref{eq:lipschitz_h_hat}) follows immediately from the definition of the seminorm $\|\cdot\|_{X_{t-}}$.
\end{proof}

\subsection{Player's optimization problem}
In this subsection, we show that the optimization problem of the player facing a given mean field of state and control can be characterized by a BSDE driven by the continuous-time Markov chain $\bX$. Let us fix measurable flows $\bp:[0,T]\rightarrow \mathcal{S}$ and $\bnu:[0,T]\rightarrow \mathcal{P}(A)$, an admissible strategy $\balpha\in\mathbb{A}$, and let us consider the BSDE:
\begin{equation}\label{eq:bsde_total_cost}
Y_t = g(X_T, p_T) + \int_t^T H(s, X_{s-}, Z_s, \alpha_s, p_s, \nu_s) ds - \int_t^T Z_s^*\cdot d\mathcal{M}_s.
\end{equation}

\begin{lemma}
\label{lem:bsde_total_cost}
The BSDE (\ref{eq:bsde_total_cost}) admits a unique solution $(\bY,\bZ)$ and $J(\balpha,p,\nu) = \mathbb{E}^{\mathbb{P}}[Y_0]$.
\end{lemma}

\begin{proof}
From the boundedness of the transition rate function $q$ guaranteed by Assumption \ref{hypo:boundedness}, it is easy to check that the driver function $H$ of the BSDE (\ref{eq:bsde_total_cost}) is Lipschitz in $z$ with respect to the semi-norm $\|\cdot\|_{X_{t-}}$. Therefore by Lemma \ref{lem:bsde_ex_un}, it admits a unique solution $(\bY,\bZ)$. Moreover, we have:
\begin{align*}
Y_0 =&\;\; g(X_T, p_T) + \int_0^T H(t, X_{t-}, Z_t, \alpha_t, p_t, \nu_s) dt - \int_0^T Z_t^*\cdot d\mathcal{M}_t\\
= &\;\; g(X_T, p_T) + \int_0^T f(t, X_{t-}, \alpha_t, p_t, \nu_t) dt 
- \int_0^T Z_t^*\cdot (d\mathcal{M}_t - ( Q^*(t, \alpha_t, p_t, \nu_t) - Q^0)\cdot  X_{t-}dt)\\
= &\;\; g(X_T, p_T)+ \int_0^T f(t, X_{t-}, \alpha_t, p_t, \nu_t) dt - \int_0^T Z_t^*\cdot d\mathcal{M}^{(\balpha,\bp,\bnu)}_t.
\end{align*}
Since $\mathcal{M}^{(\balpha,\bp,\bnu)}$ is a martingale under the measure $\mathbb{Q}^{(\balpha,\bp,\bnu)}$, we take expectation under $\mathbb{Q}^{(\balpha,\bp,\bnu)}$ and obtain $J(\balpha,\bp,\bnu) = \mathbb{E}^{\mathbb{Q}^{(\balpha,\bp,\bnu)}}[Y_0]$. Now since $Y_0$ is $\mathcal{F}_0$-measurable, and $\mathbb{Q}^{(\balpha,\bp,\bnu)}$ coincides with $\mathbb{P}$ on $\mathcal{F}_0$, we obtain $J(\alpha,p,\nu) = \mathbb{E}^{\mathbb{P}}[Y_0]$.
\end{proof}

\vskip 6pt
Now we consider the following BSDE:
\begin{equation}\label{eq:bsde_optimality}
Y_t = g(X_T, p_T) + \int_t^T \hat H(s, X_{s-}, Z_s, p_s, \nu_s) ds - \int_t^T Z_s^*\cdot d\mathcal{M}_s,
\end{equation}
and we show that it characterizes the optimality of the control problem (\ref{eq:optimization_pb}).

\begin{proposition}
\label{prop:bsed_optimal_control}
For any measurable function $\bp$ from $[0,T]$ to $\mathcal{S}$ and any measurable function $\bnu$ from $[0,T]$ to $\mathcal{P}(A)$, the BSDE (\ref{eq:bsde_optimality}) admits a unique solution $(\bY,\bZ)$. The value function of the optimal control problem (\ref{eq:optimization_pb}) is given by $V(\bp,\bnu) =  \mathbb{E}^{\mathbb{P}}[Y_0]$ and the process $\hat{\balpha}^{(p,\nu)}$ defined by:
\begin{equation}\label{eq:optimal_control}
\hat \alpha_t^{(\bp,\bnu)} := \hat a(t, X_{t-}, Z_t, p_t)
\end{equation}
is an optimal control. In addition, if $\balpha'\in\mathbb{A}$ is an optimal control, we have $\alpha'_t = \hat \alpha_t^{(\bp,\bnu)}$, $dt \otimes d\mathbb{P}$-a.e.
\end{proposition}

\begin{proof}
The existence and uniqueness of the solution to (\ref{eq:bsde_optimality}) is easily verified by using the Lipschitz property of $\hat H$ provided by Lemma \ref{lem:lip_h_hat}. Let $(\bY,\bZ)$ be this unique solution and define the process $\hat{\balpha}$ by $\hat\alpha_t := \hat a(t, X_{t-}, Z_t, p_t)$. Recall the definition of $\hat a$ in equation (\ref{eq:def_a_hat}). We have:
\[
\hat a(t, X_{t-}, Z_t, p_t) = \sum_{i=1}^m \mathbbm{1}(X_{t-} = e_i)\hat a_i(t,Z_t, p_t) = X_{t-}^* \cdot \left(\sum_{i=1}^m \hat a_i(t,Z_t, p_t) e_i\right).
\]
Since $\hat a_i$ is measurable for each $i \in E$, we see that $\hat a$ is a measurable mapping from $[0,T]\times \mathbb{R}^m \times \mathbb{R}^m\times \mathcal{S}$ to $A$. Since both the processes $t \rightarrow X_{t-}$ and $\bZ$ are predictable, we conclude that $\hat\balpha$ is a predictable process and therefore an admissible control.

Now let us fix an arbitrary admissible control $\balpha \in \mathbb{A}$, and denote by $(\bY^\balpha, \bZ^\balpha)$ the solution of the corresponding BSDE (\ref{eq:bsde_total_cost}),
and by $(\bY, \bZ)$ the unique solution of:
\begin{equation}\label{eq:proof_bsde_control_optimal1}
Y _t = \int_t^T H(s,X_{s-}, Z_s, \hat\alpha_s, p_s, \nu_s)ds - \int_t^T Z_s^* \cdot d\mathcal{M}_s.
\end{equation}
Setting $\Delta \bY := \bY^\balpha - \bY$ and $\Delta \bZ := \bZ^\balpha - \bZ$ and computing the difference of the two BSDEs, we notice that $\Delta \bY$ and $\Delta \bZ$ solve the following BSDE:
\[
\Delta Y _t = \int_t^T [H(s,X_{s-}, Z_s^\balpha, \alpha_s, p_s, \nu_s) - H(s,X_{s-}, Z_s, \hat\alpha_s, p_s, \nu_s)]ds - \int_t^T \Delta Z_s^* \cdot d\mathcal{M}_s.
\]
We can further decompose the driver of the above BSDE as:
\begin{align*}
&H(s,X_{s-}, Z_s^\balpha, \alpha_s, p_s, \nu_s) - H(s,X_{s-}, Z_s, \hat\alpha_s, p_s, \nu_s)\\
&\hskip 45pt
=H(s,X_{s-}, Z_s^\balpha, \alpha_s, p_s, \nu_s) - H(s,X_{s-}, Z_s, \alpha_s, p_s, \nu_s) + H(s,X_{s-}, Z_s, \alpha_s, p_s, \nu_s)\\
&\hskip 95pt - H(s,X_{s-}, Z_s, \hat\alpha_s, p_s, \nu_s)\\
&\hskip 45pt
=[H(s,X_{s-}, Z_s, \alpha_s, p_s, \nu_s) - H(s,X_{s-}, Z_s, \hat\alpha_s, p_s, \nu_s)] + X_{s-}^*\cdot(Q(s, \alpha_s, p_s, \nu_s) - Q^0)\cdot \Delta Z.
\end{align*}
Define the processes $\bpsi$ and $\bgamma$ by $\psi_t := H(t,X_{t-}, Z_t, \alpha_t, p_t, \nu_t) - H(t,X_{t-}, Z_t, \hat\alpha_t, p_t, \nu_t)$ and $\gamma_t := (Q^*(t, \alpha_t, p_t, \nu_t) - Q^0) \cdot X_{t-}$. Therefore $(\Delta \bY, \Delta \bZ)$ appears as the solution to a linear BSDE of the form (\ref{eq:linear_bsde}) with $\bpsi$ and $\bgamma$ defined previously and $\bbeta = 0$. Clearly $\bpsi$ and $\bgamma$ are both predictable. Since $\hat\alpha_t$ minimizes the Hamiltonian, $\bpsi$ is nonnegative. The boundedness of $\bgamma$ follows from the boundedness of the transition rate function $q$. It remains to check that $1 + \gamma_t^*\cdot \psi_t^+\cdot (e_j - X_{t-}) \ge 0$.

When $X_{t-} = e_j$, the above inequality holds clearly. So we assume that $X_{t-} = e_i \neq e_j$. We have $\psi_t^+\cdot (e_j - X_{t-}) =\frac{m-1}{m} e_j - \sum_{i\neq j}\frac{1}{m} e_i $. Therefore when $X_{t-} = e_i \neq e_j$, we have:
\begin{align*}
&\gamma_t^*\cdot \psi_t^+\cdot (e_j - X_{t-}) = X_{t-}^*\cdot (Q(t, \alpha_t, p_t, \nu_t) - Q^0)\cdot \psi_t^+\cdot (e_j - X_{t-})\\
&\hskip 45pt
= e_i^*\cdot (Q(t, \alpha_t, p_t, \nu_t) - Q^0)\cdot (\frac{m-1}{m} e_j - \sum_{k\neq j}\frac{1}{m} e_k)\\
&\hskip 45pt
= \frac{m-1}{m} (q(t,i,j,\alpha_t,p_t,\nu_t) - q^0_{i,j}) - \frac{1}{m} \sum_{k\neq j}(q(t,i,k,\alpha_t,p_t,\nu_t) - q^0_{i,k})\\
&\hskip 45pt
= q(t,i,j,\alpha_t,p_t,\nu_t) - q^0_{i,j},
\end{align*}
where the last equality is due to the fact that $\sum_{k}(q(t,i,k,\alpha_t,p_t,\nu_t) - q^0_{i,k}) = 0$. Therefore we have:
\[
1 + \gamma_t^*\cdot \psi_t^+\cdot (e_j - X_{t-}) = 1 + q(t,i,j,\alpha_t,p_t,\nu_t) - q^0_{i,j} = q(t,i,j,\alpha_t,p_t,\nu_t) \ge 0.
\]
By Lemma \ref{lem:linear_bsde_comp}, we conclude that $\Delta \bY$ is nonnegative and in particular $Y_0^\balpha \ge Y_0$. Since $\balpha$ is an arbitrary admissible control, in light of Lemma \ref{lem:bsde_total_cost}, this means that $\mathbb{E}^{\mathbb{P}}[Y_0] \le \inf_{\balpha\in\mathbb{A}}J(\balpha,\bp,\bnu) = V(\bp,\bnu)$. Finally, we notice that $Y_0$ is the expected total cost when the control is $\hat\balpha$. We conclude that $\hat\balpha$ is an optimal control and $\mathbb{E}^{\mathbb{P}}[Y_0] = V(p,\nu)$.

Now we show that $\hat\balpha$ is the unique optimal control. Let $\balpha'$ be another optimal control. We consider the solution $(\bY',\bZ')$ to the following BSDE:
\begin{equation}\label{eq:proof_bsde_control_optimal2}
Y'_t = \int_t^T H(s,X_{s-}, Z'_s, \alpha'_s, p_s, \nu_s)ds - \int_t^T (Z'_s)^* \cdot d\mathcal{M}_s.
\end{equation}
Since $\balpha'$ is optimal, we have $\mathbb{E}^{\mathbb{P}}[Y_0'] = J(\balpha', \bp, \bnu) = V(\bp, \bnu) = \mathbb{E}^{\mathbb{P}}[Y_0]$. Now taking the difference of the BSDE (\ref{eq:proof_bsde_control_optimal1}) and (\ref{eq:proof_bsde_control_optimal2}), we obtain:
\begin{align*}
Y_0 - Y'_0
=&  \int_0^T \big[H(t, X_{t-}, Z_t, \hat\alpha_t, p_t, \nu_t) - H(t, X_{t-}, Z'_t, \alpha'_t, p_t, \nu_t)\big]dt - \int_0^T (Z_t - Z'_t)^*\cdot d\mathcal{M}_t\\
=& \int_0^T \big[X_{t-}^*\cdot(Q(t,\hat\alpha_t, p_t, \nu_t) - Q^0)\cdot Z_t - X_{t-}^*\cdot(Q(t,\alpha'_t, p_t,\nu_t) - Q^0)\cdot Z'_t\big]dt\\
&+ \int_0^T \big[f(t, X_{t-}, \hat\alpha_t, p_t, \nu_t) - f(t, X_{t-}, \hat\alpha_t, p_t, \nu_t)\big] dt - \int_0^T (Z_t - Z'_t)^*\cdot d\mathcal{M}_t\\
=& \int_0^T \big[f(t, X_{t-}, \hat\alpha_t, p_t, \nu_t) - f(t, X_{t-}, \alpha'_t, p_t, \nu_t) + X_{t-}^*\cdot(Q(t,\hat \alpha_t, p_t, \nu_t) - Q(t,\alpha'_t, p_t, \nu_t))\cdot Z_t \big]dt\\
& -  \int_0^T (Z_t - Z'_t)^*\cdot \big[d\mathcal{M}_t - (Q^*(t,\alpha'_t, p_t,\nu_t) - Q^0)\cdot X_{t-}dt\big]\\
=& \int_0^T \big[H(t, X_{t-}, Z_t, \hat\alpha_t, p_t, \nu_t) - H(t, X_{t-}, Z_t, \alpha'_t, p_t, \nu_t)\big] dt-  \int_0^T (Z_t - Z'_t)^*\cdot d\mathcal{M}^{(\balpha',\bp,\bnu)}_t.
\end{align*}
Taking $\mathbb{Q}^{(\balpha',\bp,\bnu)}$-expectations and using the fact that $\mathbb{Q}^{(\balpha',\bp,\bnu)}$ coincides with $\mathbb{P}$ in $\mathcal{F}_0$, we get:
\begin{align*}
0=&\;\;\mathbb{E}^{\mathbb{P}}[Y_0 - Y'_0]=\mathbb{E}^{\mathbb{Q}^{(\balpha',\bp,\bnu)}}[Y_0 - Y'_0]\\
=&\;\;\mathbb{E}^{\mathbb{Q}^{(\balpha',\bp,\bnu)}}\left[ \int_0^T \big[H(t, X_{t-}, Z_t, \hat\alpha_t, p_t, \nu_t) - H(t, X_{t-}, Z_t, \alpha'_t, p_t, \nu_t)\big] dt \right] \le 0,
\end{align*}
where the last inequality is due to the fact that $\hat\alpha_t$ minimizes the Hamiltonian. In fact, we have $\hat\alpha_t = \alpha'_t$, $dt\otimes d\mathbb{Q}^{(\balpha',\bp,\bnu)}$-a.e. If we assume otherwise, the last inequality would be strict, since the minimizer of the Hamiltonian is unique by Assumption \ref{hypo:lipschitz_optimizer}. Since $\mathbb{P}$ is equivalent to $\mathbb{Q}^{(\balpha',\bp,\bnu)}$, we have $\hat\alpha_t = \alpha'_t$, $dt\otimes d\mathbb{P}$-a.e.
\end{proof}

\section{Existence of Nash Equilibria}
We state the main result of this section:

\begin{theorem}
\label{theo:existence_nash_equilibrium}
Under Assumptions \ref{hypo:boundedness}, \ref{hypo:lipschitz_cost} and \ref{hypo:lipschitz_optimizer}, there exists a Nash equilibrium $(\balpha^*,\bp^*,\bnu^*)$ for the weak formulation of the finite state mean field game in the sense of Definition \ref{def:equilibrium}.
\end{theorem}

The rest of this section is devoted to the proof of Theorem \ref{theo:existence_nash_equilibrium}. As in the case of diffusion-based mean field games, we shall rely on a fixed point argument to show  existence of Nash equilibria. We start from a measurable function $\bp:[0,T]\rightarrow \mathcal{S}$ and a measurable function $\bnu:[0,T]\rightarrow\mathcal{P}(A)$ where we recall that $\mathcal{S}$ is the $m$-dimensional simplex which we identify with the space of probability measures on $E$, while $\mathcal{P}(A)$ is the space of probability measures on $A$. We then solve the BSDE (\ref{eq:bsde_optimality}), and obtain the solution $(\bY^{(\bp,\bnu)}, \bZ^{(\bp,\bnu)})$ as well as the optimal control $\hat \balpha^{(\bp,\bnu)}$ given by (\ref{eq:optimal_control}). Finally, we compute the probability measure $\hat{\mathbb{Q}}^{(\bp,\bnu)} := \mathbb{Q}^{(\hat \balpha^{(\bp,\bnu)}, \bp,\bnu)}$ as defined in (\ref{eq:q_measure}), and consider the push-forward measures of $\hat{\mathbb{Q}}^{(\bp,\bnu)}$ by $(X_t,\hat \alpha^{(\bp,\bnu)}_t)$. Clearly, we identified a Nash equilibrium if we find a fixed point for the mapping $(\bp,\bnu)\rightarrow \hat{\mathbb{Q}}^{(\bp,\bnu)}_{\# (X_t, \hat \alpha^{(\bp,\bnu)}_t)}$.

\vspace{3mm}
In practice however, the implementation of the fixed-point argument mentioned above is prone to several difficulties. The foremost challenge lies in the lack of results allowing us to identify compact subsets of the spaces of measurable functions from $[0,T]$ to $\mathcal{S}$ or $\mathcal{P}(A)$. This makes it difficult to apply Schauder's theorem or similar versions of fixed point theorems.
For this reason, we shall resort to different descriptions of the mean field for the state and the control. For the mean field of the state, since we have assumed from the very beginning that $\bX$ is a c\`adl\`ag process, we will directly deal with its probability law on the space $D$ of all c\`adl\`ag functions from $[0,T]$ to $E = \{e_1,\dots,e_m\}$ endowed with the Skorokhod topology. The space of probability measures on $D$ and its topological properties have been studied thoroughly (see \cite{jacod1987} for a detailed account), and a simple criterion for compactness is available. 

\vspace{3mm}
Unfortunately, resolving the corresponding issue for the control is more involved. Here, we adopt the  technique   based on the \emph{stable topology} used in \cite{carmona2015probabilistic}. Indeed, a measurable mapping from $[0,T]$ to $\mathcal{P}(A)$ can be viewed as a random variable defined on the space $([0,T], \mathcal{B}([0,T]), \mathcal{L})$ taking values in $\mathcal{P}(A)$. Here, $\mathcal{B}([0,T])$ is the Borel $\sigma$-field of $[0,T]$, $\mathcal{L}$ is the uniform probability measure on $[0,T]$ and $\mathcal{P}(A)$ is endowed with the Wasserstein-1 distance. To obtain compactness, the idea is to use randomization. We consider the space of probability measures on $[0,T]\times\mathcal{P}(A)$, denoted by $\mathcal{P}([0,T]\times\mathcal{P}(A))$. Then for each measurable mapping $\bnu$ from $[0,T]$ to $\mathcal{P}(A)$, we consider the measure $\eta$ on $[0,T]\times\mathcal{P}(A)$ given by $\eta (dt, dm) := \mathcal{L}(dt)\times \delta_{\nu_t}(dm)$ where $\delta$ is the Dirac measure. We may endow the space $\mathcal{P}([0,T]\times\mathcal{P}(A))$ with the so-called \emph{stable topology} introduced in \cite{jacod1981}, for which convenient results on compactness are readily available.

\vspace{3mm}
In the following, we detail the steps that lead to the existence of  Nash equilibria. We start by specifying the topology we use for the space of mean fields on the state as well as the control. We then properly define the mapping compatible with the definition of Nash equilibrium, we show its continuity, and construct a stable compact. Once these ingredients are in place, we apply Schauder's fixed point theorem to conclude.

\subsection{Topology for the space of mean fields}
We first consider the mean field for the state by endowing the state space $E:=\{e_1, \dots, e_m\}$ with the discrete metric $d_E(x,y) := \mathbbm{1}(x\neq y)$. Then it is well known that $(E, d_E)$ is a Polish space. Then, the Skorokhod space:
\begin{equation}\label{eq:def_skorokhod}
D := \{x: [0,T] \rightarrow E, \text{$x$ is c\`adl\`ag and left continuous on $T$}\}
\end{equation}
is endowed with the J1 metric:
\begin{equation}\label{eq:def_skorokhod_metric}
d_D(x,y) := \inf_{\lambda \in \Lambda} \max\{ \sup_{t\le T} |\lambda (t) - t|, \sup_{t\le T} |y(\lambda (t)) - x(t)|\}
\end{equation}
where $\Lambda$ is the set of all strictly increasing, continuous bijections from $[0,T]$ to itself. It can be proved that $d_D$ is a metric on $D$ and the metric space $(D, d_D)$ is a Polish space. Let us denote by $\mathcal{P}$ the collection of probability measures on $(D, d_D)$ endowed with the weak topology. Recall that the reference measure $\mathbb{P}$ is an element of $\mathcal{P}$. Let $\mathcal{P}_0$ be the subset of $\mathcal{P}$ defined by:
\begin{equation}\label{eq:def_p_0}
\mathcal{P}_0 := \{ \mathbb{Q}: \frac{d\mathbb{Q}}{d\mathbb{P}} = L, \text{\;\;with\;\;} \mathbb{E}^{\mathbb{P}}[L^2] \le C_0\}.
\end{equation}
where $C_0$ is a constant which we will specify later (see the proof of Proposition \ref{prop:stable_subset_by_phi}). We have the following result:

\begin{proposition}
\label{prop:topo_law_of_state}
$\mathcal{P}_0$ is convex and relatively compact in $\mathcal{P}$.
\end{proposition}

\begin{proof}
The convexity of $\mathcal{P}_0$ is trivial. Let us show that $\mathcal{P}_0$ is relatively compact. We proceed in three steps.

\vspace{3mm}
\noindent\emph{Step 1. } For $K\in\mathbb{N}$ and $\delta >0$, we define $D_{\delta,K}$ as the collection of paths in $D$ which meet the following criteria: (a) the path has no more than $K$ discontinuities, (b) the first jump time, if any, happens on or after $\delta$, (c) the last jump happens on or before $T-\delta$, and (d) the amounts of time between jumps are greater or equal than $\delta$. We now show that $D_{\delta,K}$ is compact in $D$. Since $D$ is Polish space it is enough to show the sequential compactness. Let us fix a sequence $x_n$ in $D_{\delta,K}$. For each $x_n$, we use the following notation: $k_n$ is the number of its jumps, $\delta\le t^1_n < t^2_n < \dots < t^{k_n}_n \le T-\delta$ are the times of its jumps. $\Delta t^1_n := t^1_n$ and $\Delta t^i_n := t^i_n - t^{i-1}_n$ for $i = 2,\dots, k_n$ are the time elapsed between consecutive jumps and $x_n^0, x_n^1, \dots, t^{k_n}_n$ are the value taken by $x_n$ in each interval defined by the jumps. Then we can represent $x_n$ using the vector $y_n$ of dimension $2(K+1)$:
\[
y_n = [k_n, \Delta t^1_n, \Delta t^2_n, \dots, \Delta t^{k_n}_n, 0, \dots, 0, x_n^0, x_n^1, \dots, x^{k_n}_n, 0, \dots, 0].
\]
In the above representation, the first coordinate of $y_n$ is the number of jumps. Coordinate $2$ to $K+1$ are the times elapsed between jumps defined above, and if there are fewer than $K$ jumps, we complete the vector by $0$. Coordinates $K+2$ to $2(K+1)$ are the values taken by the path $x$ and completed with $0$. Clearly there is a bijection from $x_n$ to $y_n$ by this representation. By the definition of the set $D_{\delta,K}$, we have $\Delta t_n^i \in [\delta, T]$ for $i\le k_n$ and $\sum_{i=1}^{k_n} \Delta t_n^i \le T-\delta$, whereas the rest of the coordinates of $y_n$ belongs to a finite set. This implies that $y_n$ lives in a compact and therefore we can extract a converging subsequence which we still denote by $y_n$. Again, since $k_n$ and the last $K+1$ components can only take finitely many values by their definition, therefore there exists $N_0$ such that for $n\ge N_0$, we have $k_n = k$ and $x_n^i = x^i$ for all $i\le k$. In addition we have $\Delta t_n^i$ converges to $\Delta t^i$ for all $i\le k$, where $\Delta t^i \ge \delta$ for all $i \le k$ and $\sum_{i=1}^k \Delta t^i \le T-\epsilon$. We consider the path represented by the vector $y$:
\[
y = [k, \Delta t^1, \Delta t^2, \dots, \Delta t^k, 0, \dots, 0, x^0, x^1, \dots, x^k, 0, \dots, 0].
\]
Clearly $x$ belongs to the set $D_{\delta. K}$ and it is straightforward to verify that $x_n$ converge to $x$ in J1 metric, where $x_n$ is the path represented by the vector $y_n$. This implies that $D_{\delta. K}$ is compact.

\vspace{3mm}
\noindent\emph{Step 2. }
Now we show that for any $\epsilon >0$, there exists $\delta >0$ and $K \in \mathbb{N}$ such that $\mathbb{P}(D_{\delta. K}) \ge 1- \epsilon$. Recall that $\mathbb{P}$ is the reference measure and under $\mathbb{P}$ the canonical process $\bX$ is a continuous-time Markov chain with transition rate matrix $Q^0$. Therefore the time of first jump, as well as the time between consecutive jumps thereafter, which we denote by $\Delta t_1, \Delta_2, \dots$ are i.i.d. exponential random variables of parameter $(m-1)$ under the measure $\mathbb{P}$. We have:
$$
\mathbb{P}(D_{\delta, K}) =\mathbb{P}[\Delta t_1 > T]
 + \sum_{k=1}^K \mathbb{P}\left[\{\Delta t_1 \ge \delta\} \cap \dots \cap \{\Delta t_k \ge \delta\} \cap \{\sum_{i=1}^{k+1} \Delta t_i > T \}\cap \{\sum_{i=1}^k \Delta t_i \le {T-\delta} \}\right].
$$
For each $k = 1,\dots,K$, we have:
\begin{align*}
&\mathbb{P}\left[\{\Delta t_1 \ge \delta\} \cap \dots  \cap \{\Delta t_k \ge \delta\} \cap \{\sum_{i=1}^{k+1} \Delta t_i > T \}\cap \{\sum_{i=1}^k \Delta t_i \le {T-\delta} \}\right]\\
&\hskip 45pt
\ge\;\; \mathbb{P}\left[\{\Delta t_1 \ge \delta\} \cap \dots \cap \{\Delta t_k \ge \delta\}\right] +  \mathbb{P}\left[ \{\sum_{i=1}^{k+1} \Delta t_i > T \}\cap \{\sum_{i=1}^k \Delta t_i \le {T-\delta} \}\right] - 1\\
&\hskip 45pt
= \;\; (\mathbb{P}[\Delta t_1 \ge \delta])^k +  \mathbb{P}\left[ \{\sum_{i=1}^{k+1} \Delta t_i > T \}\cap \{\sum_{i=1}^k \Delta t_i \le {T-\delta} \}\right] - 1\\
&\hskip 45pt
= \;\; (\exp(-k(m-1)\delta) - 1) + \exp(-(m-1)T)\frac{(m-1)^k(T-\delta)^k}{k!}.
\end{align*}
It follows that:
\begin{align*}
\mathbb{P}(D_{\delta, K})\ge&\;\;\sum_{k=1}^K (\exp(-k(m-1)\delta) - 1) +  \exp(-(m-1)T) \sum_{k=0}^K \frac{(m-1)^k(T-\delta)^k}{k!}\\
\ge&\;\;\sum_{k=1}^K (\exp(-k(m-1)\delta) - 1)
+ \exp(-(m-1)T) \sum_{k=0}^K \frac{(m-1)^kT^k}{k!} - (1-\exp(-(m-1)\delta)).
\end{align*}
We can first pick $K$ greater enough such that $(\exp(-(m-1)T)\sum_{k=0}^K(m-1)^kT^k / k!)$ is greater than $(1-\epsilon/2)$ and then pick $\delta$ small enough to make the rest of the terms greater than $-\epsilon/2$, which eventually makes $\mathbb{P}(D_{\delta, K})$ greater than $(1-\epsilon)$.

\vspace{3mm}
\noindent\emph{Step 3. }
Finally we show that $\mathcal{P}_0$ is tight. For any $\epsilon > 0$, by Step 2, we can pick $\delta >0$ and $K \in \mathbb{N}$ such that $\mathbb{P}(D \setminus D_{\delta, K}) \le (\epsilon/C_0)^2$. For all $\mathbb{Q} \in \mathcal{P}_0$, we have $d\mathbb{Q}/d\mathbb{P} = L$ and $(\mathbb{E}^{\mathbb{P}}[L^2])^{1/2}\le C_0$ and by Cauchy-Schwartz inequality we obtain:
\[
\mathbb{Q}(D\setminus D_{\delta, K})=\mathbb{E}^{\mathbb{P}}[L \cdot 1_{x \in D\setminus D_{\delta, K}}] \le (\mathbb{E}^{\mathbb{P}}[L^2])^{1/2} \mathbb{P}(D\setminus D_{\delta,K})^{1/2} \le \epsilon.
\]
This implies the tightness of $\mathcal{P}_0$. Finally by Prokhorov's Theorem we conclude that $\mathcal{P}_0$ is relatively compact.
\end{proof}

\vspace{3mm}
We now need to link the convergence of measures on path space to the convergence in $\mathcal{S}$, i.e. measures on state space. We define the function $\pi$ by:
$$
\pi: [0,T] \times \mathcal{P}(\mathbb{E}) \ni (t,\mu)\rightarrow [\mu_{\#X_t}(\{e_1\}), \mu_{\#X_t}(\{e_2\}), \dots, \mu_{\#X_t}(\{e_m\})]\in \mathcal{S}
$$
and prove the following result:

\begin{lemma}
\label{lem:convergence_in_simplex}
If $\mu^n\mapsto \mu$ in $\mathcal{P}$, there exists a subset $\mathcal{D}(\mu)$ of $[0,T)$ at most countable such that for all $t \not\in \mathcal{D}(\mu)$:
\begin{equation}\label{eq:converge_measure_state}
\lim_{n\rightarrow +\infty}\pi(t,\mu^n) = \pi(t,\mu).
\end{equation}
\end{lemma}
\begin{proof}
Define $\mathcal{D}(\mu) := \{0\le t\le T;\, \mu(X_t-X_{t-}\neq 0) >0\}$. By Lemma 3.12 in \cite{jacod1987}, the set $\mathcal{D}(\mu)$ is at most countable. In addition, we have $T\not\in\mathcal{D}(\mu)$ since all the paths in $D$ is left-continuous on $T$. In light of Proposition 3.14 in \cite{jacod1987}, we have $\mu^n_{\#X_t}$ converges to $\mu_{\#X_t}$ weakly for all $t \not\in\mathcal{D}(\mu)$. To conclude, we use the fact that $\mu^n_{\#X_t}$ for all $t\in[0,T]$ and $n$ are counting measure on the discrete set $E$. 
\end{proof}

We now turn to the mean field of control. Let $(\mathcal{P}(A), \mathcal{W}_1)$ be the space of probability measures on the compact set $A\subset\mathbb{R}^l$ endowed with the weak topology and metricized by the Wasserstein-1 distance. $(\mathcal{P}(A), \mathcal{W}_1)$ is  a Polish space. Since $A$ is compact, it is easy to show that $\mathcal{P}(A)$ is tight and therefore by Prokhorov's theorem $(\mathcal{P}(A), \mathcal{W}_1)$ is in fact compact. We endow $\mathcal{P}(A)$ with its Borel $\sigma-$algebra denoted by $\mathcal{B}(\mathcal{P}(A))$. We endow $[0,T]$ with its Borel $\sigma-$algebra $\mathcal{B}([0,T])$ and the (normalized) Lebesgue measure $\mathcal{L}(dt) := \frac{1}{T}dt$. Finally, we construct the product space $[0,T]\times \mathcal{P}(A)$ endowed with the $\sigma$-algebra $\mathcal{B}([0,T])\otimes \mathcal{B}(\mathcal{P}(A))$. The space of probability measures on $[0,T]\times \mathcal{P}(A)$ can be viewed as a randomized version of the space of mean field of control. We introduce the stable topology on this space:

\begin{definition}\label{def:stable_topology}
Let us denote by $\mathcal{R}$ the space of probability measures on $([0,T]\times \mathcal{P}(A), \mathcal{B}([0,T])\otimes \mathcal{B}(\mathcal{P}(A)))$. We call the stable topology of $\mathcal{R}$ the coarsest topology such that the mappings $\eta\rightarrow\int g(t,m) \eta(dt,dm)$ are continuous for all bounded and measurable mappings $g$ defined on $[0,T]\times \mathcal{P}(A)$ such that $m \rightarrow g(t,m)$ is continuous for each fixed $t\in[0,T]$.
\end{definition}
We collect a few useful results on the space $\mathcal{R}$ endowed with the stable topology.

\begin{proposition}
\label{prop:stable_topology_polish}
The topology space $\mathcal{R}$ is compact, metrizable, and Polish.
\end{proposition}

\begin{proof}
Notice that both $[0,T]$ and $\mathcal{P}(A)$ are Polish for their respective topologies. This implies that the $\sigma$-algebra $\mathcal{B}([0,T])\otimes \mathcal{B}(\mathcal{P}(A)))$ is separable. It follows from Proposition 2.10 in \cite{jacod1981} that $\mathcal{R}$ is metrizable. 

We now show that $\mathcal{R}$ is compact. Notice that for an element $\eta$ in $\mathcal{R}$, its first marginal is a probability measure on $[0,T]$ and its second marginal is a probability measure on $\mathcal{P}(A)$. It is trivial to see that both the spaces of probability measures on $[0,T]$ and on $\mathcal{P}(A)$ are tight and therefore relatively compact by Prokhorov's theorem. We then apply Theorem 2.8 in \cite{jacod1981} and obtain the compactness of $\mathcal{R}$.

Having showed that $\mathcal{R}$ is compact and metrizable, we see that $\mathcal{R}$ is separable. Compactness also leads to completeness. Therefore $\mathcal{R}$ is Polish space. Finally, we notice that $\mathcal{R}$ is also sequential compact since $\mathcal{R}$ is metrizable.
\end{proof}

The following result provides a more convenient way to characterize the convergence in the stable topology.

\begin{lemma}\label{lem:characterization_conv_stable_topo}
Denote by $\mathcal{H}$ the collection of mappings $f$ of the form $f(t,\nu) = 1_B(t) \cdot g(\nu)$ where $B$ is a Borel subset of $[0,T]$ and $g:\mathcal{P}(A)\rightarrow \mathbb{R}$ is a bounded Lipschitz function (with respect to the Wasserstein-1 distance on $\mathcal{P}(A)$). Then the stable topology introduced in Definition \ref{def:stable_topology} is the coarsest topology which makes the mappings $\eta \rightarrow \int_{[0,T]\times\mathcal{P}(A)} f(t,\nu) \eta(dt,d\nu)$ continuous for all $f \in \mathcal{H}$.
\end{lemma}

\begin{proof}
Let $\mathcal{H}_0$ be the collection of mappings $f$ of the form $f(t,\nu) = 1_B(t) \cdot g(\nu)$ where $B$ is a Borel subset of $[0,T]$ and $g:\mathcal{P}(A)\rightarrow \mathcal{R}$ is a bounded and uniformly continuous function. Then clearly we have $\mathcal{H}\subset\mathcal{H}_0$. By Proposition 2.4 in \cite{jacod1981}, the stable topology is the coarsest topology under which the mappings $\eta \rightarrow \int_{[0,T]\times\mathcal{P}(A)} f(t,\nu) \eta(dt,d\nu)$ are continuous for all $f \in \mathcal{H}_0$. Therefore, we only need to show that if $\eta^n$ is a sequence of elements in $\mathcal{R}$ such that $\int f(t,\nu) \eta^n(dt,d\nu) \rightarrow \int f(t,\nu) \eta^0(dt,d\nu)$ for all $f \in \mathcal{H}$, then we have $\int f(t,\nu) \eta^n(dt,d\nu) \rightarrow \int f(t,\nu) \eta^0(dt,d\nu)$ for all $f \in \mathcal{H}_0$ as well.

\vspace{3mm}
Now let us fix $f \in \mathcal{H}_0$ with $f(t,\nu) = 1_B(t)\cdot g(\nu)$, Note that $\mathcal{P}(A)$ is a compact metric space and $g$ is a bounded, uniformly continuous and real-valued function. A famous result from \cite{georgano1967} (see also \cite{miculescu2000}) shows that $g$ can be approximated uniformly by bounded Lipschitz continuous function. That is, for all $\epsilon>0$, we can find $g_{\epsilon} \in \mathcal{H}$ such that $\sup_{\nu \in \mathcal{P}(A)} |g_{\epsilon}(\nu) - g(\nu)| \le \epsilon/3$. By our assumption we have $\int 1_B(t)g_{\epsilon}(\nu) \eta^n(dt,d\nu) \rightarrow \int 1_B(t)g_{\epsilon}(\nu) \eta^0(dt,d\nu)$. Therefore there exists $N_0$ such that $|\int 1_B(t) g_{\epsilon}(\nu) \eta^n(dt,d\nu) - \int 1_B(t) g_{\epsilon}(\nu) \eta^0(dt,d\nu)| \le \epsilon/3$ for all $n \ge N_0$. Combining these facts we have, for $n\ge N_0$:
\begin{align*}
|\int 1_B(t) g(\nu) \eta^n(dt,d\nu) - \int 1_B(t) g(\nu) \eta^0(dt,d\nu)|
 \le& |\int 1_B(t) g_{\epsilon}(\nu) \eta^n(dt,d\nu) - \int 1_B(t) g_{\epsilon}(\nu) \eta^0(dt,d\nu)| \\
 & \hskip -45pt
 + \int 1_B(t)|g_{\epsilon}(\nu) - g(\nu)|\eta^n(dt,d\nu) + \int 1_B(t)|g_{\epsilon}(\nu) - g(\nu)|\eta^0(dt, d\nu)\\
 \le& \epsilon/3 + \epsilon/3 + \epsilon/3 = \epsilon,
\end{align*}
which shows that $\int f(t,\nu) \eta^n(dt,d\nu) \rightarrow \int f(t,\nu) \eta^0(dt,d\nu)$.
\end{proof}

Now we consider the following subset of $\mathcal{R}$:
\[
\mathcal{R}_0:= \{\eta\in\mathcal{R};\, \text{the marginal distribution of $\eta$ on $[0,T]$ is $\mathcal{L}$}\}.
\]
We have the following result:
\begin{lemma}\label{lem:stable_topology_subset}
$\mathcal{R}_0$ is a convex and compact subset of $\mathcal{R}$.
\end{lemma}

\begin{proof}
We apply Theorem 2.8 in \cite{jacod1981}. In particular, we verify without difficulty that $\{\eta^{[0,T]};\, \eta\in\mathcal{R}_0\} = \{\mathcal{L}\}$ is compact and $\{\eta^{\mathcal{P}(A)};\, \eta\in\mathcal{R}_0\}$ is a subset of $\mathcal{P}(\mathcal{P}(A))$, which is relatively compact as well.
\end{proof}

For any $\eta \in \mathcal{R}_0$, since its first marginal is $\mathcal{L}$, by disintegration we can write $\eta(dt,dm) = \mathcal{L}(dt) \times \eta_t(dm)$ where the mapping $[0,T]\ni t\rightarrow \eta_t(\cdot) \in \mathcal{P}(\mathcal{P}(A))$ is a measurable mapping and the decomposition is unique up to almost everywhere equality.
On the other hand, for any measurable function $\nu:[0,T]\rightarrow\mathcal{P}(A)$, we may construct an element $\Psi(\nu)$ in $\mathcal{R}_0$ by:
\begin{equation}
\Psi(\nu)(dt, dm) := \mathcal{L}(dt)\times \delta_{\nu_t}(dm).
\end{equation}
Since we have changed the way we represent the mean field of control, we need to modify accordingly the definition of transition rate matrix as well as the cost functionals in order to make them compatible with the randomization procedure. For any function $F:\mathcal{P}(A)\rightarrow\mathbb{R}$ possibly containing other arguments, we denote $\ubar F: \mathcal{P}(\mathcal{P}(A)) \rightarrow \mathbb{R}$ by $\ubar F(m) := \int_{\nu \in \mathcal{P}(A)}F(\nu) m(d\nu)$, which we call the randomized version of $F$. Obviously we have $\ubar F(\delta_{\nu}) = F(\nu)$. In this way, we define without any ambiguity  the randomized version $\ubar q$ of the rate function $q$, as well as its matrix representation $\ubar Q$. We also define $\ubar f$ as the randomized version of cost functional $f$. Since the terminal cost $g$ does not depend on the mean field of control, we do not need to consider its randomized version.

\vspace{3mm}
Recall from Assumption \ref{hypo:lipschitz_optimizer} that the minimizer $\hat a_i$ of the reduced Hamiltonian is only a function of $t$, $z$ and $p$. Consequently, for $\ubar H$, $\ubar H_i$, $\ubar{\hat{H}}$ and $\ubar{\hat{H}}^i$, which are  the randomized version of $H$, $H_i$, $\hat H$ and $\hat H_i$ respectively, we still have:
\begin{align*}
\ubar{\hat{H}}(t, x, z, p, m) =& \inf_{\alpha\in A} \ubar H(t, x, z, \alpha, p, m),\\
\hat a(t, x, z, p) =& \arg\inf_{\alpha\in A} \ubar H(t, x, z, \alpha, p, m).
\end{align*}
In addition, we have the following result on the Lipschitz property of $\ubar{\hat H}$ and $\hat a$:

\begin{lemma}
\label{lem:lip_h_hat_randomized}
There exists a constant $C>0$ such that for all $(\omega,t)\in\Omega\times(0,T]$, $p,p'\in\mathcal{S}$, $\alpha, \alpha' \in A$, $z,z'\in\mathbb{R}^m$ and $m,m' \in \mathcal{P}(\mathcal{P}(A))$, we have:
\begin{equation}
\label{eq:lipschitz_q_hat_randomized}
|\ubar{q}(t,i,j,\alpha,p,m) - \ubar{q}(t,i,j,\alpha',p',m')| \le C(\|\alpha - \alpha'\| + \|p - p'\| + \bar{\mathcal{W}}_1(m, m')),
\end{equation}
and
\begin{equation}
\label{eq:lipschitz_h_hat_randomized}
|\ubar{\hat{H}}(t,X_{t-},z,p,m) - \ubar{\hat{H}}(t,X_{t-},z',p',m')| 
\le C\|z - z'\|_{X_{t-}} + C(1 +  \|z\|_{X_{t-}})(\|p - p'\| + \bar{\mathcal{W}}_1(m, m')).
\end{equation}
\end{lemma}

\begin{proof}
We have:
\begin{align*}
&\hskip -45pt
|\ubar{\hat{H}}(t,X_{t-},z,p,m) - \ubar{\hat{H}}(t,X_{t-},z',p',m')| \\
\le& |\ubar{\hat{H}}(t,X_{t-},z,p,m) - \ubar{\hat{H}}(t,X_{t-},z,p,m')| + |\ubar{\hat{H}}(t,X_{t-},z,p,m') - \ubar{\hat{H}}(t,X_{t-},z',p',m')|\\
\le&\left| \int_{\nu\in\mathcal{P}(A)} \hat{H}(t,X_{t-},z,p,\nu) (m(d\nu) - m'(d\nu))\right| \\
&+ \int_{\nu\in\mathcal{P}(A)}| \hat{H}(t,X_{t-},z,p,\nu) -  \hat{H}(t,X_{t-},z',p',\nu)| m'(d\nu)\\
\le& \left| \int_{\nu\in\mathcal{P}(A)} \hat{H}(t,X_{t-},z,p,\nu) (m(d\nu) - m'(d\nu))\right| + C(1+\|z\|_{X_{t-}}) \|p - p'\|+ C\|z - z'\|_{X_{t-}}.
\end{align*}
Since the space $\mathcal{P}(A)$ is compact and the mapping $\nu \rightarrow \hat{H}(t,X_{t-},z,p,\nu)$ is Lipschitz, with Lipschitz constant equal to $C(1+\|z\|_{X_{t-}})$, Kantorovich-Rubinstein duality theory implies:
\[
\left| \int_{\nu\in\mathcal{P}(A)} \hat{H}(t,X_{t-},z,p,\nu) (m(d\nu) - m'(d\nu))\right| \le C(1+\|z\|_{X_{t-}}) \bar{\mathcal{W}_1}(m,m'),
\]
where $\bar{\mathcal{W}_1}(m,m')$ is the Wasserstein-1 distance on the space of probability measure on $\mathcal{P}(A)$ whose definition we recall for the sake of definiteness:
\begin{equation}\label{eq:def_w1_ppa}
\bar{\mathcal{W}_1}(m,m') := \inf_{\pi\in\mathcal{P}(\mathcal{P}(A)\times\mathcal{P}(A))}\int_{\mathcal{P}(A)\times\mathcal{P}(A)} \mathcal{W}_1(\nu, \nu') \pi(d\nu, d\nu').
\end{equation}
Combined with the estimation above, we obtain the desired inequality for $\ubar{\hat H}$. The Lipschitz property for $\ubar q$ can be proved in the same way.
\end{proof}

\subsection{Mapping fixed points}
We now define the mapping whose fixed points characterize the Nash equilibria of the mean field game in its weak formulation. For any $(\mu, \eta) \in \mathcal{P} \times \mathcal{R}_0$, where $\eta$ has the disintegration $\eta(dt,dm) = \mathcal{L}(dt) \times \eta_t(dm)$, we consider the solution $(\bY^{(\mu,\eta)}, \bZ^{(\mu,\eta)})$ to the BSDE:
\begin{equation}\label{eq:mapping_bsde}
Y_t = g(X_T, p_T) + \int_t^T \ubar{\hat H}(s, X_{s-}, Z_s, \pi(s,\mu), \eta_s) ds - \int_t^T Z_s^*\cdot d\mathcal{M}_s.
\end{equation}
Denote by $\hat\balpha^{(\mu,\eta)}$ the predictable process $t\rightarrow\hat a(t, X_{t-}, Z^{(\mu,\eta)}_t, \pi(t,\mu))$, which is the optimal control of the player faced with the mean field $(\mu,\eta)\in \mathcal{P}(E) \times \mathcal{R}_0$. Next, we consider the scalar martingale $\bL^{(\mu,\eta)}$ defined by:
\begin{equation}\label{eq:scalar_martingale_phi}
L^{(\mu,\eta)}_t := \int_0^t X_{s-}^*\cdot(\ubar Q(s,\hat\alpha^{(\mu,\eta)}_s,\pi(s,\mu), \eta_s) - Q^0)\cdot d\mathcal{M}_s.
\end{equation}
Define the probability measure $\hat{\mathbb{Q}}^{(\mu,\eta)}$ by:
\begin{equation}\label{eq:doleans_dade_exponential_phi}
\frac{d\hat{\mathbb{Q}}^{(\mu,\eta)}}{d\mathbb{P}} := \mathcal{E}(\bL^{(\mu,\eta)})_T,
\end{equation}
where $\mathcal{E}(\bL^{(\mu,\eta)})$ is the Dol\'eans-Dade exponential of the martingale $\bL^{(\mu,\eta)}$. Finally we define the mappings $\Phi^\mu$, $\Phi^\eta$ and $\Phi$ respectively by:

\begin{equation}
\label{eq:def_mapping_fixed_point1}
\Phi^\mu: \mathcal{P} \times \mathcal{R}_0 \ni  (\mu,\eta) \rightarrow \hat{\mathbb{Q}}^{(\mu,\eta)} \in \mathcal{P}\\
\end{equation}

\begin{equation}
\label{eq:def_mapping_fixed_point2}
\Phi^\eta: \mathcal{P} \times \mathcal{R}_0 \ni (\mu,\eta) \rightarrow  \mathcal{L}(dt)\times\delta_{\hat{\mathbb{Q}}^{(\mu,\eta)}_{\#\hat\alpha^{(\mu,\eta)}_t}}(d\nu) \in\mathcal{R}_0
\end{equation}

\begin{equation}
\label{eq:def_mapping_fixed_point}
\Phi: \mathcal{P} \times \mathcal{R}_0 \ni (\mu,\eta) \rightarrow \big(\Phi^\mu(\mu,\eta), \Phi^\eta(\mu,\eta)\big) \in \mathcal{P} \times \mathcal{R}_0.
\end{equation}

\begin{remark}\label{rem:well_posedness_mapping_phi}
Before delving into its properties of $\Phi$, we first need to show that the mapping $\Phi$ is well-defined. More specifically, we need to show that given $(\mu, \eta) \in \mathcal{P} \times \mathcal{R}_0$, the outputs $\hat{\mathbb{Q}}^{(\mu,\eta)}$ and $\mathcal{L}(dt)\times\delta_{\hat{\mathbb{Q}}^{(\mu,\eta)}_{\#\hat\alpha^{(\mu,\eta)}_t}}(d\nu))$ does not depend on which solution to the BSDE (\ref{eq:mapping_bsde}) we use to construct  $\hat\balpha^{(\mu,\eta)}$, $\bL^{(\mu,\eta)}$ and $\mathcal{E}(L^{(\mu,\eta)})$. To this end, let us consider $(\bY, \bZ)$ and $(\bY', \bZ')$ two solutions to BSDE (\ref{eq:mapping_bsde}), $\hat\balpha$ and $\hat\balpha'$ the corresponding optimal controls, $\bL$ and $\bL'$ the corresponding martingales defined in (\ref{eq:scalar_martingale_phi}), and $\mathbb{Q}$ and $\mathbb{Q}'$ the resulting probability measures defined in (\ref{eq:doleans_dade_exponential_phi}). By uniqueness of solution to (\ref{eq:mapping_bsde}), we have $\mathbb{E}\left[\int_0^T \|Z'_t - Z_t\|_{X_{t-}}^2 dt\right] = 0$. Using the Lipschitz continuity of $\hat a$ and $q$, it is straightforward to show $\mathbb{E}\left[\int_0^T \|\alpha'_t - \alpha_t\|^2 dt\right] = 0$ and eventually $\mathbb{Q} = \mathbb{Q}'$.
\end{remark}

\begin{proposition}
\label{prop:stable_subset_by_phi}
Let us denote by $\bar{\mathcal{P}}_0$ the closure of the set $\mathcal{P}_0$ defined in (\ref{eq:def_p_0}). Then the set $\bar{\mathcal{P}}_0 \times \mathcal{R}_0$ is stable for the mapping $\Phi$.
\end{proposition}

\begin{proof}
It suffices to show that for all $(\mu,\eta) \in  \mathcal{P} \times \mathcal{R}_0$, we have $\Phi^\mu(\mu,\eta) \in \mathcal{P}_0$. By the definition of $\mathcal{P}_0$ in (\ref{eq:def_p_0}), this boils down to showing that there exists a constant $C_0 > 0$ such that for all $(\mu,\eta)$, we have:
\[
\mathbb{E}^{\mathbb{P}}[(\mathcal{E}(\bL^{(\mu,\eta)})_T)^2] \le C_0.
\]
Let us denote $W_t := \mathcal{E}(\bL^{(\mu,\eta)})_t$. By It\^o's lemma we have:
\[
d(W_t^2) = 2 W_{t-} dW_t + d[W, W]_t
\]
since $dL^{(\mu,\eta)}_t = X_{t-}^*\cdot(\ubar Q(t,\hat\alpha^{(\mu,\eta)}_t,\pi(t,\mu), \eta_t) - Q^0)\cdot \psi^+_t\cdot d\mathcal{M}_t$ and $dW_t = W_{t-} dL^{(\mu,\eta)}_t$, denoting $I_t:= \psi^+_t\cdot(\ubar Q^*(t,\hat\alpha^{(\mu,\eta)}_t,\pi(t,\mu), \eta_t) - Q^0)\cdot X_{t-}$ we have:
\[
d(W_t^2) = 2 W^2_{t-} dL^{(\mu,\eta)}_t + W^2_{t-} I^*_t \cdot d[\mathcal{M}, \mathcal{M}]_t \cdot I_t.
\]
We know that the optional quadratic variation of $\bcM$ can be decomposed as:
\[
[\bcM, \bcM]_t = G_t + \langle \bcM, \bcM \rangle_t =G_t + \int_0^t \psi_s ds,
\]
where $G$ is a martingale. Therefore we have:
\[
d(W_t^2) = 2 W^2_{t-} dL^{(\mu,\eta)}_t + W^2_{t-} I^*_t \cdot dG_t \cdot I_t + W^2_{t-} I^*_t \cdot \psi_t \cdot I_t dt.
\]
Let $T_n$ be a sequence of stopping time converging to $+\infty$ which localizes both the local martingales $\int_0^t W^2_{s-} dL^{(\mu,\eta)}_s$ and $\int_0^t W^2_{s-} I^*_s \cdot dG_s \cdot I_s$. Then integrating the above SDE between $0$ and $T\wedge T_n$ and taking the expectation under $\mathbb{P}$ we obtain:
\begin{align*}
\mathbb{E}^{\mathbb{P}}[W_{T\wedge T_n}^2] =& 1 + \mathbb{E}^{\mathbb{P}}\left[\int_0^{T\wedge T_n}W^2_{t-} I^*_t \cdot \psi_t \cdot I_t dt\right] = 1 + \mathbb{E}^{\mathbb{P}}\left[\int_0^{T\wedge T_n}W^2_{t-} I^*_t \cdot \psi_t \cdot I_t dt\right]\\
 \le &1 + \int_0^T\mathbb{E}^{\mathbb{P}}\left[W^2_{t\wedge T_n} I^*_{t\wedge T_n} \cdot \psi_{t\wedge T_n} \cdot I_{t\wedge T_n}\right]dt \le 1 + C_0 \int_0^T\mathbb{E}^{\mathbb{P}}[W^2_{t\wedge T_n}].
\end{align*}
Here we have used Tonelli's theorem as well as the fact that $I^*_s \cdot \psi_s \cdot I_s$ is bounded by a constant $C_0$ independent of $\mu,\eta$ and $n$, which is a consequence of the boundedness of the transition rate function $q$. Now applying Gronwall's lemma we obtain:
\[
\mathbb{E}^{\mathbb{P}}[W_{T\wedge T_n}^2] \le C_0.
\]
where the constant $C_0$ does not depend on $n$, $\mu$ or $\eta$. Notice that $W_{T\wedge T_n}^2$ converges to $W_t^2$ almost surely, we apply Fatou's lemma and obtain $\mathbb{E}^{\mathbb{P}}[W_T^2] \le C_0$.
\end{proof}

\subsection{Existence of Nash equilibria}
The last missing piece in applying Schauder's fixed point theorem is to show the continuity of the mapping $\Phi$ on $\mathcal{P} \times \mathcal{R}_0$ for the product topology. To this end, we show the continuity of the mappings $\Phi^\mu$ and $\Phi^\eta$. Notice that both $\mathcal{P}$ and $\mathcal{R}_0$ are metrizable, so we only need to show sequential continuity.

Let us fix a sequence $(\mu^{(n)}, \eta^{(n)})_{n\ge 1}$ converging to $(\mu^{(0)},\eta^{(0)})$  in $\mathcal{P} \times \mathcal{R}_0$, with the decomposition $\eta^{(n)}(dt, d\nu) = \mathcal{L}(dt)\times \eta_t^{(n)} (d\nu)$. To simplify the notation we denote $\bY^{(\mu^{(n)}, \eta^{(n)})}$, $\bZ^{(\mu^{(n)}, \eta^{(n)})}$, $\hat \balpha^{(\mu^{(n)}, \eta^{(n)})}$,  $\bL^{(\mu^{(n)}, \eta^{(n)})}$, $\hat{\mathbb{Q}}^{(\mu^{(n)},\eta^{(n)})}$ respectively by $\bY^{(n)}$, $\bZ^{(n)}$, $\hat\balpha^{(n)}$, $\bL^{(n)}$ and $\mathbb{Q}^{(n)}$ for $n\ge 0$. We also denote by $\mathbb{E}^{(n)}$ the expectation under $\mathbb{Q}^{(n)}$ and $p^{(n)}_t = \pi(t, \mu^{(n)})$, whereas $\mathbb{E}$ still denotes the expectation under the reference measure $\mathbb{P}$. 

We start by proving the continuity of $\Phi^{\mu}$, or equivalently the convergence of $\mathbb{Q}^{(n)}$ toward $\mathbb{Q}^{(0)}$. We divide the proof into several intermediary results.

\begin{lemma}
\label{lem:boundedness_z}
Without any loss of generality, we may assume that there exists a constant $C$ such that  $\|Z_t^{(0)}\|_{X_{t-}} \le C$ for all $(\omega,t) \in \Omega\times[0,T]$.
\end{lemma}
\begin{proof}
We consider the following ODE of unknown $V_t = [V_1(t), \dots, V_m(t)] \in \mathbb{R}^m$:
\begin{equation}\label{eq:boundedness_z_ode}
\begin{array}{rl}
0 &=\displaystyle \frac{d V_i(t)}{dt} + \ubar{\hat H}_i(t, V(t), p_t^{(0)}, \eta_t^{(0)}) + \sum_{j\neq i}[V_j(t) - V_i(t)],\\
V_i(T) &= g(e_i, p_T^{(0)}),\quad i = 1,\dots, m.
\end{array}
\end{equation}
Set $\zeta:[0,T]\times \mathbb{R}^m \ni (t,v) \rightarrow [\zeta_1(t,v), \dots, \zeta_m(t,v)] \in \mathbb{R}^m$ where $\zeta_i(t,v) := \ubar{\hat H}_i(t, v, p_t^{(0)}, \eta_t^{(0)}) + \sum_{j\neq i}[v_j - v_i]$. By Lemma \ref{lem:lip_h_hat_randomized}, we see that $t \rightarrow \zeta(t,v)$ is measurable for all $v\in \mathbb{R}^m$ and $v \rightarrow \zeta(t,v)$ is Lipschitz in $v$ uniformly in $t$. By Theorem 1 and Theorem 2 in \cite{filippov2013}, the ODE (\ref{eq:boundedness_z_ode}) admits a unique solution on the interval $[0,T]$, which is absolutely continuous. Now we define $Y_t = \sum_{i=1}^m \mathbbm{1}(X_t = e_i)V_i(t)$ and $Z_t = V_t$. By continuity of $V$, we have $\Delta Y_t := Y_t - Y_{t-} = V_t^* (X_t - X_{t-}) = Z_t^* \cdot\Delta X_t$. Applying Ito's formula to $\bY$, we obtain:
\begin{align*}
Y_t =& Y_T - \int_t^T  \sum_{i=1}^m \mathbbm{1}(X_t = e_i) \dot{ V}_i(s) ds - \sum_{t < s\le T} \Delta Y_s \\
=& g(X_T, p_T^{(0)}) + \int_t^T \sum_{i=1}^m \mathbbm{1}(X_t = e_i)  \ubar{\hat H}^i(t, V(t), p_t^{(0)}, \eta_t^{(0)})\\
& + \int_t^T \sum_{i=1}^m \mathbbm{1}(X_t = e_i)\sum_{j\neq i}[V_j(t) - V_i(t)] - \int_t^T Z_s^* \cdot d X_s\\
= & g(X_T, p_T^{(0)}) + \int_t^T \ubar{\hat H}^i(s, X_s, Z_s, p_s^{(0)}, \eta_s^{(0)}) ds - \int_t^T Z_s^*\cdot d\mathcal{M}_s,
\end{align*}
where in the last equality we used the fact that $d X_s = Q^0 \cdot X_{s-} ds + d\mathcal{M}_s$ and $V_t = Z_t$. Therefore $(\bY, \bZ)$ and $(\bY^{(0)}, \bZ^{(0)})$ solves the same BSDE. As we have discussed in Remark \ref{rem:well_posedness_mapping_phi}, we may assume that $\bZ^{(0)} = \bZ$. Therefore $Z^{(0)}_t = V(t)$. It follows from the continuity of $t\rightarrow V(t)$ that $\|Z_t^{(0)}\|_{X_{t-}}$ is bounded for all $\omega\in\Omega$ and $t\in[0,T]$ by a uniform constant.
\end{proof}

Now we show that $\bZ^{(n)}$ converges toward $\bZ^{(0)}$.

\begin{proposition}
\label{prop:convergence_bsde_solution}
We have:
\begin{equation}
\label{eq:convergence_bsde_solution}
\lim_{n\to\infty}\mathbb{E}\left[\int_0^T \|Z_t^{(n)} - Z_t^{(0)}\|_{X_{t-}}^2 dt\right] =0.
\end{equation}
\end{proposition}

\begin{proof}
By Lemma \ref{lem:bsde_apriori}, it suffices to check that:
\[
I_n(t) := \mathbb{E}\left[(\int_t^T\ubar{\hat H}(s, X_{s-}, Z_s^{(0)}, p^{(n)}_s, \eta^{(n)}_s) - \ubar{\hat H}(s, X_{s-}, Z_s^{(0)}, p^{(0)}_s, \eta^{(0)}_s) ds)^2 \right]
\]
converges to $0$ for all $t\le T$, and that $I_n(t)$ is bounded by $C$ uniformly in $t$ and $n$. We also need to check $J_n := \mathbb{E}[|g(X_T, p^{(n)}_T) - g(X_T, p^{(0)}_T)|^2]$ converges to $0$. By Lipschitz property of the cost functional $g$ and Lemma \ref{lem:convergence_in_simplex}, we have:
\[
J_n \le C \|p^{(n)}_T - p^{(0)}_T\|^2 = C\| \pi(\mu^{(n)}, T) - \pi(\mu^{(0)}, T)\|^2 \rightarrow 0, 
\]
as $n\rightarrow +\infty$. To check the uniform boundedness of $I_n(t)$, we recall from Lemma \ref{lem:lip_h_hat} that:
$$
|\ubar{\hat H}(t, X_{t-}, Z_t^{(0)}, p^{(n)}_t, \eta^{(n)}_t) - \ubar{\hat H}(t, X_{t-}, Z_t^{(0)}, p^{(0)}_t, \eta^{(0)}_t)| 
\le C(1 + \|Z_t^{(0)}\|_{X_{t-}}) ( \|p^{(n)}_t - p^{(0)}_t\| + \bar{\mathcal{W}}_1(\eta^{(n)}_t, \eta^{(0)}_t) ),
$$
where $\bar{\mathcal{W}}_1$ is the Wasserstein distance on the space $\mathcal{P}(\mathcal{P}(A))$. Clearly $\|p^{(n)}_t - p^{(0)}_t\|$ can be bounded by a constant since $p^{(n)}_t$ is in the simplex $\mathcal{S}$. On the other hand, we have:
\[
\bar{\mathcal{W}}_1(\eta^{(n)}_t, \eta^{(0)}_t) \le \int_{(\nu_1, \nu_2)\in\mathcal{P}(A)^2} \mathcal{W}_1(\nu_1, \nu_2) \eta^{(n)}_t(d\nu_1)\eta^{(0)}_t(d\nu_2).
\]
Since $A$ is compact, $\mathcal{W}_1(\nu_1, \nu_2)$ for $(\nu_1, \nu_2)\in\mathcal{P}(A)^2$ is bounded, which implies that $\bar{\mathcal{W}}_1(\eta^{(n)}_t, \eta^{(0)}_t)$ is also bounded by a constant uniformly in $n$ and $t$. This implies:
\[
I_n(t) \le C\mathbb{E}[\int_t^T (1 + \|Z_s^{(0)}\|_{X_{s-}}) ds] \le C( 1+ (\mathbb{E}[\int_0^T \|Z_s^{(0)}\|^2_{X_{s-}} ds])^{1/2}) < +\infty,
\]
which means that $I_n(t)$ is uniformly bounded in $n$ and $t$. To show that $I_n(t)$ converges to $0$, we write:
\begin{align*}
I_n(t) \le&\;\; 2 \mathbb{E}\left[\left(\int_t^T (\ubar{\hat H}(s, X_{s-}, Z_s^{(0)}, p^{(n)}_s, \eta^{(n)}_s) - \ubar{\hat H}(s, X_{s-}, Z_s^{(0)}, p^{(0)}_s, \eta^{(n)}_s))dt\right)^2 \right]\\
&\;\; + 2 \mathbb{E}\left[\left(\int_t^T( \ubar{\hat H}(s, X_{s-}, Z_s^{(0)}, p^{(0)}_s, \eta^{(n)}_s) - \ubar{\hat H}(s, X_{s-}, Z_s^{(0)}, p^{(0)}_s, \eta^{(0)}_s))dt\right)^2 \right]\\
\le&\;\; 2C \mathbb{E}\left[\int_t^T (1 + \|Z_s^{(0)}\|_{X_{s-}})^2\|p_s^{(n)} - p_s^{(0)}\|^2 ds\right]\\
&\;\; + 2 \mathbb{E}\left[\left(\int_t^T (\ubar{\hat H}(s, X_{s-}, Z_s^{(0)}, p^{(0)}_s, \eta^{(n)}_s) - \ubar{\hat H}(s, X_{s-}, Z_s^{(0)}, p^{(0)}_s, \eta^{(0)}_s))dt\right)^2 \right].
\end{align*}
By Lemma \ref{lem:convergence_in_simplex}, we have $ (1 + \|Z_s^{(0)}\|_{X_{s-}})^2\|p_s^{(n)} - p_s^{(0)}\|^2 \rightarrow 0$, $dt\otimes d\mathbb{P}$-a.e. On the other hand, we have:
\[
(1 + \|Z_s^{(0)}\|_{X_{s-}})^2\|p_s^{(n)} - p_s^{(0)}\|^2 \le C (1 + \|Z_s^{(0)}\|_{X_{s-}})^2,
\]
where the right hand side is $ds\otimes d\mathbb{P}$-integrable. Therefore by the dominated convergence theorem, we obtain:
\[
 \mathbb{E}\left[\int_t^T (1 + \|Z_s^{(0)}\|_{X_{s-}})^2\|p_s^{(n)} - p_s^{(0)}\|^2 ds\right] \rightarrow 0, 
\]
as $n\rightarrow +\infty$. It remains to show that:
\[
K_n := \mathbb{E}\left[\left(\int_t^T (\ubar{\hat H}(s, X_{s-}, Z_s^{(0)}, p^{(0)}_s, \eta^{(n)}_s) - \ubar{\hat H}(s, X_{s-}, Z_s^{(0)}, p^{(0)}_s, \eta^{(0)}_s))dt\right)^2\right]
\]
converges to $0$. For a fix $w\in\Omega$ and $s \le T$, we have:
\begin{align*}
&\hskip -45pt
\int_t^T(\ubar{\hat H}(s, X_{s-}, Z_s^{(0)}, p^{(0)}_s, \eta^{(n)}_s) - \ubar{\hat H}(s, X_{s-}, Z_s^{(0)}, p^{(0)}_s, \eta^{(0)}_s))ds \\
=& \int_t^T\int_{\nu \in \mathcal{P}(A)}\hat H(s, X_{s-}, Z_s^{(0)}, p^{(0)}_s, \nu) (\eta^{(n)}_s - \eta^{(0)}_s)(d\nu)ds\\
=&\int_{[0,T]\times \mathcal{P}(A)}\kappa(s,\nu)\eta^{(n)}(ds, d\nu) - \int_{[0,T]\times \mathcal{P}(A)}\kappa(s,\nu)\eta^{(0)}(ds, d\nu),
\end{align*}
where we defined  $\kappa(s,\nu) := 1_{t\le s\le T}H(s, X_{s-}, Z_s^{(0)}, p^{(0)}_s, \nu)$. Clearly $\kappa$ is continuous in $\nu$ for all $s$. On the other hand, by inequality (\ref{eq:lipschitz_h_hat}) in Lemma \ref{lem:lip_h_hat}, for all $t\le s\le T$ and $\nu \in \mathcal{P}(A)$ we have:
$$
|H(s,X_{s-},Z_s^{(0)},p^{(0)}_s,\nu)|
\le  | H(s,X_{s-},0,0,0)| + C\|Z_s^{(0)}\|_{X_{s-}} + C(1 +  \|Z_s^{(0)}\|_{X_{s-}})(\|p^{(0)}_s\| + \mathcal{W}_1(\nu, 0)).
$$
Therefore by Lemma \ref{lem:boundedness_z} and the boundedness of $\mathcal{P}(A)$, we conclude that the mapping $(s,\nu) \rightarrow \kappa(s,\nu)$ is bounded. It follows from the definition of stable topology and $\eta^{(n)}\rightarrow\eta^{(0)}$ that:
\[
\lim_{n\rightarrow +\infty}\int_t^T(\ubar{\hat H}(s, X_{s-}, Z_s^{(0)}, p^{(0)}_s, \eta^{(n)}_s) - \ubar{\hat H}(s, X_{s-}, Z_s^{(0)}, p^{(0)}_s, \eta^{(0)}_s))ds = 0,
\]
for all $w\in\Omega$. In addition, we have:
\begin{align*}
&\hskip -45pt
\left(\int_t^T(\ubar{\hat H}(s, X_{s-}, Z_s^{(0)}, p^{(0)}_s, \eta^{(n)}_s) - \ubar{\hat H}(s, X_{s-}, Z_s^{(0)}, p^{(0)}_s, \eta^{(0)}_s))ds\right)^2\\
\le& (T-t) \int_t^T|\ubar{\hat H}(s, X_{s-}, Z_s^{(0)}, p^{(0)}_s, \eta^{(n)}_s) - \ubar{\hat H}(s, X_{s-}, Z_s^{(0)}, p^{(0)}_s, \eta^{(0)}_s)|^2 ds \\
\le &C \int_t^T (1 + \|Z_s^{(0)}\|_{X_{s-}})^2 (\bar{\mathcal{W}}_1(\eta^{(n)}_s, \eta^{(0)}_s))^2 ds \le C \int_t^T (1 + \|Z_s^{(0)}\|_{X_{s-}})^2 ds
\end{align*}
and $\int_t^T (1 + \|Z_s^{(0)}\|_{X_{s-}})^2 ds$ is integrable. Apply once again the dominated convergence theorem, we conclude that $K_n$ converges to $0$. This completes the proof.
\end{proof}

We will also need a result on a more convenient representation of the Dol\'eans-Dade exponential of $\bL^{(n)}$.

\begin{lemma}
\label{lem:dd_sde}
Denote by $\bW^{(n)}$ the Dol\'eans-Dade exponential of $\bL^{(n)}$. Then the It\^o differential of $\log(\bW^{(n)})$ satisfies:
$$
d[\log(W^{(n)}_t)] =X_{t-}^*\cdot (\ubar Q(t,\hat\alpha^{(n)}_t,p^{(n)}_t, \eta^{(n)}_t) - Q^0 + \ubar O(t,\hat\alpha^{(n)}_t,p^{(n)}_t, \eta^{(n)}_t)\cdot  Q^0)\cdot X_{t-} dt
+ X_{t-}^*\cdot  \ubar O(t,\hat\alpha^{(n)}_t,p^{(n)}_t, \eta^{(n)}_t)\cdot  d\mathcal{M}_t,
$$
where $\ubar O(t,\hat\alpha^{(n)}_t,p^{(n)}_t, \eta^{(n)}_t)$ is the matrix with $\log(\ubar q(t,i,j,\hat\alpha^{(n)}_t,p^{(n)}_t, \eta^{(n)}_t))$ as off-diagonal elements and zeros on the diagonal.
\end{lemma}

\begin{proof}
Since $\bW^{(n)}$ is the Dol\'eans-Dade exponential of $\bL^{(n)}$, $\bW^{(n)}$ satisfies the SDE $dW^{(n)}_t = W^{(n)}_{t-} dL^{(n)}_t$. Applying Ito's formula and noticing that the continuous martingale part of $\bL^n$ is zero, we have:
{\small
\[
d\log(W^{(n)}_t) = dL^{(n)}_t -\Delta L^{(n)}_t + \log(1 + \Delta L^{(n)}_t).
\]
}
Then using $dL^{(n)}_t = X_{t-}^*\cdot(\ubar Q(t,\hat\alpha^{(n)}_t,p^{(n)}_t, \eta^{(n)}_t) - Q^0)\cdot\psi^+_t\cdot d\mathcal{M}_t$ and noticing that the jumps of $\bL^n$ are driven by the jumps of $\mathcal{M}$, and hence $X$, we obtain:
\begin{align*}
d\log(W^{(n)}_t) =& -X_{t-}^*\cdot (\ubar Q(t,\hat\alpha^{(n)}_t,p^{(n)}_t, \eta^{(n)}_t) - Q^0)\cdot \psi^+\cdot Q^0\cdot X_{t-} dt +  \log(1 + \Delta L^{(n)}_t) \\
=& X_{t-}^*\cdot (\ubar Q(t,\hat\alpha^{(n)}_t,p^{(n)}_t, \eta^{(n)}_t) - Q^0)\cdot X_{t-} dt +  \log(1 + \Delta L^{(n)}_t),
\end{align*}
where we have used the fact that for all $q$-matrices $A$, we have $X^*_{t-}\cdot A\cdot  \psi^+\cdot Q^0\cdot X_{t-} = -X_{t-}^*\cdot A\cdot  X_{t-}$. Piggybacking on the derivation following equation (\ref{eq:jump_of_l}), for $X_{t-} = e_i$ and $X_t = e_j$ we have:
\[
\log(1 + \Delta L^{(n)}_t) = \log(\ubar q(t,i,j,\hat\alpha^{(n)}_t,p^{(n)}_t, \eta^{(n)}_t)).
\]
Using matrix notation and recalling the definition of $\ubar O$ in the statement of Lemma \ref{lem:dd_sde}, we may write:
\[
\log(1 + \Delta L^{(n)}_t) = X_{t-}^* \cdot \ubar O(t,\hat\alpha^{(n)}_t,p^{(n)}_t, \eta^{(n)}_t) \cdot \Delta X_t.
\]
Using again the equality $\Delta X_t = dX_t = Q^0 \cdot X_{t-} dt + d\mathcal{M}_t$, we arrive at the desired representation of the differential of $\log(W_t)$.
\end{proof}

We now show the first component of the mapping $\Phi$ is sequentially continuous.

\begin{proposition}\label{prop:phi_mu_continuous}
$\mathbb{Q}^{(n)}$ converges to $\mathbb{Q}^{(0)}$ in $\mathcal{P}$.
\end{proposition}
\begin{proof}
For two probability measures $\mathbb{Q}$ and $\mathbb{Q}'$ in $\mathcal{P}$, the total variation distance $d$ between $\mathbb{Q}$ and $\mathbb{Q}'$ is:
\begin{equation}\label{eq:total_variation_dist}
d_{TV}(\mathbb{Q}, \mathbb{Q}') := \sup\{ |\mathbb{Q}(A)- \mathbb{Q}'(A)|, A \in \mathcal{B}(D)\}.
\end{equation}
It is well-known that convergence in total variation implies weak convergence,  hence  convergence in the topological space $\mathcal{P}$. Therefore our aim is to show that $d_{TV}(\mathbb{Q}^{(n)}, \mathbb{Q}^{(0)}) \rightarrow 0 $ as $n\rightarrow +\infty$.
By Pinsker's inequality, we have:
\[
d_{TV}^2(\mathbb{Q}^{(0)}, \mathbb{Q}^{(n)}) \le \frac{1}{2} \mathbb{E}^{(0)}\left[\log\left(\frac{d\mathbb{Q}^{(0)}}{d\mathbb{Q}^{(n)}}\right)\right].
\]
Since $\frac{d\mathbb{Q}^{(n)}}{d\mathbb{P}} = \mathcal{E}(\bL^{(n)})_T$, we have:
\[
d_{TV}^2(\mathbb{Q}^{(0)}, \mathbb{Q}^{(n)}) \le \mathbb{E}^{(0)}[\log(\mathcal{E}(L^{(0)})_T) - \log(\mathcal{E}(L^{(n)})_T)].
\]
Using Lemma \ref{lem:dd_sde}, we have:
\begin{align*}
&\hskip -45pt
\mathbb{E}^{(0)}[\log(\mathcal{E}(\bL^{(0)})_T) - \log(\mathcal{E}(\bL^{(n)})_T)] \\
=&\mathbb{E}^{(0)}\left[\int_0^T X_{t-}^*\cdot (\ubar Q(t,\hat\alpha_t^{(0)},p^{(0)}_t, \eta^{(0)}_t) - \ubar Q(t,\hat\alpha_t^{(n)},p^{(n)}_t, \eta^{(n)}_t))\cdot X_{t-} dt\right]\\
 &+\mathbb{E}^{(0)}\left[\int_0^T X_{t-}^*\cdot (\ubar O(t,\hat\alpha_t^{(0)},p^{(0)}_t, \eta^{(0)}_t) - \ubar O(t,\hat\alpha_t^{(n)},p^{(n)}_t, \eta^{(n)}_t))\cdot Q^0\cdot X_{t-} dt\right] \\
&+ \mathbb{E}^{(0)}\left[\int_0^T  X_{t-}^* \cdot (\ubar O(t,\hat\alpha_t^{(0)},p^{(0)}_t, \eta^{(0)}_t) - \ubar O(t,\hat\alpha_t^{(n)},p^{(n)}_t, \eta^{(n)}_t)) \cdot  d\mathcal{M}_t\right].
\end{align*}
By Assumption \ref{hypo:boundedness}, the process $t \rightarrow \int_0^t  X_{s-}^* \cdot (O(s,\hat\alpha_s^{(0)},p^{(0)}_s, \nu^{(n)}_s) - O(s,\hat\alpha_s^{(n)},p^{(n)}_s, \nu^{(n)}_s)) \cdot  d\mathcal{M}_s$ is a true martingale therefore have zero expectation. We now deal with the convergence of the term $\mathbb{E}^{0}[\int_0^T X_{t-}^*\cdot (\ubar Q(t,\hat\alpha_t^{(n)},p^{(n)}_t, \eta^{(n)}_t) - \ubar Q(t,\hat\alpha_t^{(0)},p^{(0)}_t, \eta^{(0)}_t))\cdot X_{t-} dt]$, whereas the term $\mathbb{E}^{0}[\int_0^T X_{t-}^*\cdot (\ubar O(t,\hat\alpha_t^{(n)},p^{(n)}_t, \eta^{(n)}_t) - \ubar O(t,\hat\alpha_t^{(0)},p^{(0)}_t, \eta^{(0)}_t))\cdot Q^0\cdot X_{t-} dt]$ can be dealt with in the exact the same way. Using the Lipschitz property of $\hat a$ and $\ubar Q$ in Lemma \ref{lem:lip_h_hat} and Lemma \ref{lem:lip_h_hat_randomized}, we obtain:
\begin{align*}
&\hskip -45pt
\mathbb{E}^{(0)}\left[\int_0^T X_{t-}^*\cdot(\ubar Q(t,\hat\alpha_t^{(0)},p^{(0)}_t, \eta^{(0)}_t) - \ubar Q(t,\hat\alpha_t^{(n)},p^{(n)}_t, \eta^{(n)}_t))\cdot X_{t-} dt\right]\\
\le&\mathbb{E}^{(0)}\left[\int_0^T X_{t-}^*\cdot(\ubar Q(t,\hat\alpha_t^{(0)},p^{(0)}_t, \eta^{(n)}_t) - \ubar Q(t,\hat\alpha_t^{(n)},p^{(n)}_t, \eta^{(n)}_t))\cdot X_{t-} dt\right] \\
&+ \mathbb{E}^{(0)}\left[\int_0^T X_{t-}^*\cdot(\ubar Q(t,\hat\alpha_t^{(0)},p^{(0)}_t, \eta^{(0)}_t) - \ubar Q(t,\hat\alpha_t^{(0)},p^{(0)}_t, \eta^{(n)}_t))\cdot X_{t-} dt\right]\\
\le&\;\;\mathbb{E}^{(0)}\left[\int_0^T C( \|\hat\alpha^{(n)}_t - \hat\alpha^{(0)}_t\| + \|p^{(n)}_t - p^{(0)}_t\|)dt\right] \\
&+ \mathbb{E}^{(0)}\left[\int_0^T X_{t-}^*\cdot (\ubar Q(t,\hat\alpha^{(0)}_t,p^{(0)}_t, \eta^{(n)}_t) - \ubar Q(t,\hat\alpha^{(0)}_t,p^{(0)}_t, \eta^{(0)}_t))\cdot X_{t-} dt\right]\\
\le&\mathbb{E}^{(0)}\left[\int_0^T C\|Z^{(n)}_t - Z^{(0)}_t\|_{X_{t-}}dt\right] + \mathbb{E}^{(0)}\left[\int_0^T C(1 + \|Z^{(0)}_t\|_{X_{t-}})\|p^{(n)}_t - p^{(0)}_t\|dt\right]\\
& + \mathbb{E}^{(0)}\left[\int_0^T X_{t-}^*\cdot(\ubar Q(t,\hat\alpha_t^{(0)},p^{(0)}_t, \eta^{(n)}_t) - \ubar Q(t,\hat\alpha_t^{(0)},p^{(0)}_t, \eta^{(0)}_t))\cdot X_{t-} dt\right].
\end{align*}
We deal with these terms separately. For the first expectation, by Cauchy-Schwartz inequality, we have:
\begin{align*}
\left(\mathbb{E}^{(0)}\left[\int_0^T C \|Z^{(n)}_t - Z^{(0)}_t\|_{X_{t-}} dt \right]\right)^2 &= \left(\mathbb{E}\left[W^{(0)}_T \int_0^T C \|Z^{(n)}_t - Z^{(0)}_t\|_{X_{t-}} dt \right]\right)^2\\
&\le\mathbb{E}[(W^{(0)}_T)^2]\mathbb{E}\left[ \left(\int_0^T C \|Z^{(n)}_t - Z^{(0)}_t\|_{X_{t-}} dt\right)^2\right]\\
& \le C\mathbb{E}[(W^{(0)}_T)^2]\mathbb{E}\left[ \int_0^T \|Z^{(n)}_t - Z^{(0)}_t\|^2_{X_{t-}} dt\right].
\end{align*}
This converges to $0$ by Proposition \ref{prop:convergence_bsde_solution}. For the second expectation, we notice from Lemma \ref{lem:boundedness_z} that $\|Z^{(0)}_t\|_{X_{t-}}$ is bounded by a constant for all $(\omega,t) \in \Omega\times [0,T]$. Therefore we have:
\begin{align*}
\mathbb{E}^{(0)}\left[\int_0^T C(1 + \|Z^{(0)}_t\|_{X_{t-}})\|p^{(n)}_t - p^{(0)}_t\|dt\right] \le C \int_0^T C\|p^{(n)}_t - p^{(0)}_t\|dt,
\end{align*}
where the right-hand side converges to $0$ by dominated convergence theorem. Finally for the third expectation, we rewrite the integrand as:
\begin{align*}
&\hskip -75pt
\int_0^T X_{t-}^*\cdot(\ubar Q(t,\hat\alpha_t^{(0)},p^{(0)}_t, \eta^{(n)}_t) - \ubar Q(t,\hat\alpha_t^{(0)},p^{(0)}_t, \eta^{(0)}_t))\cdot X_{t-} dt \\
=&\;\; \int_{[0,T]\times\mathcal{P}(A)} X_{t-}^*\cdot Q(t,\hat\alpha_t^{(0)},p^{(0)}_t, \nu)\cdot X_{t-} (\eta^{(n)}(dt,d\nu) - \eta^{(0)}(dt,d\nu)).
\end{align*}
This converges to $0$, since $\eta^{(n)}$ converges to $\eta^{(0)}$ in stable topology and the mapping $\nu \rightarrow Q(t,\hat\alpha_t^{(0)},p^{(0)}_t, \nu)$ is continuous for all $t$. Notice also that the integrand is bounded by a constant, since $q$ is bounded according to Assumption \ref{hypo:boundedness}. Then by dominated converges theorem the thrid expectation converges to $0$ as well. This completes the proof.
\end{proof}

To show the continuity of $\Phi^\eta$, we need the following lemma.

\begin{lemma}
\label{lem:conv_in_stable_topo_r0}
Let $(\nu^{(n)}_t)_{t\le T}$ be a sequence of measurable functions from $[0,T]$ to $\mathcal{P}(A)$ such that $\int_0^T \mathcal{W}_1(\nu^{(n)}_t, \nu^{(0)}_t)\rightarrow 0$. Then $\mathcal{L}(dt)\times \delta_{\nu^{(n)}_t}(d\nu)$ converges to $\mathcal{L}(dt)\times \delta_{\nu^{(0)}_t}(d\nu)$ in $\mathcal{R}_0$ in the sense of the stable topology.
\end{lemma}

\begin{proof}
Set $\lambda^{(n)}(dt,d\nu) := \mathcal{L}(dt)\times \delta_{\nu^{(n)}_t}(d\nu)$ for $n\ge 0$, let $f : [0,T]\times\mathcal{P}(A)\rightarrow\mathbb{R}$ be a mapping of the form $f(t,\nu) = 1_{t \in B}\cdot g(\nu)$ where $B$ is measurable subset of $[0,T]$ and $g$is  a bounded Lipschitz function on $\mathcal{P}(A)$. We then have:
$$
\left|\int_{[0,T]\times\mathcal{P}(A)} f(t,\nu) \lambda^{(n)}(dt, d\nu) - \int_{[0,T]\times\mathcal{P}(A)} f(t,\nu) \lambda^{(0)}(dt, d\nu)\right|
 \le \int_{t \in B}|g(\nu^{(n)}_t) - g(\nu^{(0)}_t)| dt \le C \int_0^T \mathcal{W}_1(\nu^{(n)}_t, \nu^{(0)}_t)dt.
$$
By Lemma \ref{lem:characterization_conv_stable_topo}, we conclude that $\lambda^{(n)}$ converges to $\lambda^{(0)}$ for the stable topology.
\end{proof}

\begin{proposition}
\label{prop:phi_eta_continuous}
$\mathcal{L}(\cdot)\times\delta_{\mathbb{Q}^{(n)}_{\#\hat\alpha^{(n)}_t}}(\cdot)$ converges to $\mathcal{L}(\cdot)\times\delta_{\mathbb{Q}^{(0)}_{\#\hat\alpha^{(0)}_t}}(\cdot)$ in $\mathcal{R}_0$ in the sense of the stable topology.
\end{proposition}

\begin{proof}
By Lemma \ref{lem:conv_in_stable_topo_r0}, we only need to show $\int_0^T \mathcal{W}_1(\mathbb{Q}^{(n)}_{\#\hat\alpha^{(n)}_t}, \mathbb{Q}^{(0)}_{\#\hat\alpha^{(0)}_t})dt$ converges to $0$. Notice that:
\[
\int_0^T \mathcal{W}_1(\mathbb{Q}^{(n)}_{\#\hat\alpha^{(n)}_t}, \mathbb{Q}^{(0)}_{\#\hat\alpha^{(0)}_t}) dt \le \int_0^T \mathcal{W}_1(\mathbb{Q}^{(n)}_{\#\hat\alpha^{(n)}_t}, \mathbb{Q}^{(0)}_{\#\hat\alpha^{(n)}_t})dt + \int_0^T \mathcal{W}_1(\mathbb{Q}^{(0)}_{\#\hat\alpha^{(n)}_t}, \mathbb{Q}^{(0)}_{\#\hat\alpha^{(0)}_t})dt.
\]
By the very definition of the total variation distance (recall equation \ref{eq:total_variation_dist}), we have clearly:
\[
d_{TV}(\mathbb{Q}^{(n)}_{\#\hat\alpha^{(n)}_t}, \mathbb{Q}^{(0)}_{\#\hat\alpha^{(n)}_t}) \le d_{TV}(\mathbb{Q}^{(n)}, \mathbb{Q}^{(0)}),
\]
which converges to $0$ according to the proof of Proposition \ref{prop:phi_mu_continuous}. By Theorem 6.16 in \cite{villani2008}, since $A$ is bounded and $\mathbb{Q}^{(n)}_{\#\hat\alpha^{(n)}_t} \in \mathcal{P}(A)$, there exists a constant $C$ such that:
\[
\mathcal{W}_1(\mathbb{Q}^{(n)}_{\#\hat\alpha^{(n)}_t}, \mathbb{Q}^{(0)}_{\#\hat\alpha^{(n)}_t}) \le C \cdot d_{TV}(\mathbb{Q}^{(n)}_{\#\hat\alpha^{(n)}_t}, \mathbb{Q}^{(0)}_{\#\hat\alpha^{(n)}_t}).
\]
This shows that $\mathcal{W}_1(\mathbb{Q}^{(n)}_{\#\hat\alpha^{(n)}_t}, \mathbb{Q}^{(0)}_{\#\hat\alpha^{(n)}_t})$ converges to $0$. In addition, it is also bounded since $A$ is bounded. The dominated convergence theorem then implies that:
\[
\int_0^T \mathcal{W}_1(\mathbb{Q}^{(n)}_{\#\hat\alpha^{(n)}_t}, \mathbb{Q}^{(0)}_{\#\hat\alpha^{(n)}_t})dt \rightarrow 0, n\rightarrow +\infty.
\]
Now for the other term, we have:
\begin{align*}
int_0^T \mathcal{W}_1(\mathbb{Q}^{(0)}_{\#\hat\alpha^{(n)}_t}, \mathbb{Q}^{(0)}_{\#\hat\alpha^{(0)}_t})dt &\le \int_0^T\mathbb{E}^{(0)}[\|\hat\alpha^{(n)}_t - \hat\alpha^{(0)}_t\|]dt\\
& = \mathbb{E}\left[W^{(0)}_T\int_0^T\|\hat\alpha^{(n)}_t - \hat\alpha^{(0)}_t\| dt\right]\\
& \le (\mathbb{E}[(W^{(0)}_T)^2])^{1/2}\left(\mathbb{E}\left[T\int_0^T\|\hat\alpha^{(n)}_t - \hat\alpha^{(0)}_t\|^2 dt\right]\right)^{1/2}.
\end{align*}
The Lipschitz property of $\hat a$ (see Lemma \ref{lem:lip_h_hat}) and Proposition \ref{prop:convergence_bsde_solution} imply that $\mathbb{E}\left[T\int_0^T\|\hat\alpha^{(n)}_t - \hat\alpha^{(0)}_t\|^2 dt\right]\to 0$.
\end{proof}
We are now ready to show the existence of  Nash equilibria.

\begin{proof}(of Theorem \ref{theo:existence_nash_equilibrium})
Consider the product space $\Gamma:=\mathcal{P}\times\mathcal{R}$ endowed with the product topology of the weak topology on $\mathcal{P}$ and the stable topology on $\mathcal{R}$. By Proposition \ref{prop:stable_topology_polish}, $\Gamma$ is a Polish space. By Proposition \ref{prop:topo_law_of_state} and Lemma \ref{lem:stable_topology_subset}, $\Gamma_0 := \bar{\mathcal{P}}_0\times\mathcal{R}_0$ is a compact and convex subset of $\Gamma$ and it is stable by the mapping $\Phi$ defined in (\ref{eq:def_mapping_fixed_point}). In addition, we see from Proposition \ref{prop:phi_mu_continuous} and Proposition \ref{prop:phi_eta_continuous} that $\Phi$ is continuous. Therefore applying Schauder's fixed point theorem, we conclude that $\Phi$ admits a fixed point $(\mu^*,\eta^*) \in \bar{\mathcal{P}}_0\times\mathcal{R}_0$.

\vspace{3mm}
Now let us define $p_t^* := \pi(t,\mu^*) \in \mathcal{S}$ and  $\alpha^*_t := \hat a(t, X_{t-}, Z_t^*, \pi(t,\mu^*))$ where $(\bY^*, \bZ^*)$ is the solution to the BSDE (\ref{eq:mapping_bsde}) with $\mu = \mu^*$ and $\eta = \eta^*$. We then define $\mathbb{P}^* := \mathbb{P}^{\mu^*, \eta^*}$ and $\nu^*_t := \mathbb{P}^*_{\#\alpha^*_t}$. Since $(\mu^*,\eta^*)$ is the fixed point of the mapping $\Phi$, we have $\eta^*_t = \delta_{\nu^*_t}$ and $\mu^* = \mathbb{P}^*$. It follows that $p_t^* = \pi(t,\mathbb{P}^*) = [\mathbb{P}^*(X_t = e_i)]_{1\le i \le m}$. By Proposition \ref{prop:bsed_optimal_control}, we see that $\balpha^*$ is the solution to the optimal control problem (\ref{eq:optimization_pb}) when the mean field of state is $\bp^*$ and the mean field of control is $\bnu^*$. This implies that $(\balpha^*, \bp^*, \bnu^*)$ is a Nash equilibrium.
\end{proof}

\section{Uniqueness of Nash equilibrium}
Uniqueness of Nash equilibria will be proven under the following conditions.

\begin{hypothesis}
\label{hypo:uniqueness}
(i) The transition rate function $q$ does depend neither on the mean field of state $\bp$ nor on the mean field of control $\bnu$. The cost functional $f$ is separable in the sense that it is of the form:
\begin{equation}\label{eq:f_unique_form}
f(t,x,\alpha,p,\nu) = f_0(t, x,\alpha) + f_1(t, X, p) + f_2(t,p,\nu).
\end{equation}
(ii) For all $t\in[0,T]$, $i\in\{1,\dots,m\}$, $z\in\mathbb{R}^m$, $p\in\mathcal{S}$ and $\nu\in\mathcal{P}(A)$, the mapping $\alpha \rightarrow H_i(t,z,\alpha,p,\nu)$ admits a unique minimizer, whichbecause of assumption (i), only depends on $t$ and $z$. We denote it by $\hat a_i(t,z)$. In addition, we assume that $\hat a_i$ is a measurable  from $[0,T]\times \mathbb{R}^m$ into $A$, and that there exists a constant $C>0$ such that for all $i\in\{1,\dots,m\}$, and $z, z'\in\mathbb{R}^m$:
\begin{equation}\label{eq:optimizer_lipschitz1}
\|\hat a_i(t,z) - \hat a_i(t,z')\| \le C\|z - z'\|_{e_i}.
\end{equation}
(iii) For all $p, p' \in \mathcal{S}$ and $t\in[0,T]$, we have:
\begin{align}
\sum_{i=1}^m (g(e_i, p) - g(e_i, p'))(p_i - p'_i) \ge 0,\label{eq:monotonocity1}\\
\sum_{i=1}^m (f_1(t,e_i, p) - f_1(t,e_i, p'))(p_i - p'_i) \ge 0.\label{eq:monotonocity2}
\end{align}
\end{hypothesis}

\begin{remark}
Item (ii) of Assumption \ref{hypo:uniqueness} holds if we impose additional conditions of linearity and strong convexity on the transition rate function and the cost function, for example:
\end{remark}

\begin{hypothesis}\label{hypo:uniqueness_add}
\noindent (i)The transition rate function $q$ takes the form $q(t,i,j,\alpha) = q_0(t, i, j) + q_1(t,i,j)\cdot\alpha$, where the mappings $q_0:[0,T]\times E^2 \rightarrow \mathbb{R}$ and $q_1:[0,T]\times E^2 \rightarrow \mathbb{R}^l$ are continuous.

\noindent (ii) $f_0$ is $\gamma-$strongly convex in $\alpha$, i.e., for all $(t,i)\in[0,T]\times E$ and $\alpha, \alpha' \in A$, we have:
\begin{equation}\label{eq:strong_convex_f0}
f_0(t,e_i,\alpha) - f_0(t,e_i,\alpha') - (\alpha - \alpha') \cdot \nabla_{\alpha} f_0(t, e_i, \alpha) \ge \gamma \|\alpha' - \alpha\|^2
\end{equation}
\end{hypothesis}

\begin{theorem}
\label{theo:uniqueness_nash_equilibrium}
Under Assumptions \ref{hypo:boundedness}, \ref{hypo:lipschitz_cost} and \ref{hypo:uniqueness}, there exists at most one Nash equilibrium for the weak formulation of the finite state mean field game.
\end{theorem}

\begin{proof}
Let $(\balpha^{(1)}, \bp^{(1)}, \bnu^{(1)})$ and $(\balpha^{(2)}, \bp^{(2)}, \bnu^{(2)})$ be two Nash equilibria of the mean field game. For $i = 1,2$, we denote by $(\bY^{(i)}, \bZ^{(i)})$ the solution to the BSDE (\ref{eq:bsde_optimality}) with $\bp = \bp^{(i)}$, $\bnu = \bnu^{(i)}$, which is written as:
\begin{align*}
Y_0^{(i)} = g(X_T, p_T^{(i)}) + \int_0^T \hat H(t, X_{t-}, Z_t^{(i)}, p_t^{(i)}, \nu_t^{(i)}) dt - \int_0^T (Z_t^{(i)})^*\cdot d\mathcal{M}_t
\end{align*}
we have $\hat\alpha^{(i)}_t := \hat a(t,Z^{(i)}_t)$, $d\mathbb{P}\otimes dt$-a.e. Let us denote $\mathbb{Q}^{(i)} := \mathbb{Q}^{(\hat\balpha^{(i)}, \bp^{(i)}, \bnu^{(i)})}$, the controlled probability measure defined in (\ref{eq:q_measure}), under which $\mathcal{M}^{(i)} := \mathcal{M}^{(\hat\balpha^{(i)}, \bp^{(i)}, \bnu^{(i)})}$ is a martingale. In addition, we use the abbreviation $f_t^{(i)} := f(t, X_{t-}, \hat\alpha^{(i)}_t, p_t^{(i)}, \nu_t^{(i)})$, $g^{(i)} = g(X_T, p_T^{(i)})$ and $Q_t^{(i)} := Q(t, \hat\alpha_t^{(i)})$. Taking the difference of the BSDEs we obtain:
\begin{align*}
Y_0^{(1)} - Y_0^{(2)} 
=&g^{(1)} -  g^{(2)} +\int_0^T (f_t^{(1)} - f_t^{(2)} + X^*_{t-} \cdot (Q_t^{(1)} - Q^0) \cdot  Z_t^{(1)} - X^*_{t-} \cdot (Q_t^{(2)} - Q^0) \cdot Z_t^{(2)}) dt\\
& \hskip 45pt
+ \int_0^T (Z_t^{(1)} - Z_t^{(2)})^* \cdot d\mathcal{M}_t\\
=& g^{(1)} -  g^{(2)} + \int_0^T (f_t^{(1)} - f_t^{(2)} + X^*_{t-} \cdot (Q_t^{(1)} - Q_t^{(2)}) \cdot Z_t^{(1)}) dt + \int_0^T (Z_t^{(1)} - Z_t^{(2)})^* \cdot d\mathcal{M}^{(2)}_t\\
=& g^{(1)} -  g^{(2)} +\int_0^T (f_t^{(1)} - f_t^{(2)} + X^*_{t-} \cdot (Q_t^{(1)} - Q_t^{(2)}) \cdot Z_t^{(2)}) dt + \int_0^T (Z_t^{(1)} - Z_t^{(2)})^* \cdot d\mathcal{M}^{(1)}_t.
\end{align*}
Taking expectations with respect to $\mathbb{Q}^{(1)}$ and $\mathbb{Q}^{(2)}$ and using the fact that 
$$
\mathbb{E}^{\mathbb{P}}[Y_0^{(1)} - Y_0^{(2)}] = \mathbb{E}^{\mathbb{Q}^{(1)}}[Y_0^{(1)} - Y_0^{(2)}] = \mathbb{E}^{\mathbb{Q}^{(2)}}[Y_0^{(1)} - Y_0^{(2)}], 
$$
we obtain the following equality:
\begin{equation}
\label{eq:proof_uniqueness_1}
\begin{aligned}
&\mathbb{E}^{\mathbb{Q}^{(1)}}\left[g^{(1)} -  g^{(2)} + \int_0^T (f_t^{(1)} - f_t^{(2)} + X^*_{t-} \cdot (Q_t^{(1)} - Q_t^{(2)}) \cdot Z_t^{(2)}) dt\right] \\
&\hskip 45pt
= \mathbb{E}^{\mathbb{Q}^{(2)}}\left[g^{(1)} -  g^{(2)} + \int_0^T (f_t^{(1)} - f_t^{(2)} + X^*_{t-} \cdot (Q_t^{(1)} - Q_t^{(2)}) \cdot Z_t^{(1)}) dt\right].
\end{aligned}
\end{equation}
Next we notice that:
\begin{equation*}
\begin{aligned}
f_t^1 + X^*_{t-} \cdot Q_t^{(1)}\cdot Z_t^2
&= f(t, X_{t-}, \hat\alpha_t^{(1)}, p_t^{(1)}, \nu_t^{(1)}) + X_{t-}^* \cdot Q(t, \hat\alpha_t^{(1)}) \cdot Z_t^{(2)} \\
&=H(t, X_{t-}, \hat\alpha_t^{(1)}, Z_t^{(2)}, p_t^{(1)}, \nu_t^{(1)})\ge H(t, X_{t-}, \hat\alpha_t^{(2)}, Z_t^{(2)}, p_t^{(1)}, \nu_t^{(1)})\\
&= H(t, X_{t-}, \hat\alpha_t^{(2)}, Z_t^{(2)}, p_t^{(2)}, \nu_t^{(2)}) + (f_1(t, X_{t-}, p_t^{(1)}) - f_1(t,  X_{t-}, p_t^{(2)}))\\
&\hskip 45pt
+ (f_2(t, p_t^{(1)}, \nu_t^{(1)}) - f_2(t,  p_t^{(2)}, \nu_t^{(2)}))
\end{aligned}
\end{equation*}
and using the inequality:
\[
H(t, X_{t-}, \hat\alpha_t^{(1)}, Z_t^{(2)}, p_t^{(1)}, \nu_t^{(1)}) \ge H(t, X_{t-},\hat \alpha_t^{(2)}, Z_t^{(2)}, p_t^{(1)}, \nu_t^{(1)})
\]
which is due to the fact that $\hat\alpha_t^{(2)}$ minimizes the Hamiltonian $\alpha \rightarrow H(t, X_{t-}, \alpha, Z_t^{(2)}, p_t^{(1)}, \nu_t^{(1)})$ and Assumption \ref{hypo:uniqueness} that the minimizer does not depend on the mean field terms, we get:
\[
f_t^{(1)} - f_t^{(2)} + X^*_{t-} (Q_t^{(1)} - Q_t^{(2)}) Z_t^{(2)} \ge (f_1(t, X_{t-}, p_t^{(1)}) - f_1(t,  X_{t-}, p_t^{(2)})) + (f_2(t, p_t^{(1)}, \nu_t^{(1)}) - f_2(t,  p_t^{(2)}, \nu_t^{(2)})).
\]
Interchanging the indices we obtain:
\[
f_t^{(2)} - f_t^{(1)} + X^*_{t-} \cdot (Q_t^{(2)} - Q_t^{(1)}) \cdot Z_t^{(1)} \ge (f_1(t, X_{t-}, p_t^{(2)}) - f_1(t,  X_{t-}, p_t^{(1)})) + (f_2(t, p_t^{(2)}, \nu_t^{(2)}) - f_2(t,  p_t^{(1)}, \nu_t^{(1)})).
\]
Injecting these inequalities into equation (\ref{eq:proof_uniqueness_1}) we have:
\begin{align*}
0 = &\;\mathbb{E}^{\mathbb{Q}^{(1)}}\left[g^{(1)} -  g^{(2)} + \int_0^T (f_t^{(1)} - f_t^{(2)} + X^*_{t-} \cdot (Q_t^{(1)} - Q_t^{(2)}) \cdot Z_t^{(2)}) dt\right]\\
&\; - \mathbb{E}^{\mathbb{Q}^{(2)}}\left[g^{(1)} -  g^{(2)} +\int_0^T (f_t^{(1)} - f_t^{(2)} + X^*_{t-} \cdot (Q_t^{(1)} - Q_t^{(2)}) \cdot Z_t^{(1)}) dt\right] \\
\ge &\;\mathbb{E}^{\mathbb{Q}^{(1)}}\left[g^{(1)} -  g^{(2)}+ \int_0^T (f_1(t, X_{t-}, p_t^{(1)}) - f_1(t,  X_{t-}, p_t^{(2)}) + f_2(t, p_t^{(1)}, \nu_t^{(1)}) - f_2(t,  p_t^{(2)}, \nu_t^{(2)})) dt\right]\\
&\;-\mathbb{E}^{\mathbb{Q}^{(2)}}\left[g^{(1)} -  g^{(2)} + \int_0^T (f_1(t, X_{t-}, p_t^{(1)}) - f_1(t,  X_{t-}, p_t^{(2)}) + f_2(t, p_t^{(1)}, \nu_t^{(1)}) - f_2(t,  p_t^{(2)}, \nu_t^{(2)})) dt\right]\\
=&\;\mathbb{E}^{\mathbb{Q}^{(1)}}\left[g^{(1)} -  g^{(2)} +\int_0^T (f_1(t, X_{t-}, p_t^{(1)}) - f_1(t,  X_{t-}, p_t^{(2)}))dt\right]\\
&\; - \mathbb{E}^{\mathbb{Q}^{(2)}}\left[g^{(1)} -  g^{(2)} +\int_0^T (f_1(t, X_{t-}, p_t^{(1)}) - f_1(t,  X_{t-}, p_t^{(2)}))dt\right]
\end{align*}
where the last equality is due to the fact that $[f_2(t, p_t^{(2)}, \nu_t^{(2)}) -  f_2(t, p_t^{(1)}, \nu_t^{(1)})]$ is deterministic. 

From Proposition \ref{prop:bsed_optimal_control}, since $\balpha^{(i)}$ is the optimal control with regard to the mean field $\bp^{(i)}$ and $\bnu^{(i)}$, we have $\alpha^{(i)}_t = \hat\alpha^{(i)}_t$, $dt \otimes d\mathbb{P}$-a.e. This implies that $\mathbb{Q}^{(i)}[X_{t-} = e_k] = \mathbb{Q}^{(\balpha^{(i)}, \bp^{(i)}, \bnu^{(i)})}[X_{t-} = e_k]$ for all $i=1,2$ and $k = 1,\dots,m$. Since $(\balpha^{(i)}, \bp^{(i)}, \bnu^{(i)})$ is a Nash equilibrium, we have $ \mathbb{Q}^{(\balpha^{(i)}, \bp^{(i)}, \bnu^{(i)})}[X_{t-} = e_k] = [p_t^{(i)}]_k$. Therefore we obtain $\mathbb{Q}^{(i)}[X_{t-} = e_k] = [p_t^{(i)}]_k$ for all $i=1,2$ and $k = 1,\dots,m$. Now using item (iii) of Assumption \ref{hypo:uniqueness}, we have:
\begin{equation}\label{eq:proof_uniqueness_3}
\begin{array}{ll}
0 \ge&\displaystyle\sum_{i=1}^m (g(e_i, p_T^{(1)}) - g(e_i, p_T^{(2)}))([p_T^{(1)}]_i - [p_T^{(2)}]_i)\\
&\displaystyle+ \int_0^T \sum_{i=1}^m (f_0(t,e_i, p_t^{(1)}) - f_0(t,e_i, p_t^{(2)}))([p_t^{(1)}]_i - [p_t^{(2)}]_i) dt \ge 0
\end{array}
\end{equation}
Assume that there exists a measurable subset $N$ of $[0,T]\times\Omega$ with strictly positve $dt\otimes d\mathbb{Q}^{(1)}$ measure, such that $\hat\alpha_t^{(1)} \neq \hat\alpha_t^{(2)}$ on $N$. By Assumption \ref{hypo:uniqueness}, the mapping $\alpha \rightarrow H(t, X_{t-},\alpha, Z_t^{(2)}, p_t^{(1)}, \nu_t^{(1)})$ admits a unique minimizer and therefore for all $(t,w) \in N$, we have:
\[
H(t, X_{t-}, \hat\alpha_t^{(1)}, Z_t^{(2)}, p_t^{(1)}, \nu_t^{(1)}) > H(t, X_{t-}, \hat\alpha_t^{(2)}, Z_t^{(2)}, p_t^{(1)}, \nu_t^{(1)})
\]
Piggybacking on the argument laid out above, we see that the first inequality is strict in (\ref{eq:proof_uniqueness_3}) which leads to a contradiction. Therefore we have $\hat\alpha_t^{(1)} = \hat\alpha_t^{(2)}$, $dt\otimes d\mathbb{Q}^1$-a.e., and $dt\otimes d\mathbb{P}$-a.e., since $\mathbb{P}$ is equivalent to $\mathbb{Q}^{(1)}$. It follows that $\alpha_t^{(1)} = \alpha_t^{(2)}$, $dt\otimes d\mathbb{P}$-a.e. Finally, using the same type of argument as in the proof of Proposition \ref{prop:phi_mu_continuous}, we obtain $\mathbb{Q}^{(1)} = \mathbb{Q}^{(2)}$ which finally leads to $(\bp^{(1)}, \bnu^{(1)}) = (\bp^{(2)}, \bnu^{(2)})$. This completes the proof of the uniqueness.
\end{proof}

\section{Approximate Nash Equilibrium for Games with Finite Many Players}
\label{sec:approx_nash}
In this section we show that the solution of a mean field game can be used to construct approximate Nash equilibria for games with finitely many players. We first set the stage for the weak formulation of the game with $N$ players in finite state spaces. Recall that $\Omega$ is the space of c\`adl\'ag mappings from $[0,T]$ to $E = \{e_1,\dots,e_M\}$ which are continuous on $T$, $t\rightarrow X_t$ is the canonical process and $\mathbb{F} := (\mathcal{F}_t)_{t\le T}$ is the natural filtration generated by $\bX$. Let us fix $p^{\circ} \in \mathcal{S}$ a probability distribution on the state space $E$. Let $\mathbb{P}$ be the probability on $(\Omega, \mathcal{F}_T)$ under which $\bX$ is a continuous-time Markov chain with transition rate matrix $Q^0$ and initial distribution $p^{\circ}$. Let $\Omega^{N}$ be the product space of $N$ copies of $\Omega$, and ${\mathbb{P}}^N$ be the product probability measure of $N$ identical copies of $\mathbb{P}$. For $n=1,\dots,N$, define the process $X_t^n (w) := w^n_t$ of which the natural filtration is denoted by $\mathbb{F}^{n,N} := (\mathcal{F}^{n,N}_t)_{t\in[0,T]}$. We also denote by $\mathbb{F}^{N} :=(\mathcal{F}^N_t)_{t\in[0,T]}$ the natural filtration generated by the process $(\bX^{1}, \bX^{2},\dots, \bX^N)$. Denote $\mathcal{M}_t^n := X_t^n - X_t^0 - \int_0^t Q^0\cdot X_{s-}^n ds$. It is clear that under ${\mathbb{P}}^N$, $\bX^1, \dots, \bX^N$ are $N$ independent continuous-time Markov chains with initial distribution $\mathbf{p}^{\circ}$ and $Q^0$ as the transition rate matrix, and $\mathcal{M}^1,\dots, \mathcal{M}^N$ are independent $\mathcal{F}^N$-martingales. For later use, for $i=1,\dots,N$, we define the matrix $\psi_t^n$ by $\psi_t^n := diag(Q^0 \cdot X^n_{t-}) - Q^0 \cdot diag( X^n_{t-}) - diag( X^n_{t-}) \cdot Q^0$.

\vspace{3mm}
Throughout this section, we let Assumptions \ref{hypo:boundedness}, \ref{hypo:lipschitz_cost} and \ref{hypo:lipschitz_optimizer} hold. In addition, we adopt the following assumption:
\begin{hypothesis}\label{hypo:propagation_chaos}
The transition rate function $q$ does not depend on the mean field of state, nor the mean field of control.
\end{hypothesis}

We assume that each player can observe the entire past history of every player's state. We denote by $\mathbb{A}^N$ the collection of $\mathbb{F}^N$-predictable processes taking values in $A$. Each player $n$ chooses a strategy $\balpha^n\in\mathbb{A}^N$. We define the martingale $L^{(\balpha^1, \dots,\balpha^N)}$ by:
\begin{equation}\label{eq:martingale_L_nplayer}
L^{(\balpha^1, \dots,\balpha^N)}_t := \int_0^t \sum_{n=1}^N(X^n_{s^-})^* \cdot(Q(s, \alpha^n_s) - Q^0)\cdot(\psi^n_s)^+ \cdot d\mathcal{M}^n_s,
\end{equation}
and the probability measure $\mathbb{Q}^{(\balpha^1, \dots,\balpha^N)}$ by:
\begin{equation}\label{eq:q_measure_nplayer}
\frac{d\mathbb{Q}^{(\balpha^1, \dots,\balpha^N)}}{d\mathbb{P}^N} = \mathcal{E}^{(\balpha^1, \dots,\balpha^N)}_T,
\end{equation}
where we denote by  $\mathcal{E}^{(\balpha^1, \dots,\balpha^N)}$ the Dol\'eans-Dade exponential of $L^{(\balpha^1, \dots,\balpha^N)}$. Finally we introduce the empirical distribution of the states:
\begin{equation}\label{eq:emp_dist_states}
p_t^N := \frac{1}{N}\left[\sum_{n=1}^N \mathbbm{1}(X_t^n = e_1), \sum_{n=1}^N \mathbbm{1}(X_t^n = e_2), \dots, \sum_{n=1}^N \mathbbm{1}(X_t^n = e_{m})\right] \in \mathcal{S},
\end{equation}
as well as the empirical distribution of the controls:
\begin{equation}\label{eq:emp_dist_controls}
\nu(\alpha^1_t, \dots, \alpha^N_t) := \frac{1}{N} \sum_{n=1}^N \delta_{\alpha^n_t} \in \mathcal{P}(A),
\end{equation}
where $\delta_a(\cdot)$ is the Dirac measure on $a$. The total expected cost of player $n$ in the game with $N$ players, denoted by $J^{n,N}(\alpha^1,\dots,\alpha^N)$, is defined as:
\begin{equation}\label{eq:cost_nplayer}
J^{n,N}(\balpha^1,\dots,\balpha^N) 
:= \mathbb{E}^{\mathbb{Q}^{(\balpha^1, \dots,\balpha^N)}}\left[\int_0^T f(t, X_t^n, \alpha^n_t, p_t^N, \nu(\alpha^1_t, \dots, \alpha^N_t)) dt + g(X_T^n, p_T^N)\right].
\end{equation}
Now let us consider a Nash equilibrium $(\balpha^*, \bp^*, \bnu^*)$ of the mean field game  in the sense of Definition \ref{def:equilibrium}. Recall that $\\balpha^*$ is a predictable process with respect to the natural filtration generated by the canonical process $\bX$. For each $n = 1,\dots, N$, we may define the control $\hat\balpha^n$ of player $n$ by:
\begin{equation}
\hat\balpha^n(w^1, \dots, w^N) := \balpha^*(w^n).
\end{equation}
Clearly, $\hat\balpha^n$ is $\mathcal{F}^{n,N}$-predictable. In other words, it only depends on the observation of player $n$'s own path. Therefore the strategy profile $\hat{\boldsymbol{\alpha}}^{(N)} := (\hat\balpha^1, \dots, \hat\balpha^N)$ is a distributed strategy profile, which means that every player's strategy is only based on the observation of its own path.  

In the following, we will show that $\hat{\boldsymbol{\alpha}}^{(N)}$ is an approximate Nash equilibrium in a sense to be made clear later on. To this end, we first give a result on the propagation of chaos, which compare players $n$'s total expected cost in the mean field game versus its total expected cost in the finite player game. To simplify the notations, we use the abbreviation $(\bbeta, \hat\balpha^{-n, N})$ for $(\hat\balpha^1, \dots, \hat\balpha^{n-1}, \bbeta, \hat\balpha^{n+1}, \dots, \hat\balpha^N)$, $\hat{\mathbb{Q}}^{(N)}$ for $\mathbb{Q}^{\hat{\boldsymbol{\alpha}}^{(N)}} $,  $\hat{\mathbb{E}}^{(N)} $for $\mathbb{E}^{\hat{\mathbb{Q}}^{(N)}}$, and finally $\hat{\mathcal{E}}^{(N)}$ for $\mathcal{E}^{\hat{\boldsymbol{\alpha}}^{(N)}}$. We start from the following lemmas:

\begin{lemma}
\label{lem:propagation_of_chaos1}
There exists a sequence $(\delta_N)_{N\ge 0}$ such that $\delta_N \rightarrow 0$ as $ N\rightarrow +\infty$, and such that for all $N\ge 1$, $n \le N$ and $t\le T$ we have:
\begin{equation}\label{eq:propagation_of_chaos_lemma}
\max\big\{\hat{\mathbb{E}}^{(N)}[\mathcal{W}_1^2(\nu(\beta_t, \hat\alpha_t^{-1,N}), \nu^*_t)],\;\;\hat{\mathbb{E}}^N[\|p^N_t - p^*_t\|^2]\big\} \le \delta_N.
\end{equation}
\end{lemma}
\begin{proof}
Since $\hat{\mathbb{Q}}^{(N)} = \mathbb{Q}^{\hat{\boldsymbol{\alpha}}^{(N)}}$ and the fact that $(\balpha^*,\bp^*,\bnu^*)$ is an equilibrium of the mean field game, we deduce that under the measure $\hat{\mathbb{Q}}^{(N)}$, the states  $X^1_t,\dots, X^N_t$ are independent and have the same distribution characterized by $p^*_t$, and that the controls $\alpha^1_t, \dots, \alpha^N_t$ are independent and have the same distribution $\nu^*_t$. Therefore, for $i \in \{1,\dots,M\}$, we have:
$$
\hat{\mathbb{E}}^{(N)}\left[\left( \frac{1}{N}\sum_{n=1}^N \mathbbm{1}(X_t^n = e_i) - \hat{\mathbb{Q}}^N[X_t^1 = e_i]\right)^2\right]
=\frac{1}{N} (\hat{\mathbb{Q}}^{(N)}[X_t^1 = e_i] -  (\hat{\mathbb{Q}}^{(N)}[X_t^1 = e_i])^2) \le \frac{1}{4N},
$$
which leads to:
$$
\hat{\mathbb{E}}^{(N)}[\|p^N_t - p^*_t\|^2] = \sum_{i=1}^{M} \hat{\mathbb{E}}^{(N)}\left[\left( \frac{1}{N}\sum_{n=1}^N \mathbbm{1}(X_t^n = e_i) - \hat{\mathbb{Q}}^N[X_t^1 = e_i]\right)^2\right]\le\frac{M}{4N}.
$$
On the other hand, $\nu(\beta_t, \hat\alpha_t^{-1,N})$ and $\nu^*_t$ are in $\mathcal{P}(A)$ with $A$ being a compact subset of $\mathbb{R}^d$. We have:
\begin{align*}
\hat{\mathbb{E}}^{(N)}[\mathcal{W}_1^2(\nu(\beta_t, \hat\alpha_t^{-1,N}), \nu^*_t)]
& \le C \hat{\mathbb{E}}^{(N)}[\mathcal{W}_1(\nu(\beta_t, \hat\alpha_t^{-1,N}), \nu^*_t)]\\
& \le C \hat{\mathbb{E}}^{(N)}[\mathcal{W}_1(\nu(\beta_t, \hat\alpha_t^{-1,N}), \nu(\hat{\alpha}_t^{(N)})) + \mathcal{W}_1(\nu(\hat{\alpha}_t^{(N)}), \nu^*_t)]\\
 &\le C(\hat{\mathbb{E}}^{(N)}[\frac{1}{N} \|\beta_t - \hat\alpha_t^{1,N}\|] + \hat{\mathbb{E}}^{(N)}[\mathcal{W}_1(\nu(\hat{\alpha}_t^{(N)}), \nu^*_t)])\\
 & \le C(\frac{1}{N} + \hat{\mathbb{E}}^{(N)}[\mathcal{W}_1(\nu(\hat{\alpha}_t^{(N)}), \nu^*_t)]),
\end{align*}
where $C$ is a constant only depending on $\sup_{a\in A} \|a\|$ which changes its value from line to line. Now applying Theorem 1 in \cite{fournier2015}, we have:
$$
\hat{\mathbb{E}}^{(N)}[\mathcal{W}_1(\nu(\hat{\alpha}_t^{(N)}), \nu^*_t)]
 \le\;\; \sup_{a\in A} \|a\| \cdot [\mathbbm{1}(d\le 2) (N^{-1/2} \log(1+N) + N^{-2/3}) + \mathbbm{1}(d>2) (N^{-1/d} + N^{-1/2})].
$$
Combining with the estimates previously shown, we obtain the desired result.
\end{proof}

\begin{lemma}
\label{lem:propagation_of_chaos2}
There exists a constant $C$ which only depends on the bound of the transition rate $q$, such that for all $N>0$ and $\bbeta \in \mathbb{A}$ we have:
\begin{equation}\label{eq:lem_propagation_of_chaos2}
\hat{\mathbb{E}}^{(N)}\left[\left(\frac{\mathcal{E}^{(\bbeta, \hat\balpha^{-1,N})}_T}{\hat{\mathcal{E}}^{(N)}_T}\right)^2\right] \le C.
\end{equation}
\end{lemma}

\begin{proof}
Let us denote $W_t := \mathcal{E}^{(\bbeta, \hat\balpha^{-1,N})}_t / \hat{\mathcal{E}}^{(N)}_t$. By Ito's formula we have:
\begin{align*}
dW_t =&  \frac{d\mathcal{E}^{(\bbeta, \hat\balpha^{-1,N})}_t - \Delta \mathcal{E}^{(\bbeta, \hat\balpha^{-1,N})}_t}{\hat{\mathcal{E}}^{(N)}_{t-}} - \frac{\mathcal{E}^{(\bbeta, \hat\balpha^{-1,N})}_{t-}(d\hat{\mathcal{E}}^{(N)}_t - \Delta \hat{\mathcal{E}}^{(N)}_t)}{(\hat{\mathcal{E}}^{(N)}_{t-})^2} + \Delta W_t\\
=& W_{t-} \left(\frac{d\mathcal{E}^{(\bbeta, \hat\balpha^{-1,N})}_t - \Delta \mathcal{E}^{(\bbeta, \hat\balpha^{-1,N})}_t}{\mathcal{E}^{(\bbeta, \hat\balpha^{-1,N})}_{t-}} - \frac{d\hat{\mathcal{E}}^{(N)}_t - \Delta \hat{\mathcal{E}}^{(N)}_t}{\hat{\mathcal{E}}^{(N)}_{t-}} \right)+ \Delta W_t
\end{align*}
Recall that:
\begin{align*}
&\frac{d\hat{\mathcal{E}}^{(N)}_t}{\hat{\mathcal{E}}^{(N)}_{t-}} =\sum_{n=1}^N(X^n_{t^-})^* \cdot(Q(t, \hat\alpha^n_t) - Q^0)\cdot(\psi^n_t)^+ \cdot d\mathcal{M}^n_t,\\
&\frac{d\mathcal{E}^{(\bbeta, \hat\balpha^{-1,N})}_t}{\mathcal{E}^{(\bbeta, \hat\balpha^{-1,N})}_{t-}} = (X^1_{t^-})^* \cdot (Q(t, \beta_t) - Q^0)\cdot(\psi^1_t)^+\cdot d\mathcal{M}^1_t +\sum_{n=2}^N(X^n_{t^-})^* \cdot (Q(t, \hat\alpha^n_t) - Q^0)\cdot(\psi^n_t)^+ \cdot d\mathcal{M}^n_t,
\end{align*}
and $d\mathcal{M}^n_t = \Delta \mathcal{M}^n_t - Q^0 X^n_{t-}dt$. Noticing that for $n\neq 1$, the jumps of $\mathcal{M}^n_t$ do not result in the jumps of $W_t$, we obtain:
\begin{align*}
\Delta W_t =&\;\; \Delta \left(\frac{\mathcal{E}^{(\bbeta, \hat\balpha^{-1,N})}_t}{\mathcal{E}^{\hat{\boldsymbol{\alpha}}}_t}\right) = \frac{\mathcal{E}^{(\bbeta, \hat\balpha^{-1,N})}_{t-}}{\mathcal{E}^{\hat{\boldsymbol{\alpha}}}_
{t-}} \cdot \left(\frac{1 + (X^1_{t^-})^* \cdot (Q(t, \beta_t) - Q^0) \cdot (\psi^1_t)^+ \cdot \Delta\mathcal{M}^1_t}{1 + (X^1_{t^-})^* \cdot (Q(t, \hat\alpha^1_t) - Q^0)\cdot(\psi^1_t)^+\cdot \Delta\mathcal{M}^1_t} - 1\right)\\
=&\;\;W_{t-}\frac{(X^1_{t^-})^* \cdot (Q(t, \beta_t) - Q(t, \hat\alpha^1_t))\cdot(\psi^1_t)^+\cdot\Delta\mathcal{M}^1_t}{1 + (X^1_{t^-})^* \cdot(Q(t, \hat\alpha^1_t) - Q^0)\cdot(\psi^1_t)^+ \cdot\Delta\mathcal{M}^1_t}.
\end{align*}
Piggybacking on the computation in equation (\ref{eq:jump_of_l}), we see that when $X_{t-}^1 = e_i \neq e_j = X_{t}$, we have $\Delta\mathcal{M}^1_t = \Delta X^1_t = e_j - e_i$ and:
\[
\Delta W_t = W_{t-} \frac{q(t,i,j,\beta_t) - q(t,i,j,\hat\alpha^1_t)}{q(t,i,j,\hat\alpha^1_t)}.
\]
Let us define $\Xi_t^\beta$ to be an $m$ by $m$ matrix where the diagonal elements are $0$ and the element on the $i$-th row and the $j$-th column is $\frac{q(t,i,j,\beta_t) - q(t,i,j,\hat\alpha^1_t)}{q(t,i,j,\hat\alpha^1_t)}$. Then it is clear that $\Delta W_t = e_i^*\cdot \Xi_t^\beta \cdot (e_j - e_i)$. It follows that:
\[
\Delta W_t = W_{t-} \cdot (X_{t-}^1)^*\cdot \Xi_t^\beta \cdot \Delta \mathcal{M}^1_t.
\]
Injecting the above equation into the It\^o decomposition of $W_t$, we obtain:
\begin{align*}
dW_t = &\;\;W_{t-} [(Q(t,\hat\alpha_t^1) - Q(t,\beta_t))\cdot (\psi_t^1)^+ \cdot Q^0 \cdot X_{t-}^1 dt +(X_{t-}^1)^*\cdot \Xi_t^\beta \cdot \Delta \mathcal{M}^1_t]\\
=& \;\; W_{t-} [(Q(t,\hat\alpha_t^1) - Q(t,\beta_t))\cdot (\psi_t^1)^+ \cdot Q^0 \cdot X_{t-}^1 dt +(X_{t-}^1)^* \cdot \Xi_t^\beta \cdot (d\hat{\mathcal{M}}^1_t + Q^*(t,\hat\alpha_t^1) \cdot X_{t-}^1 dt)].
\end{align*}
In the second equality, we use the fact that under the measure $\hat{\mathbb{Q}}^{(N)}$, the state process $X_t^1$ has the canonical decomposition $dX_t^1 = d\hat{\mathcal{M}}^1_t + Q^*(t,\hat\alpha_t^1) \cdot X_t^1 dt$ where $\hat{\mathcal{M}}^1$ is a  $\hat{\mathbb{Q}}^{(N)}$-martingale. We also use the equality $\Delta \mathcal{M}^1_t = \Delta X_t^1 = dX_t^1$. In addition, by replacing $X_t^1$ with $e_i$ for $i=1,\dots,M$, it is plain to check the following equality:
\[
(Q(t,\hat\alpha_t^1) - Q(t,\beta_t))\cdot (\psi_t^1)^+ \cdot Q^0 \cdot  X_t^1 + (X_t^1)^* \cdot \Xi_t^\beta \cdot Q^*(t,\hat\alpha_t^1) \cdot X_t^1 = 0.
\]
This leads to the following representation of $W_t$:
\[
dW_t = W_{t-}(X_t^1)^*\cdot \Xi_t^\beta \cdot d\hat{\mathcal{M}}^1_t,
\]
which is a local martingale under the measure $\hat{\mathbb{Q}}^{(N)}$. At this stage, the rest of the proof is exactly the same as the proof of Proposition \ref{prop:stable_subset_by_phi}. In particular, we make use of Assumption \ref{hypo:boundedness}, that is the transition rate $q$ being bounded uniformly with regard to the controls.
\end{proof}

We are now ready to prove the form of the propagation of chaos result which we need.

\begin{proposition}
\label{prop:propagation_of_chaos}
There exists a sequence $(\epsilon_N)_{N\ge 0}$ such that $\epsilon_N \rightarrow 0$ as $N\rightarrow +\infty$ and such that for all $N \ge 0$, $n \le N$ and $\bbeta \in \mathbb{A}$:
\begin{equation}\label{eq:propagation_of_chaos}
\left|J^{n,N}(\bbeta, \hat\balpha^{-n,N}) - \mathbb{E}^{\mathbb{Q}^{(\bbeta, \hat\balpha^{-n,N})}}\left[\int_0^T f(t, X_t^n, \beta_t, p^*_t, \nu^*_t) dt + g(X_T^n, p^*_T)\right]\right|\le\epsilon_N.
\end{equation}
\end{proposition}

\begin{proof}
Due to symmetry, we only need to show the claim for $n = 1$. Let $N>0$ and $\bbeta \in \mathbb{A}$. Using successively Cauchy-Schwartz inequality, Assumption \ref{hypo:lipschitz_cost}, Lemma \ref{lem:propagation_of_chaos1} and Lemma \ref{lem:propagation_of_chaos2}, we have:
{\footnotesize
\begin{align*}
&\left|J^{n,N}(\bbeta, \hat\balpha^{-1,N}) - \mathbb{E}^{\mathbb{Q}^{(\bbeta, \hat\balpha^{-1,N})}}\left[\int_0^T f(t, X_t^1, \beta_t, p^*_t, \nu^*_t) dt + g(X_T^1, p^*_T)\right]\right|\\
&\hskip 15pt
\le\mathbb{E}^{\mathbb{Q}^{(\bbeta, \hat\balpha^{-1,N})}}\left[\int_0^T |f(t, X_t^1, \beta_t, p^*_t, \nu^*_t) -  f(t, X_t^1, \beta_t, p^N_t, \nu(\beta_t, \hat\alpha_t^{-1,N} ))| dt + |g(X_T^1, p^*_T) - g(X_T^1, p^N_T)|\right]\\
&\hskip 15pt
=\hat{\mathbb{E}}^{(N)}\left[\frac{\mathcal{E}^{(\bbeta, \hat\balpha^{-1,N})}_T}{\hat{\mathcal{E}}^{(N)}_T}\int_0^T |f(t, X_t^1, \beta_t, p^*_t, \nu^*_t) -  f(t, X_t^1, \beta_t, p^N_t, \nu(\beta_t, \hat\alpha_t^{-1,N} ))| dt + |g(X_T^1, p^*_T) - g(X_T^1, p^N_T)|\right]\\
&\hskip 15pt
\le \hat{\mathbb{E}}^{(N)}\left[\left(\frac{\mathcal{E}^{(\bbeta, \hat\balpha^{-1,N})}_T}{\hat{\mathcal{E}}^{(N)}_T}\right)^2\right]^{1/2}
\hat{\mathbb{E}}^{(N)}\left[\left(\int_0^T |f(t, X_t^1, \beta_t, p^*_t, \nu^*_t) -  f(t, X_t^1, \beta_t, p^N_t, \nu(\beta_t, \hat\alpha_t^{-1,N} ))| dt + |g(X_T^1, p^*_T) - g(X_T^1, p^N_T)|\right)^2\right]^{1/2}\\
&\hskip 15pt
\le C\hat{\mathbb{E}}^{(N)}\left[\left(\frac{\mathcal{E}^{(\bbeta, \hat\balpha^{-1,N})}_T}{\hat{\mathcal{E}}^{(N)}_T}\right)^2\right]^{1/2}
\left[\int_0^T (\hat{\mathbb{E}}^{(N)}[\|p^N_t - p^*_t\|^2] + \hat{\mathbb{E}}^{(N)}[\mathcal{W}_1^2(\nu(\beta_t, \hat\alpha_t^{-1,N}), \nu^*_t)])dt + \hat{\mathbb{E}}^{(N)}[\|p^N_T - p^*_T\|^2]\right]^{1/2}\\
&\hskip 15pt
\le C\sqrt{\delta_N},
\end{align*}
}
where $\delta_N$ is as appeared in Lemma \ref{lem:propagation_of_chaos1}, and $C$ is a constant only depending on $T$, the Lipschitz constant of $f$ and $g$ and the constant appearing in Lemma \ref{lem:propagation_of_chaos2}. this gives us the desired inequality.
\end{proof}

As a direct consequence of the above result on the propagation of chaos, we show that the Nash equilibrium of the mean field game consists of an approximate Nash equilibrium for the game with finite many players.

\begin{theorem}
\label{thm:approx_nash_eq}
There exists a sequence $\epsilon_N$ converging to $0$ such that for all $N>0$, $\bbeta \in \mathbb{A}$ and $n \le N$, we have:
\[
J^{n,N}(\bbeta, \hat\balpha^{-n, N}) \le J^{n,N}(\hat\balpha^{(N)}) + \epsilon_N.
\]
\end{theorem}
\begin{proof}
Recall that the strategy profile is $\hat{\boldsymbol{\alpha}}^{(N)} = (\hat\balpha^1,\dots, \hat\balpha^N)$ is defined as:
\[
\hat\balpha^n(w^1,w^2,\dots,w^N) := \balpha^*(w^n),
\]
where $\balpha^*$ is the strategy of the mean field game equilibrium, together with $\bp^*$ as the mean field of states and $\bnu^*$ as the mean field of control. For a strategy profile $(\balpha^1,\dots, \balpha^N)$ we use the notation:
\[
K^{n,N}(\balpha^1,\dots,\balpha^N) := \mathbb{E}^{\mathbb{Q}^{(\balpha^1,\dots,\balpha^N)}}\left[\int_0^T f(t, X_t^n, \alpha_t^n, p^*_t, \nu^*_t) dt + g(X_T^n, p^*_T)\right].
\]
Now taking $n=1$, we observe that $K^{1,N}(\hat{\boldsymbol{\alpha}}^{(N)}) = \mathbb{E}^{\mathbb{P}^N}[Y_0^{(\hat{\boldsymbol{\alpha}}^{(N)})}]$, where $Y_0^{(\hat{\boldsymbol{\alpha}}^N)}$ is the solution (at time $t=0$) of the following BSDE:
\begin{equation}
\label{eq:proof_thm_approx_nash_eq1}
Y_t = g(X_T^1, p^*_T) + \int_t^T H(s, X^1_{s-}, Z^1_s, \hat\alpha^1_s, p^*_s, \nu^*_s) ds - \int_t^T (Z^1_s)^*\cdot d\mathcal{M}^1_s.
\end{equation}
By the optimality of the equilibrium, we know that for all $t\in[0, T]$, $\hat\alpha^1_t$ minimizes the mapping $\alpha\rightarrow H(t, X^1_{t-}, Z^1_t, \alpha, p^*_t, \nu^*_t)$. Clearly, the solution of the above BSDE (\ref{eq:proof_thm_approx_nash_eq1}) is also the unique solution to the following BSDE:
\begin{equation}
\label{eq:proof_thm_approx_nash_eq2}
Y_t =g(X_T^1, p^*_T) + \int_t^T \left[H(s, X^1_{s-}, Z^1_s, \hat\alpha^1_s, p^*_s, \nu^*_s) + \sum_{n = 2}^N (X^{n}_t)^*\cdot(Q(s, \hat\alpha^{n}_s) - Q^0)\cdot Z^{n}_s \right]ds
 -\int_t^T \sum_{n=1}^N \int_t^T (Z_s^n)^* \cdot d\mathcal{M}^n_s,
\end{equation}
with $Z_t^n = 0$ for $n=2,\dots,N$. Indeed, the existence and uniqueness of the BSDE (\ref{eq:proof_thm_approx_nash_eq2}) can be checked easily by applying Theorem \ref{thm:bsde_ex_un_multi}. On the other hand, by following exactly the same argument as in the proof of Lemma \ref{lem:bsde_total_cost}, we can show that $K^{1,N}(\bbeta, \hat\balpha^{-1, N}) = \mathbb{E}^{\mathbb{P}^N}[Y_0^{(\bbeta, \hat\balpha^{-1, N})}]$, where $Y_0^{(\bbeta, \hat\balpha^{-1, N})}$ is the solution (at time $t=0$) of:
\begin{equation}
\label{eq:proof_thm_approx_nash_eq3}
Y_t =g(X_T^1, p^*_T) + \int_t^T \left[H(s, X^1_{s-}, Z^1_s, \beta_s, p^*_s, \nu^*_s) + \sum_{n = 2}^N (X^{n}_t)^*\cdot(Q(s, \hat\alpha^{n}_s) - Q^0)\cdot Z^{n}_s \right]ds
-\int_t^T \sum_{n=1}^N \int_t^T (Z_s^n)^* \cdot d\mathcal{M}^n_s.
\end{equation}
Notice that $H(s, X^1_{s-}, Z^1_s, \alpha, p^*_s, \nu^*_s) = f(s, X^1_{s-}, \alpha, p^*_s, \nu^*_s) + (X^{1}_{s-})^*\cdot(Q(s, \alpha) - Q^0)\cdot Z^{1}_s$, and $H(s, X^1_{s-}, Z^1_s, \hat\alpha^1_s, p^*_s, \nu^*_s) \ge H(s, X^1_{s-}, Z^1_s, \beta_s, p^*_s, \nu^*_s)$. Applying the comparison principle as stated in Theorem \ref{theo:linear_bsde_comp_multi} to the BSDEs (\ref{eq:proof_thm_approx_nash_eq2}) and (\ref{eq:proof_thm_approx_nash_eq3}), we conclude that $K^{1,N}(\bbeta, \hat\balpha^{-1, N}) \le K^{1,N}(\hat{\boldsymbol{\alpha}}^{(N)})$ for all $\bbeta \in \mathbb{A}$. Now thanks to symmetry, we have $K^{n,N}(\bbeta, \hat\balpha^{-n, N}) \le K^{1,N}(\hat{\boldsymbol{\alpha}}^{(N)})$ for all $\bbeta \in \mathbb{A}$ and $n=1,\dots,N$. The desired results immediately follows by applying Proposition \ref{prop:propagation_of_chaos}.
\end{proof}

\section*{Appendix: BSDEs Driven by Multiple Independent Continuous-Time Markov Chains}\label{sec:appendix}
Let us consider a probability space $(\Omega, \mathcal{F}, \mathbb{P})$ supporting $N$ independent continuous-time Markov chains $\bX^1, \dots, \bX^N$. For each $n=1,\dots,N$, we assume that $\bX^n$ takes only $m_n$ states, which are represented by the basis vectors of the space $\mathbb{R}^{m_n}$. We assume that under $\mathbb{P}$, the transition rate matrix of $\bX^n$ is $Q^{0,n}$, which is an $m_n\times m_n$ matrix where all the diagonal elements equal $-(m_n-1)$ and all the off-diagonal elements equal $1$. We denote by $\FF=(\mathcal{F}_t)_{t\in[0,T]}$ the natural filtration generated by $(\bX^1, \dots, \bX^N)$. It is clear that for each $n$, we can decompose the Markov chain $\bX^n$ as $X_t^n = X_0^n + \int_{0}^t Q^{0,n}\cdot X_{s-}^n ds + d\mathcal{M}^n_t$, where $\bcM^n$ is an $\FF$-martingale. In addition, due to the independence of the Markov chains, for all $n_1 \neq n_2$ and $t \le T$, $\mathbb{P}$-almost surely we have $\Delta X^{n_1}_t = 0$ or $\Delta X^{n_2}_t = 0$. In other words, any two Markov chains cannot jump simultaneously. 

Let us consider the process $\tilde \bX$ defined by $\tilde X_t := X_t^1 \otimes X_t^2 \otimes \dots \otimes X_t^N$ where $\otimes$ stands for the Kronecker product. Indeed, $\tilde\bX$ is a Markov chain encoding the joint states of the the $N$ independent Markov chains, and $\tilde\bX$ only takes values among the unit vectors of the space $\mathbb{R}^{m_1\times \dots \times m_N}$. We have the following result on the decomposition of $\tilde\bX$.

\begin{lemma}
\label{lem:N_ind_mc_decomp}
$\tilde\bX$ is a continuous-time Markov chain with transition rate matrix $\tilde{Q}^0$ given by:
\begin{equation}\label{eq:N_ind_mc_q_matrix}
\tilde{Q}^0 := \sum_{n=1}^{N} I_{m_1}\otimes \dots \otimes I_{m_{n-1}} \otimes Q^{0,n} \otimes I_{m_{n+1}} \otimes \dots \otimes I_{m_{N}}.
\end{equation}
In addition it has the canonical decomposition:
\begin{equation}\label{eq:N_ind_mc_decomp}
d\tilde X_t = \tilde{Q}^0\cdot \tilde{X}_{t-} dt + d\tilde{\mathcal{M}}_t,
\end{equation}
where $\tilde\bcM$ is a $\FF$-martingale which satisfies:
\begin{equation}
\label{eq:N_ind_mc_decomp}
d\tilde{\mathcal{M}}_t = \sum_{n=1}^{N} (X_{t-}^{1}\otimes \dots \otimes X_{t-}^{n-1} \otimes I_{m_n} \otimes X_{t-}^{n+1} \otimes \dots \otimes X_{t-}^{N}) \cdot d\mathcal{M}^n_t.
\end{equation}
\end{lemma}

\begin{proof}
In order to keep the notation to a reasonable level of complexity, we only argue the proof for $N=2$. Applying It\^o's formula to $X_t^1 \otimes X_t^2$ and noticing that $X_t^1$ and $X_t^2$ have no simultaneous jumps, we obtain:
\begin{align*}
d(X_t^1 \otimes X_t^2) =&\;\; dX_t^1 \otimes X_{t-}^2 + X_{t-}^1 \otimes dX_t^2 \\
=&\;\; (Q^{0,1}\cdot X_{t-}^1)\otimes X_{t-}^2 dt + d\mathcal{M}^1_t \otimes X_{t-}^2 + X_{t-}^1 \otimes (Q^{0,2}\cdot X_{t-}^2)dt + X_{t-}^1 \otimes d\mathcal{M}^2_t.
\end{align*}
Using the properties of the Kronecker product, we have:
\begin{align*}
(Q^{0,1}\cdot X_{t-}^1)\otimes X_{t-}^2 =&\;\; (Q^{0,1}\cdot X_{t-}^1)\otimes (I_{m_2}\cdot X_{t-}^2)= (Q^{0,1}\otimes I_{m_2})\cdot(X_{t-}^1\otimes X_{t-}^2)\\
d\mathcal{M}^1_t \otimes X_{t-}^2 =&\;\; (I_{m_1} \cdot d\mathcal{M}^1_t )\otimes (X_{t-}^2 \cdot 1) \\
=&\;\; (I_{m_1}\otimes X_{t-}^2) \cdot (d\mathcal{M}^1_t \otimes 1) = (I_{m_1}\otimes X_{t-}^2) \cdot d\mathcal{M}^1_t\\
X_{t-}^1 \otimes (Q^{0,2}\cdot X_{t-}^2) =&\;\; (I_{m_1}\cdot X_{t-}^1) \otimes (Q^{0,2}\cdot X_{t-}^2)=(I_{m_1} \otimes Q^{0,2}) \cdot(X_{t-}^1\otimes X_{t-}^2)\\
X_{t-}^1\otimes d\mathcal{M}^2_t  =&\;\; (X_{t-}^1 \cdot 1)\otimes (I_{m_2} \cdot d\mathcal{M}^2_t )\\
 =&\;\; (X_{t-}^1 \otimes I_{m_2}) \cdot (1 \otimes d\mathcal{M}^2_t) = (X_{t-}^1 \otimes I_{m_2}) \cdot d\mathcal{M}^2_t.
\end{align*}
Plugging the above equalities into the It\^o decomposition yields the desired result for $N=2$. The case $N>2$ can be treated by applying a simple argument of induction, which we will not detail here.
\end{proof}

As in the case of a single Markov chain, we define the stochastic matrix $\psi_t^n := diag(Q^{0,n}\cdot X_{t-}^n) - Q^{0,n} \cdot diag(X_{t-}^n) - diag(X_{t-}^n) \cdot Q^{0,n}$ for $n=1,\dots,N$ as well as $\tilde{\psi}_t := diag(\tilde{Q}^{0} \cdot \tilde{X}_{t-}) - \tilde{Q}^{0} \cdot diag(\tilde{X}_{t-}) - diag(\tilde{X}_{t-}) \cdot \tilde{Q}^{0}$. For $n = 1,\dots,N$, we define the stochastic seminorm $\|\cdot\|_{X_{t-}^n}$ by $\|Z\|^2_{X_{t-}^n} := Z^* \cdot \psi_t^n \cdot Z$ where $Z \in \mathbb{R}^{m_n}$. We then define the stochastic seminorm $\|\cdot\|_{\tilde{X}_{t-}}$ by $\|\tilde{Z}\|^2_{\tilde{X}_{t-}} := \tilde{Z}^*\cdot \tilde{\psi}_t\cdot \tilde{Z}$ where $\tilde{Z} \in \mathbb{R}^{m_1\times \dots \times m_N}$.
Our objective is to show existence and uniqueness of the following BSDE:
\begin{equation}
\label{eq:bsed_markov_multi}
Y_t = \xi + \int_t^T F(w,s,Y_s,Z_s^1, \dots, Z_s^n)ds - \sum_{n=1}^N \int_t^T (Z_s^n)^* \cdot d\mathcal{M}^n_s.
\end{equation}
Here $\xi$ is a $\mathcal{F}_T$-measurable $\mathbb{P}$-square integrable random variable and the driver $F: \Omega\times[0,T]\times \mathbb{R} \times \mathbb{R}^{m_1} \times \dots \times \mathbb{R}^{m_N} \rightarrow \mathbb{R}$ is a function such that the process $t \rightarrow F(w,t,y,z_1,\dots,z_N)$ is predictable for all $y,z_1,\dots, z_N \in  \mathbb{R} \times \mathbb{R}^{m_1} \times \dots \times \mathbb{R}^{m_N}$. The unknowns of the equation are a c\`adl\`ag process $\bY$ taking values in $\mathbb{R}$ and predictable processes $\bZ^1, \dots, \bZ^N$ taking valus in $\mathbb{R}^{m_1},\dots, \mathbb{R}^{m_N}$ respectively.
\begin{theorem}\label{thm:bsde_ex_un_multi}
Assume that there exists a constant $C > 0$ such that $dt\times\mathbb{P}$-a.s., we have:
\begin{equation}
\label{eq:bsde_ex_un_multi_condition}
|F(w,t,y,z_1,\dots,z_N) - F(w,t,\tilde y,\tilde z_1,\dots, \tilde z_N)| 
\le C \left(|y - \tilde y| + \sum_{n=1}^N\|z_n - \tilde z_n\|_{X^n_{t-}}\right).
\end{equation}
Then the BSDE (\ref{eq:bsed_markov_multi}) admits a solution $(\bY,\bZ^1,\dots,\bZ^N)$ satisfying:
\[
\mathbb{E}\left[\int_0^T |Y_t|^2 dt\right] < +\infty,\quad\quad\mathbb{E}\left[\sum_{n=1}^N\int_0^T \|Z_t^n\|_{X^n_{t-}}^2 dt\right] < +\infty
\]
Moreover, the solution is unique in the sense that if $(\bY^{(1)}, \bZ^{(1),1},\dots,\bZ^{(1),N})$ and $(\bY^{(2)}, \bZ^{(2),1},\dots,\bZ^{(2),N})$ are two solutions, then $\bY^{(1)}$ and $\bY^{(2)}$ are indistinguishable and we have $\mathbb{E}[\int_0^T \|\tilde Z^{(1)}_t- \tilde Z^{(2)}_t\|^2_{\tilde X_{t-}}dt] = 0$.
\end{theorem}

\begin{proof}
For simplicity of the presentation, we give the proof for $N=2$. It can be easily generalized to any $N>2$. Our first step is to show that the following equality holds for all $\tilde{Z} \in \mathbb{R}^{m_1 \times m_2}$:
\begin{equation}\label{eq:proof_bsde_ex_un_multi1}
\|\tilde{Z}\|^2_{\tilde{X}_{t-}} = \|(I_{m_1} \otimes (X_{t-}^2)^*)\cdot \tilde{Z}\|^2_{X^1_{t-}} + \|((X_{t-}^1)^* \otimes  I_{m_2})\cdot \tilde{Z}\|^2_{X^2_{t-}}
\end{equation}
By the definition of the semi-norm $\|\cdot\|_{X_{t-}^1}$, we have:
\begin{align*}
&\;\;\|(I_{m_1} \otimes (X_{t-}^2)^*)\cdot \boldsymbol{Z}\|^2_{X^1_{t-}} = \boldsymbol{Z}^* \cdot (I_{m_1} \otimes X_{t-}^2)\cdot \psi_t^1\cdot (I_{m_1} \otimes (X_{t-}^2)^*)\cdot \boldsymbol{Z}\\
=&\;\; \boldsymbol{Z}^* \cdot (I_{m_1} \otimes X_{t-}^2)\cdot (\psi_t^1\otimes 1)\cdot (I_{m_1} \otimes (X_{t-}^2)^*)\cdot \boldsymbol{Z} = \boldsymbol{Z}^* \cdot(\psi_t^1\otimes X_{t-}^2)\cdot (I_{m_1} \otimes (X_{t-}^2)^*)\cdot \boldsymbol{Z}\\
=&\;\; \boldsymbol{Z}^*\cdot[ \psi_t^1\otimes (X_{t-}^2\cdot (X_{t-}^2)^*)]\cdot \boldsymbol{Z} =\boldsymbol{Z}^*\cdot( \psi_t^1\otimes diag(X_{t-}^2)) \cdot \boldsymbol{Z}.
\end{align*}
Similarly we have $\|((X_{t-}^1)^* \otimes  I_{m_2})\cdot \boldsymbol{Z}\|^2_{X^2_{t-}} = \boldsymbol{Z}^*\cdot(diag(X_{t-}^1)\otimes \psi_t^2)\cdot \boldsymbol{Z}$. Now by the definition of $\tilde{\psi}_t$, we have:
\begin{align*}
\tilde{\psi}_t =&\;\; diag(\tilde{Q}^{0}\cdot \tilde{X}_{t-}) - \tilde{Q}^{0}\cdot diag(\tilde{X}_{t-}) - diag(\tilde{X}_{t-})\cdot \tilde{Q}^{0}\\
=&\;\;diag((I_{m_1}\otimes Q^{0,2} + Q^{0,1}\otimes I_{m_2}) \cdot (X^1_{t-}\otimes X^2_{t-}))- (I_{m_1}\otimes Q^{0,2} + Q^{0,1}\otimes I_{m_2})\cdot diag(X^1_{t-}\otimes X^2_{t-}) \\
&\hskip 50pt- diag(X^1_{t-}\otimes X^2_{t-}) \cdot (I_{m_1}\otimes Q^{0,2} + Q^{0,1}\otimes I_{m_2})\\
=&\;\;diag(X^1_{t-}\otimes(Q^{0,2}\cdot X^2_{t-})) + diag((Q^{0,1}\cdot X^1_{t-})\otimes X^2_{t-})\\
&\hskip 50pt-diag(X_{t-}^1)\otimes(Q^{0,2}\cdot diag(X^2_{t-})) - (Q^{0,1}\cdot diag(X^1_{t-}))\otimes diag(X_{t-}^2)\\
&\hskip 50pt-diag(X_{t-}^1)\otimes(diag(X^2_{t-})\cdot Q^{0,2}) - (diag(X^1_{t-})\cdot Q^{0,1})\otimes diag(X_{t-}^2)\\
=&\;\;\psi_t^1\otimes diag(X_{t-}^2) + diag(X_{t-}^1)\otimes \psi_t^2,
\end{align*}
where we have used the fact that for any two vectors $X^1, X^2$ we have $diag(X^1\otimes X^2) = diag(X^1)\otimes diag(X^2)$. This immediately leads to the equality (\ref{eq:proof_bsde_ex_un_multi1}). Now we consider the BSDE driven by the continuous-time Markov chain $\tilde{X}$ with terminal condition $\xi$ and the driver function $\tilde{F}$ defined by:
\[
\tilde{F}(w,t,Y,\tilde{Z}) := F(w,t,Y, (I_{m_1}\otimes (X^2_{t-})^*)\cdot \tilde{Z}, ((X^1_{t-})^*\otimes I_{m_2})\cdot \tilde{Z}).
\]
By equality (\ref{eq:proof_bsde_ex_un_multi1}) and the assumption on the regularity of $F$, we have:
\begin{align*}
&\;\;|\tilde{F}(w,t,Y_1,\tilde{Z}_1) - \tilde{F}(w,t,Y_2,\tilde{Z}_2)| \\
\le&\;\; C(|Y_1 - Y_2| + \|(I_{m_1}\otimes (X^2_{t-})^*)\cdot(\tilde{Z}_1-\tilde{Z}_2)\|_{X^1_{t-}} +  \|((X_{t-}^1)^* \otimes  I_{m_2})\cdot(\tilde{Z}_1-\tilde{Z}_2)\|_{X^2_{t-}})\\
\le &\;\; C(|Y_1 - Y_2| + \sqrt{2}\|\tilde{Z}_1-\tilde{Z}_2\|_{\tilde{X}_{t-}}).
\end{align*}
Applying Lemma \ref{lem:bsde_ex_un} we obtain the existence of the solution to the BSDE:
\[
Y_t = \xi + \int_t^T \tilde{F}(s,Y_s, \tilde{Z}_s) ds + \int_t^T \tilde{Z}_s^* \cdot d\tilde{\mathcal{M}}_s
\]
Now we set $Z^1_t := (I_{m_1}\otimes (X^2_{s-})^*)\cdot\tilde{Z}_s$ and $Z^2_t := ((X^1_{s-})^*\otimes I_{m_2})\cdot\tilde{Z}_s$. From the definition of the driver $\tilde{F}$ and $\tilde{\mathcal{M}}$ in equation (\ref{eq:N_ind_mc_decomp}), we see that:
\begin{align*}
Y_t =&\;\; \xi + \int_t^T F(w,s,Y_s, (I_{m_1}\otimes (X^2_{s-})^*)\cdot\tilde{Z}_s, ((X^1_{s-})^*\otimes I_{m_2})\cdot \tilde{Z}_s) ds \\
&\hskip 40mm + \int_t^T \tilde{Z}_s^* \cdot [(I_{m_1} \otimes X_{t-}^2) \cdot dM_t^1 + (X_{t-}^1 \otimes I_{m_2})\cdot dM_t^2]\\
= &\;\;\xi + \int_t^T F(w,s,Y_s, Z^1_s, Z^2_s) ds + \int_t^T (Z^1_s)^* \cdot dM_t^1 + \int_t^T (Z^2_s)^*\cdot dM_t^2
\end{align*}
This shows that $(Y,Z^1, Z^2)$ is a solution to BSDE (\ref{eq:bsed_markov_multi}).
\end{proof}
We also state a comparison principle for linear BSDEs driven by multiple independent Markov chains.
\begin{theorem}\label{theo:linear_bsde_comp_multi}
For each $n\in\{1,\dots,N\}$, let $\gamma^n$ be a bounded predictable process in $\mathbb{R}^{m_n}$ such that $\sum_{i=1}^{m_n}[\gamma^n_t]_i = 0$ for all $t \in [0,T]$, and $\beta$ a bounded predictable process in $\mathbb{R}$. Let $\phi$ be a non-negative predictable process in $\mathbb{R}$ such that $\mathbb{E}[\int_0^T \|\phi_t\|^2 dt] < +\infty$ and $\xi$ a non-negative square-integrable $\mathcal{F}_T$ measurable random variable in $\mathbb{R}$. Let $(Y,Z)$ be the solution of the linear BSDE:
\begin{equation}\label{eq:linear_bsde_multi}
Y_t = \xi + \int _t^T (\phi_s +\beta_s Y_s + \sum_{n=1}^N(\gamma^n_s)^*\cdot Z^n_s) ds - \sum_{n=1}^N \int_t^T (Z_s^n)^*\cdot d\mathcal{M}^n_s.
\end{equation}
Assume that for all $n=1,\dots,N$, $t\in(0,T]$ and $j$ such that $(e_j^n)^* \cdot Q^{0,n} \cdot X_{t-}^n > 0$, we have $1 + (\gamma_t^n)^*\cdot(\psi_t^n)^+\cdot(e^n_j - X^n_{t-}) \ge 0$ where $(\psi_t^n)^+$ is the Moore-Penrose inverse of the matrix $\psi_t^n$. Then $Y$ is nonnegative.
\end{theorem}
\begin{proof}
As before we treat the case for $N=2$, for which the argument can be trivially generalized to any $N>2$. Since $\gamma^n$ and $\beta$ are bounded processes and $\sum_{i=1}^{m_n}[\gamma^n_t]_i = 0$ for all $t \le T$ and $n\le N$, we easily verify that the Lipschitz condition (\ref{eq:bsde_ex_un_multi_condition}) stated in Theorem \ref{thm:bsde_ex_un_multi} is satisfied and therefore the BSDE (\ref{eq:linear_bsde_multi}) admits a unique solution. Now consider the following BSDE driven by $\boldsymbol{\mathcal{M}}$:
\begin{equation}\label{eq:linear_bsde_multi}
Y_t = \xi + \int _t^T (\phi_s +\beta_s Y_s + \boldsymbol{\gamma}_s^* \cdot \boldsymbol{Z}_s) ds - \sum_{n=1}^2 \int_t^T \boldsymbol{Z}_s^* \cdot d\boldsymbol{\mathcal{M}}_s,
\end{equation}
where $\boldsymbol{\gamma}_t := (\gamma_t^1 \otimes X_{t-}^2) + (X_{t-}^1 \otimes \gamma_t^2)$. It is easy to verify the BSDE (\ref{eq:linear_bsde_multi}) admits a unique solution $(Y,Z)$ and following the same argument as in the proof of Theorem \ref{thm:bsde_ex_un_multi}, we verify that $(Y_t,Z^1_t, Z^2_t) := (Y_t, (I_{m_1}\otimes (X^2_{s-})^*)\cdot \boldsymbol{Z}_s,((X^1_{s-})^*\otimes I_{m_2})\cdot\boldsymbol{Z}_s)$ solves the BSDE (\ref{eq:linear_bsde}), which is also its unique solution. Therefore we only need to show that the solution $Y$ to BSDE (\ref{eq:linear_bsde}) is nonnegative.
To this ends, we need to apply the comparison principal for the case of a single Markov chain, as is stated in Lemma \ref{lem:linear_bsde_comp}. Note that $X^1$ and $X^2$ do not jump simultaneously and $\boldsymbol{X}_t = X^1_t \otimes X^2_t$. For the jump of $\boldsymbol{X}$ resulting from the jump of $X^1$, we need to show that for $k=1,\dots,m_1$:
\begin{equation}\label{eq:proof_linear_bsde_comp_multi1}
1 + \boldsymbol{\gamma}_t^*\cdot\boldsymbol{\psi}_t^+\cdot(e_k^1 \otimes X_{t-}^2 - X^1_{t-} \otimes X^2_{t-}) \ge 0.
\end{equation}
Let us assume that $X^1_{t-} = e_i^1$, $X^2_{t-} = e_j^2$. If $k = i$, the above equality is trivial. In the following, we consider the case $k\neq i$. Then by the assumption of the theorem, we have:
\begin{equation}\label{eq:proof_linear_bsde_comp_multi0}
1 + (\gamma_t^1)^*\cdot(\psi_t^1)^+\cdot(e_k^1 - e_i^1) \ge 0.
\end{equation}
It can be easily verified that:
\begin{align*}
&(diag(e_i^1)\otimes \psi_t^2 + \psi_t^1 \otimes diag(e_j^2))\cdot\left[(m_1+m_2-2) e_k^1 \otimes e_j^2 - \sum_{k_0 \neq k} e_{k_0}^1 \otimes e_j^2 - \sum_{j_0\neq j}e_i^1\otimes e_{j_0}^2\right]\\
 &=e_k^1 \otimes e_j^2 - e_i^1 \otimes e_j^2,
\end{align*}
so that we have:
\begin{align*}
&\boldsymbol{\psi}_t^+\cdot(e_k^1 \otimes X_{t-}^2 - X^1_{t-} \otimes X^2_{t-})\\
 =& \frac{1}{m_1+m_2-1}\left[(m_1+m_2-2) e_k^1 \otimes e_j^2 - \sum_{k_0 \neq k} e_{k_0}^1 \otimes e_j^2 - \sum_{j_0\neq j}e_i^1\otimes e_{j_0}^2\right].
\end{align*}
It follows that:
\begin{align*}
&\;\;\boldsymbol{\gamma}_t^*\cdot\boldsymbol{\psi}_t^+\cdot(e_k^1 \otimes X_{t-}^2 - X^1_{t-} \otimes X^2_{t-})\\
=&\;\;\frac{1}{m_1+m_2-1}(\gamma_t^1 \otimes e^2_j + e^1_i \otimes \gamma_t^2)^*\cdot\left[(m_1+m_2-2) e_k^1 \otimes e_j^2 - \sum_{k_0 \neq k} e_{k_0}^1 \otimes e_j^2 - \sum_{j_0\neq j}e_i^1\otimes e_{j_0}^2\right]\\
=&\;\;\frac{1}{m_1+m_2-1}\left[(m_1+m_2-2)(e_k^1)^*\cdot\gamma_t^1 - \sum_{k_0 \neq k} (e_{k_0}^1)^*\cdot \gamma_t^1 - (e_j^2)^*\cdot\gamma_t^2- \sum_{j_0\neq j}(e_{j_0}^2)^*\cdot\gamma_t^2\right]\\
=&\;\;\frac{1}{m_1+m_2-1}\left[(m_1+m_2-1)(e_k^1)^*\cdot\gamma_t^1 - \sum_{k_0} (e_{k_0}^1)^*\cdot \gamma_t^1 - (e_j^2)^*\cdot\gamma_t^2- \sum_{j_0}(e_{j_0}^2)^*\cdot\gamma_t^2\right] \\
=&\;\; (e_k^1)^*\cdot\gamma_t^1,
\end{align*}
where in the last equality we used the assumption that $\sum_{i=1}^{m_n}[\gamma^n_t]_i = 0$ for $n=1,2$. Now noticing that $(e_k^1)^*\cdot\gamma_t^1 = (\gamma_t^1)^*\cdot(\psi_t^1)^+\cdot(e_k^1 - e_i^1)$, we obtain:
\[
1 + \boldsymbol{\gamma}_t^*\cdot\boldsymbol{\psi}_t^+\cdot(e_k^1 \otimes X_{t-}^2 - X^1_{t-} \otimes X^2_{t-}) = 1 + (\gamma_t^1)^*\cdot(\psi_t^1)^+\cdot(e_k^1 - e_i^1).
\]
Combining this with the inequality (\ref{eq:proof_linear_bsde_comp_multi0}), we obtain the inequality (\ref{eq:proof_linear_bsde_comp_multi1}). Proceeding in a similar way we can also show that for $k=1,\dots,m_2$:
\[
1 + \boldsymbol{\gamma}_t^*\cdot\boldsymbol{\psi}_t^+\cdot(X_{t-}^1 \otimes e_k^2 - X^1_{t-} \otimes X^2_{t-}) \ge 0.
\]
Applying Lemma \ref{lem:linear_bsde_comp} to the BSDE (\ref{eq:linear_bsde_multi}), we obtain the desired result.
\end{proof}

\bibliographystyle{siam}

\end{document}